\documentclass[aop]{imsart}
\RequirePackage{amsthm,amsmath,amsfonts,amssymb}
\RequirePackage[numbers]{natbib}
\RequirePackage[colorlinks,citecolor=blue,urlcolor=blue]{hyperref}
\RequirePackage{graphicx, epstopdf}
\startlocaldefs
\usepackage{tikz}
\usepackage{color}
\usepackage[percent]{overpic}
\usepackage{xcolor}
\usepackage{float}
\usepackage{mathtools}
\usepackage{enumerate}
\usepackage{array}
\usetikzlibrary{calc}
 \usepackage{pdfpages}
 \usepackage{empheq}
 \usepackage{commath}
 \usetikzlibrary{decorations.pathreplacing}

\startlocaldefs
\newcommand{\bracedincludegraphics}[2][]{%
  \sbox0{$\vcenter{\hbox{\includegraphics[#1]{#2}}}$}%
  \left\lbrace
    \vphantom{\copy0}
  \right.\kern-\nulldelimiterspace
  \underbrace{\vrule width0pt depth \dimexpr\dp0 + .3ex\relax\box0}}

\newtheorem{theorem}{Theorem}
\newtheorem{lemma}[theorem]{Lemma}
\newtheorem{proposition}[theorem]{Proposition}
\newtheorem{corollary}[theorem]{Corollary}

\theoremstyle{definition}

\renewcommand{\det}{\mbox{det\,}}
\newcommand{\Det}{\mbox{Det\,}}
\newcommand{\Pf}{\mbox{Pf\,}}
\newcommand{\pf}{\mbox{pf\,}}
\newcommand{\Tr}{\mbox{Tr\,}}
\newcommand{\sgn}{\mbox{sgn}}
\DeclareMathOperator{\sech}{sech}
\newcommand{\E}{\mathbb{E}}
\newcommand{\Pp}{\mathbb{P}}
\newcommand{\I}{\mathbb{I}}
\newcommand{\Kac}{\mbox{Kac}}
\newcommand{\R}{\mathbb{R}}
\newcommand{\Z}{\mathbb{Z}}
\newcommand{\N}{\mathbb{N}}
\newcommand{\fr}{\frac{1}{2}}
\newcommand{\bea}{\begin{eqnarray}}
\newcommand{\eea}{\end{eqnarray}}
\newcommand{\beast}{\begin{eqnarray*}}
\newcommand{\eeast}{\end{eqnarray*}}
\newcommand{\Brownian}[7]{
\draw[#4] (#6,#7)
\foreach \x in {1,...,#1}
{   -- ++(#2,rand*#3)
}
node[right] {#5};
}
\endlocaldefs
\begin{document}
\begin{frontmatter}
\title{Asymptotic expansions for a class of Fredholm Pfaffians and interacting
particle systems}
\runtitle{Fredholm Pfaffians and interacting particle systems}
\begin{aug}
\author[A]{\fnms{Will} \snm{FitzGerald}\ead[label=e1]{william.fitzgerald@manchester.ac.uk}},
\author[B]{\fnms{Roger} \snm{Tribe}\ead[label=e2,mark]{r.p.tribe@warwick.ac.uk}}
\and
\author[B]{\fnms{Oleg} \snm{Zaboronski}\ead[label=e3,mark]
{o.v.zaboronski@warwick.ac.uk}}
\address[A]{Department of Mathematics, University of Manchester,
Manchester, M$13~9$PL,
United Kingdom,
\printead{e1}}
\address[B]{Mathematics Institute, University of Warwick,
Coventry CV$4~7$AL, United Kingdom,
\printead{e2,e3}}
\end{aug}
\begin{abstract}
Motivated by the phenomenon of duality for interacting particle systems 
we introduce two classes of Pfaffian kernels
describing a number of Pfaffian point processes in the `bulk' and at the `edge'.
Using the probabilistic method due to Mark Kac, we prove two  Szeg\H{o}-type
asymptotic expansion theorems for the corresponding Fredholm Pfaffians. 
The idea of the proof is to introduce an effective random
walk with transition density determined by the Pfaffian kernel, express
the logarithm of the Fredholm Pfaffian through expectations with 
respect to the random walk, and analyse the expectations using general results
on random walks. We demonstrate the utility of the theorems
by calculating asymptotics for the empty interval and non-crossing probabilities for a number
of examples of Pfaffian point processes: 
coalescing/annihilating Brownian motions, massive coalescing
Brownian motions, real zeros of Gaussian power series and Kac polynomials, and real eigenvalues
for the real Ginibre ensemble.
\end{abstract}

\end{frontmatter}
\tableofcontents
\section{Introduction} \label{s1}
The aim of the paper is to prove
rigorously, extend and generalise a number of asymptotic formulae
for the empty interval and crossing probabilities for systems of annihilating-coalescing
Brownian motions on $\R$ obtained in the 1990's by Derrida and Zeitak
\cite{Derrida_Zeitak}
and Derrida, Hakim and Pasquier \cite{Derrida_Hakim_Pasquier}. 
A very special feature underlying the computation of   
each of these probabilities is that the corresponding  
random process is actually a Pfaffian point process for arbitrary
deterministic initial conditions. This fact
is by now well known for coalescing or annihilating Brownian motions
on $\R$, \cite{TZ}. However, it also holds true for mixed coalescing/annihilating systems,
for annihilating systems with pairwise immigration, for coelascing systems with branching, 
and for analogous systems on $\Z$,
\cite{GPTZ}, \cite{GTZ}. Therefore, asymptotics of the probabilities 
studied in \cite{Derrida_Zeitak}, \cite{Derrida_Hakim_Pasquier} can be reduced to the asymptotic analysis
of Fredholm Pfaffians
(introduced in \cite{rains2000}, or see the subsection \ref{sFDFP} below for a brief review). 
\paragraph*{Background}
For determinantal point processes the empty interval probabilities (also called gap probabilities) are given by Fredholm determinants
of the corresponding kernels, see e.g. \cite{anderson2010} or \cite{mehta2004random}
for a review. If the kernel is translationally invariant (that is
depends on the difference of the arguments only) 
then the Akhiezer-Kac formula \cite{Kac}, \cite{Akhiezer} gives the asymptotics for a Fredholm determinant
as the length of the interval grows. The Akhiezer-Kac formula is an extension of Szeg\H{o}'s theorem \cite{szego} for 
determinants of
Toeplitz matrices to the continuous case. Unfortunately, 
there seems to be no analogous result for Fredholm Pfaffians, even though
via operator manipulations it is often possible to reduce a Fredholm Pfaffian
to a Fredholm determinant of an operator with a scalar kernel, see e.g. 
\cite{Tracy_Widom} for such a reduction in the case of the Gaussian orthogonal and symplectic ensembles of random matrices or \cite{Rider_Sinclair} for the real Ginibre ensemble. This is probably due to the lack of the general theory of Pfaffian point processes, which would require proving an analogue of Soshnikov's theorem 
\cite{soshnikov2000determinantal} for determinantal point processes by classifying all
skew-symmetric $2\times 2$ matrix-valued kernels that give rise to a random
point process. 
Therefore our task is to study the asymptotics of Fredholm Pfaffians for a sufficiently large class of kernels which contains all of the empty-interval probabilities for the Pfaffian point processes in which we are interested.

\paragraph*{Classes of Pfaffian kernels}
Motivated by the duality ideas often employed in integrable probability to 
find explicit solutions for exactly solvable interacting particle systems, we 
find two classes of such processes, for the translationally invariant (`bulk')
and non-invariant (`edge') cases. Asymptotics for these
two cases (Theorems \ref{pfaffian_czego} and \ref{pfaffian_szego_edge}) are the first main results of the paper. 
As a byproduct of our analysis, we
obtain a generalisation of the Akhiezer-Kac formula  for Fredholm 
determinants for a class of `asymptotically
translationally-invariant' kernels, see Theorem \ref{nonTIszego}.

It turns out that a number of well known Pfaffian point processes fall
into one of our classes including all the reaction diffusion systems listed earlier,
the exit measure for annihilating-coalescing Brownian motions 
(see Section \ref{s2.3} below), the law of
real eigenvalues for the real Ginibre ensemble both in the bulk and near the edge
discovered in \cite{Forrester_Nagao} and \cite{Borodin_Sinclair}, as well
as the law of eigenvalues in the bulk and edge scaling limits
for the classical GOE random matrix ensemble - see 
\cite{mehta2004random},
\cite{anderson2010} for a review. Unfortunately, our theorems have nothing
to say about the last two cases due to the insufficiently fast decay of the corresponding
kernels at infinity.
%
\paragraph*{Approach to the proofs}
%
A notable feature of the class of kernels we consider is the possibility to follow the manipulations
carried out in \cite{Tracy_Widom} and \cite{Rider_Sinclair} in order to reduce
the corresponding Fredholm Pfaffian to the square root of the product of a Fredholm
determinant and a finite dimensional determinant. We are therefore left with a pure analysis
problem of computing the asymptotic behaviour of each of these determinants in the 
limit of large gap sizes. In principle, this problem can be solved using a combination of a direct
computation and an application of Szeg\H{o}-type theorems, possibly modified owing to the presence
of  a Fisher-Hartwig singularity, see \cite{Deift_Its_krasovsky} for a review of the discrete (Toeplitz)
case. Such an approach runs into difficulties due to the dearth of sufficiently general asymptotic
results for Fredholm determinants in the non-translationally invariant case. Fortunately, the
analysis problem at hand can be solved using a probabilistic approach. This was used
by Mark Kac to formulate the first asymptotic results for Fredholm determinants, thus generalising 
 Szeg\H{o}'s original theorems for Toeplitz matrices, see \cite{Kac}. It turns out that
both the finite dimensional determinant and the Fredholm determinant under consideration
can be interpreted as  expectations with respect to the law of (time-inhomogeneous) 
random walk, which can be subsequently analysed using general results concerning
random walks, such as Sparre Andersen's formula, Spitzer's formulae, cyclic symmetry,
renewal theory, the invariance principle and optional stopping. The essence of Kac's
arguments is very easy to illustrate in the translationally invariant case. Let $K_T$
be an integral operator on $L^2[0,T]$: $K_Tf(x)=\int_0^T  \rho(x-y)f(y)dy$. Assume that
$\rho$ is non-negative and such that
$\int_\R \rho(x)=1$, that is a probability density function. Given some regularity conditions
on $\rho$, for any $0\leq \lambda \leq 1$ the Fredholm determinant of the operator $I-\lambda K_T$
can be computed using the trace-log (Plemelj-Smithies) formula:
\beast
 \log \Det (I-\lambda K_T)=-\sum_{n=1}^\infty\lambda^n\frac{\Tr K_T^n}{n}.
\eeast
Using the convention $x_{n+1}=x_1$, one can write for any $n \in \N$:
\beast
\Tr K_T^n=\int_{[0,T]^n}  \prod_{k=1}\rho(x_{k+1}-x_k)dx_1 dx_2\ldots dx_n
=\int_{[0,T]}  \mathbb{P}_x\left[\tau_T > n, X_n \in dx \right],
\eeast
where $\mathbb{P}_x$ is the law of the $\R$-valued random walk $(X_n: n>0)$ started at $x$ with the distribution of increments given by
$\rho$, and $\tau_T$ is the first exit time from the interval $[0,T]$.  
Therefore the problem of computing the Fredholm determinant is reduced to the analysis of random walks constrained not to
exit the interval $[0,T]$ killed independently at every step with probability $1-\lambda$.
It was Kac's insight that such an analysis can be carried out for any random walk using
probabilistic arguments. Varying the parameter $\lambda$ is very natural for probabilists: for a point process 
with (determinantal or Pfaffian) kernel $K$, the thinned point process, where each particle is removed
independently with probability $p$, has new kernel $p K$.


The asymptotic results detailed in Theorems \ref{pfaffian_czego}, \ref{nonTIszego},
\ref{pfaffian_szego_edge} are an example of pure analysis statements proved
using probabilistic methods. As happened with Kac's Theorem for Fredholm determinants, a non-probabilistic
proof of asymptotics for Fredholm Pfaffians will most certainly appear. 
Kac's assumption of positivity $\rho \geq 0$ was used to allow 
probabilistic arguments, but the final result was subsequently found to be true without this assumption,
and we believe the analogous positivity assumption in our results is unnecessary. 
Moreover, such a proof is certainly desirable: at the moment the logarithms of both the 
finite dimensional and Fredholm determinants contain terms proportional to the logarithm
of the gap size, see Theorem \ref{nonTIszego}. The logarithmic terms cancel upon taking the final product. 
In other words, the Fredholm Pfaffian does not have the singularity
present in the corresponding Fredholm determinant.
It is 
likely that there is a 
streamlined proof of Theorems \ref{pfaffian_czego} and \ref{pfaffian_szego_edge} which avoids re-expressing the Fredholm Pfaffians
in terms of determinants, and where the logarithmic terms do not appear at the 
intermediate steps. 
\paragraph*{Application examples}
The motivation for the study of Fredholm Pfaffians came from applications
to probability theory. In
\cite{Derrida_Zeitak}, Derrida and Zeitak study the distribution of domain sizes
in the $q$-state Potts model on $\Z$ for the `infinite temperature' initial conditions
(the initial colours are chosen uniformly independently at each site). They show that,
as $L/\sqrt{t} \to \infty$,
\bea\label{intro_dz}
\Pp[\mbox{The interval } [0,L] \mbox{ contains one colour at time } t]
=e^{-A(q)\frac{L}{\sqrt{t}}+B(q)+o(1)},
\eea
where $A,B$ are some explicit functions of $q$. Derrida and Zeitak's formula
is valid for all $q\in [1,\infty)$. For Potts models, $q \in 1+\N$. However, as pointed out by the authors, there is
a physical interpretation of (\ref{intro_dz}) for $q \notin 1+\N$: it gives
the probability that the interval $[0,x]$ is contained in a domain of positive spins in zero temperature Glauber model on $\Z$ started with independent homogeneous
distribution of spins with the average spin (`magnetisation') equal to $m \in [-1,1]$.
Then the parameter $q=\frac{2}{1+m}$. The computation method employed in \cite{Derrida_Zeitak}
can possibly be made
rigorous for $q<1/2$, but it relies on the assumption of analyticity of the 
functions $A,B$ in order to obtain the answer for $q\geq 2$. These assumptions
seem hard to justify to us. Alternatively, one prove (\ref{intro_dz}) as follows: as
is well known, the boundary of monochrome domains in 
the (diffusive limit of) Potts model behaves as a system of
instantaneously
coalescing/annihilating Brownian motions on $\R$ (denoted CABM) with the annihilating 
probability $\frac{1}{q-1}$ and coalescing 
probability $\frac{q-2}{q-1}$ at each collision, see
\cite{Derrida_Zeitak} and \cite{GPTZ} for details.  Therefore, the distribution of domain
sizes in the Potts model corresponds to the empty interval 
probabilities for coalescing/annihilating
Brownian motions which is known to be a Pfaffian point process \cite{GPTZ}.
Moreover, if $\mathbf{K}$ is the kernel for the purely 
coalescing case ($q=\infty$),
then $\frac{q-1}{q}\mathbf{K}$ is the kernel for  CABM.
In other
words, CABM can be obtained from coalescing Brownian 
motions by thinning, that is deleting particles independently with probability $p=\frac{q-1}{q}$ (see 
Section \ref{s3} for precise details, including the role of the initial conditions).
Therefore the gap probability is given by the Fredholm Pfaffian
$ \Pf_{[0,L]}[\mathbf{J}-p\mathbf{K}]$, and 
applying Theorem \ref{pfaffian_czego} 
one reaches (\ref{intro_dz}). 
Notice that the interpretation 
in terms of gap probabilities
for CABM uses kernel $p \mathbf{K}$ for  $p \geq \fr$ (`weak thinning'). 
For $p<\fr$ (`strong thinning') one can interpret the answer either 
in terms of gaps between the boundaries of positive spins in 
the Glauber model as above, or there is an interpretation
in terms of gap sizes for a $massive$ coalescence model - see Lemma \ref{MCBM} below. 

Derrida and Zeitak's result (\ref{intro_dz}) can be extended as follows. 
CABM
started from half-space initial conditions is still a Pfaffian point process 
with a translationally non-invariant  kernel covered by conditions of Theorem \ref{pfaffian_szego_edge}.
The asymptotics of the corresponding Fredholm Pfaffian can therefore be computed
yielding the right (left) tail of the fixed time distribution of the leftmost 
(rightmost) particle. Notice that the value $p=\fr$ of the thinning parameter corresponds to purely
annihilating Brownian motions (ABM). It is  known \cite{TZ}, \cite{GPTZ} that the 
fixed time law of ABM
with full-space (half-space) initial conditions is a Pfaffian point process
coinciding with the bulk (edge) scaling limit of the law of real eigenvalues
for the real Ginibre random matrix model discovered in \cite{Forrester_Nagao},
\cite{Borodin_Sinclair}, \cite{borodin2016erratum}. This remark
allowed a computation of the tails of the distribution of the maximal 
real eigenvalue for the edge scaling limit of the real Ginibre ensemble in
\cite{PTZ} and \cite{fitzgerald2020sharp}. It is the generalisation
of the arguments in the last two cited papers that led to the asymptotic Theorems
\ref{pfaffian_czego}, \ref{nonTIszego}, \ref{pfaffian_szego_edge}.

In another influential investigation, Derrida, Hakim and Pasquier \cite{Derrida_Hakim_Pasquier} compute
the fraction
of sites which haven't  changed colour up to time $t$ for a q-state
Potts model on $\Z$. 
Due to the translation invariance of the `infinite temperature' initial conditions, this is 
equivalent to the probability that the colour of the state at the origin hasn't changed up to time $t$
(also known as  a `persistence' probability, see \cite{Persistence_review} for a review
of the persistence phenomenon in the context of non-equilibrium statistical mechanics).
They show, up to logarithmic precision, that 
\bea\label{intro_dhp}
\Pp\left[\mbox{The colour at $0$ does not change in }[0,t]\right]\sim t^{-\gamma(q)}
\quad \mbox{ as } t\rightarrow \infty
\eea
where
\beast
\gamma(q)=-\frac{1}{8}+\frac{2}{\pi^2} 
\left[\cos^{-1}\left(\frac{2-q}{\sqrt{2}q} \right)\right]^2.
\eeast
The authors derive this result by noticing that this event is equivalent for the domain boundaries, 
which form a system of coalescing/annihilating random walks, to the event that no boundary
crosses zero during the time interval $[0,t]$.
Motivated by (\ref{intro_dhp}), we consider CABM on $[0,\infty)$ started from the `maximal'
entrance law supported on $(a,\infty)$ for some $a>0$. Particles hitting the boundary
at $x=0$ are removed from the system and we record  the corresponding exit times.
We show that the resulting exit measure is a Pfaffian point process with
non-translationally invariant kernel belonging to our class. Then Theorem
\ref{pfaffian_szego_edge} applies and one can deduce the asymptotics for the probability that the interval
$[0,T]$ contains no exiting particles,
in the limit of large 
time $T$. This immediately gives the non-crossing probability for the position of the leftmost
particle $(L_t: t\geq 0)$ for the system of CABM on the whole real axis started from
every point of $(a,\infty)$: 
\beast
\Pp \left[\inf_{t\in [0,T]} L_t>0 \right]= K(q) \left(\frac{T}{a^2}\right)^{-\gamma(q)/2}\left(1+o(1)\right),
\eeast
for a known constant $K(q)$ (independent of $a$) and the exponent $\gamma (q)$ specified above.
If one starts CABM with particles at every point of $(-\infty,-a)\cap (a,\infty)$,
the above formula can be used to derive the following continuous counterpart of
(\ref{intro_dhp}):
\begin{equation} \label{derrida2}
\Pp\left[\mbox{No particle crosses 0 in } [0,T]\right] 
=K(q)^2 \left(\frac{T}{a^2}\right)^{-\gamma(q)}
\left(1+o(1)\right).
\end{equation}

As discussed above, the fixed time law for annihilating Brownian motions is closely
related to the real Ginibre ensemble. Similarly, 
the `bulk' limit of the Pfaffian point process describing the exit measure for ABM is identical,
up to a deterministic transformation,
to the translationally invariant kernel for the Pfaffian point process giving
the law of real roots of the Gaussian power series,  and also to the 
 large $N$ limit for the real eigenvalues of truncated orthogonal matrices
 (where one row/column has been removed to obtain a minor) - see
\cite{Forrester}, \cite{Matsumoto_Shirai}. Theorem \ref{pfaffian_czego}
allows one to compute the corresponding empty interval  probability, up to the constant term, 
as the endpoints of the interval approach $\pm 1$. For example,
\beast
\Pp \left[\mbox{The interval } [-1+2 \epsilon,1-2 \epsilon] \mbox{ contains no roots} \right]
=\epsilon^{\frac{3}{8}} \, e^{\kappa_2} \left(1+o(1)\right)\quad \mbox{as }\epsilon \downarrow 0
\eeast
where 
\[
\kappa_2 = \frac14 \log \Big(\frac{\pi^2}{2}\Big) - \frac{\gamma}{2} 
- \frac14 
 \int_0^{\infty} \log (x) (\tanh(x) + \tanh(x/2)) (\sech^2(x) + \mbox{$\frac12$} \sech^2(x/2)) dx.
 \]
Here $\gamma$ is the Euler-Mascheroni constant.
Notice that the exponent $3/8$ coincides with the value $\gamma(2)$ of the persistence
exponent (\ref{intro_dhp}). This is a reflection of a general property of the class
of Pfaffian  kernels we consider: the leading
order asymptotics of the gap probability for a non-translationally invariant kernel and its translationally invariant
`bulk' limit coincide, while the constant terms differ (but are explicitly known).  
Notice that the Pfaffian point process describing the edge scaling limit
of the Gaussian orthogonal ensemble does not seem to possess this property.
Still, it is quite satisfying that there is a rather diverse pool of examples of Pfaffian
point processes  which can be all
treated using Theorems \ref{pfaffian_czego}, \ref{nonTIszego}, \ref{pfaffian_szego_edge}.

Moreover, it turns out that our results can be applied to zeros of random
polynomials with i.i.d. mean zero, but not necessarily Gaussian, coefficents, which 
are no longer described by a Pfaffian
point process. In a remarkable paper \cite{Dembo_Poonen_Shao_Zeitouni}, Dembo et al. show
the probability that such a random polynomial of degree $n$ has no zeros 
decays polynomially as $n^{-b}$. They show that this asymptotic is controlled 
by the case of Gaussian coefficients, and characterise the decay rate in terms of the Gaussian power series. This immediately  enables an 
application of our asymptotic results for Fredholm Pfaffians to evaluate the unknown power $b$,
 see Section \ref{s2.2.1} for more details.
\paragraph*{Literature review}
The analysis of Fredholm Pfaffians is an increasingly active area of research due to
applications ranging from interacting particle systems to random matrices. 
Our 
interest in the asymptotics of Fredholm Pfaffians was generated by the work
of Peter Forrester \cite{Forrester_bulk_gap} who used the connection
between the ABM and the law of real eigenvalues in the real Ginibre
ensemble from \cite{TZ}, and formula (\ref{intro_dz}) of Derrida-Zeitak
\cite{Derrida_Zeitak} to calculate the bulk scaling limit for the corresponding gap probability.
In \cite{Poplavskyi_Schehr} Mikhail Poplavskyi and Gregory Schehr use the link between Kac polynomials 
and the ensemble
of truncated orthogonal matrices from \cite{Forrester} to calculate the leading order asymptotic for 
the gap probability
for the real roots of Kac polynomials. It is worth
stressing that their work is independent of our own. A common feature of
the Pfaffian point processes related to Kac polynomials, truncated orthogonal
random matrices and exit measures  for interacting particle systems is
the appearing of the scalar $\sech$-kernel, see Section \ref{s2} for a
review of `derived' Pfaffian point processes and the corresponding scalar kernels.
Exploiting the integrability of the $\sech$-kernel, Ivan Dornic analyses
the asymptotics of a distinguished solution to Painlev\'e VI equation
to re-derive the formula for the empty interval probability for the real roots
of Kac polynomials, \cite{dornic2018universal}. Interestingly enough,
the integrable structure of kernels associated with the single time law of
CABM and the real Ginibre ensemble is not at all obvious. However, in a recent
series of papers \cite{baik2020largest}, \cite{baik2020edge} Jinho Baik
and Thomas Bothner establish a link between the Pfaffian kernel for
the edge scaling limit of the real Ginibre ensemble and Zakharov-Shabat integrable
hierarchy. By analysing the associated Riemann-Hilbert problem using the non-linear
steepest descent method they manage to obtain the right tails of the distribution
for the maximal real eigenvalue for the thinned real Ginibre ensemble up to 
and including the constant term (Lemma 1.14 of \cite{baik2020edge}). 
The answers we obtain for this regime
in Section \ref{s2.2} coincide with the answers presented in \cite{baik2020largest}, \cite{baik2020edge}. 
It is worth pointing out, that at the moment, there
seems to be no link between the integrable structures uncovered in the context
of persistence problems and the laws of extreme particles. However, given that
these problems can be treated in a unified way using Theorems \ref{pfaffian_czego}
and \ref{pfaffian_szego_edge}, it is reasonable to conjecture that there is a universal
integrable structure underlying the classes of kernels introduced in the present work. 
(See also \cite{krajenbrink2020painleve} for an extended study of a class Fredholm determinants
appearing as solutions of Zakharov-Shabat system which includes not only
the real Ginibre case, but also the eigenvalue statistics for Gaussian orthogonal
and symplectic ensembles.)
As mentioned above, our proof of the asymptotic theorems generalises the probabilistic
method of Kac. At the moment, there seem to be very little intersection between
the approaches to the analysis of Fredholm Pfaffians based on integrable systems
and integrable probability. Nevertheless, such a connection might well exist. For example,
in a recent paper \cite{krajenbrink2021inverse} Alexandre Krajenbrink and Pierre Le Doussal study short time large-deviations behaviour of the solutions to the Kadar-Parisi-Zhang equation. They find that the corresponding critical point equations are closely
related to the non-linear Schr\"{o}dinger equations and can be solved exactly using
the inverse scattering treatment of the corresponding Zakharov-Shabat system. One
of the crucial equations appearing 
is an integral equation with a quadratic non-linearity under the integral sign. Using
the probabilistic interpretation, the authors show that this non-linear equation is equivalent
to a linear integral equation typical of the theory of random walks. 
Subsequently, they analyse the linear equation using Sparre Anderson's formula,
which is a cornerstone of our probabilistic argument as well.   

\paragraph*{Paper organisation}
In Section \ref{s2} we introduce the
two  classes of Pfaffian
point processes we consider and state, in Theorems \ref{pfaffian_czego} and \ref{pfaffian_szego_edge}, the asymptotic formulae for the corresponding
Fredholm Pfaffians. In Section \ref{s2.2} we apply these formulae to calculate
gap probabilities for the range of Pfaffian point processes described above.
Subsection \ref{s2.3} studies exit measures for systems
of coalescing/annihilating Brownian motions, establishes their Pfaffian structure
and determines the corresponding kernels. This is a new result concerning Pfaffian
point processes, which
enables the application of the asymptotic theorems to the study of the law
of the leftmost particle for coalescing/annihilating Brownian motions. The rest
of the paper is dedicated to the proof of the stated theorems: in Section \ref{s3}
we prove the asymptotic formula for Fredholm Pfaffians for translationally invariant
kernels; in Section \ref{s4} we prove the asymptotic statement for 
a non-translationally invariant case; in Section \ref{s5} we prove the Pfaffian point
process structure for exit measures. 
Finally, Section \ref{s6} gives details of proofs for some more technical 
statements made in the paper, including a Fourier form for the coefficients
of the  asymptotic expansion for translationally invariant Fredholm Pfaffians and
the analytic properties of these coefficients with respect to the thinning
parameter.

\vspace{.1in} 

\noindent
{\bf Remark.} Some results of the present
paper were reported at the conference `Randomness and Symmetry'
held at the University College Dublin in June 18-22, 2018 (see also the thesis
\cite{FWR}).
\section{Statement of results}\label{s2}
We will state two theorems on asymptotics for Fredholm Pfaffians for 
kernels on $\mathbb{R}$, the first for translationally invariant kernels
and the second for a class of non translationally invariant kernels.
We start with recalling basic facts about determinantal and Pfaffian point
processes, Fredholm determinants and Pfaffians and defining two
classes of Pfaffian kernels the paper is dealing with.
\subsection{Backround}
\subsubsection{Definitions: Fredholm determinants and Pfaffians}\label{sFDFP}
A determinantal point process $X$ on $\mathbb{R}$ with kernel $T: \mathbb{R}^2 \to \mathbb{C}$
is a simple point process, whose $n$-point intensities $\rho_n(x_1,\ldots,x_n)$ exist for all $n \geq 1$ and are given by 
\[
\rho_n(x_1,\ldots,x_n) = \det \left( T(x_i,x_j): 1 \leq i,j \leq n \right) \quad \mbox{for all $x_1, \ldots, x_n \in \mathbb{R}$.}
\]
A Pfaffian point process $X$ on $\mathbb{R}$ with kernel $\mathbf{K}: \mathbb{R}^2 \to M_{2 \times 2}(\mathbb{C})$ is a simple 
point process, whose $n$-point intensities  $\rho_n(x_1,\ldots,x_n)$ exist for all $n \geq 1$ and are given by 
\[
\rho_n(x_1,\ldots,x_n) = \pf \left( \mathbf{K}(x_i,x_j): 1 \leq i,j \leq n \right) \quad \mbox{for all $x_1, \ldots, x_n \in \mathbb{R}$.}
\]
We also meet Pfaffian point processes on intervals $I \subset \mathbb{R}$.

We will consider real Pfaffian kernels $\mathbf{K}: \mathbb{R}^2 \to M_{2 \times 2}(\mathbb{R})$, written as
\begin{equation}
\label{matrix_kernel_notation}
 \mathbf{K}(x, y)  =   \left( \begin{matrix}
                             K_{11}(x, y) & K_{12}(x, y) \\
                             K_{21}(x, y) & K_{22}(x, y)
                            \end{matrix}
\right) \text{ for all } x, y \in \mathbb{R}.
\end{equation}
Recall the antisymmetry requirement on Pfaffian kernels:
\begin{equation}
\label{kernel_antisymmetry}
K_{ij}(x, y) = - K_{ji}(y, x) \qquad \text{ for all } x, y \in \mathbb{R}, \text{ and }i, j \in \{1, 2\}
\end{equation}
which ensures that the matrix $(\mathbf{K}(x_i,x_j): 1 \leq i,j \leq n)$ is a $2n \times 2n$ 
antisymmetric matrix. Recall that the Pfaffian $\pf(A)$ of a $2n \times 2n$ anti-symmetric matrix 
is defined by 
\[
\pf(A) = \frac{1}{2^n n!} \sum_{\sigma \in \Sigma_{2n}} \sgn(\sigma) \prod_{i=1}^n A_{\sigma(2i-1), \sigma(2i)},
\]
and recall the basic properties of Pfaffians: $(\pf(A))^2 = \det(A)$ and $\pf(BAB^T) = \det(B) \pf(A)$ for any 
$B$ of size $2n \times 2n$. We usually only list the elements $(A_{ij}: i<j \leq 2n)$ defining an anti-symmetric matrix. 

For a simple point process $X$, having all intensities $(\rho_n:n \geq 1)$, the gap probabilities, 
writing $X(a,b)$ as shorthand for $X((a,b))$, are given by
\[
\Pp [X(a,b) = 0] = \sum_{n=0}^{\infty} \frac{(-1)^n}{n!} \int_{[a,b]^n} \rho_n(x_1,\ldots,x_n) dx_1 \ldots dx_n
\]
(where the term $n=0$ is taken to have the value $1$) whenever this series is 
absolutely convergent (see Chapter 5 of Daley and Vere-Jones \cite{daley_vere_jones}). 
For determinantal or Pfaffian point processes this leads to the following expressions, which can be taken 
as the definitions of the Fredholm determinant and Fredholm Pfaffian of the kernels $T$ and $\mathbf{K}$ respectively:
\begin{eqnarray} 
 \Det_{[a,b]} (I -T) &=&  \sum_{n=0}^{\infty}  \frac{(-1)^n}{n!} \int_{[a,b]^n}  \det (T(x_i,x_j): 1 \leq i,j  \leq n) \, dx_1 \ldots dx_n, \label{FDdefn} \\
\Pf_{[a,b]}(\mathbf{J} - \mathbf{K}) &=& \sum_{n=0}^{\infty} \frac{(-1)^n}{n!} \int_{[a,b]^n}  
\Pf( \mathbf{K}(x_i,x_j): 1 \leq i,j  \leq n) \, dx_1 \ldots dx_n.   \label{FPdefn}
\end{eqnarray}
Here the left hand side is merely a notation for the right hand side. However later, see (\ref{starttw}), we exploit 
operator techniques to manipulate Fredholm Pfaffians. Then we treat $\mathbf{J}$ as an operator
on $(L^2[a,b])^2$ defined by $\mathbf{J}(f,g)=(g,-f)$.
When $T$ or $\mathbf{K}$ are bounded functions (as in all our results) these series converge 
absolutely (see chapter 24 of Lax \cite{Lax}). 
We will also consider semi-infinite intervals $[a,\infty)$ and then the series converge absolutely under suitable
decay conditions on $T$ and $\mathbf{K}$. More generally we will define
$ \Det_{[a,\infty)} (I -T)  =  \lim_{b \to \infty} \Det_{[a,b]} (I -T) $ and 
$ \Pf_{[a,\infty)}(\mathbf{J} - \mathbf{K}) =  \lim_{b \to \infty}  \Pf_{[a,b]}(\mathbf{J} - \mathbf{K}) $
whenever these limits exist (as they do when they represent gap probabilities). 
%
\subsubsection{Classes of Pfaffian kernels considered} \label{s2.1}
We list the hypotheses required on the kernel of the Fredholm Pfaffians for our results. 
We consider kernels $\mathbf{K}$ in {\it derived form} (see \cite{anderson2010} Definition 3.9.18 for essentially this notion). 
For us, this means that $\mathbf{K}$ is derived from a single scalar function $K \in C^2(\mathbb{R}^2)$ as follows.  
There is a simple jump discontinuity along $x=y$ and we let $S(x,y) = \sgn(y-x)$ (with the convention that $S(x,x)=0$) to display 
this discontinuity. Then $\mathbf{K}$ has the form
\begin{equation}
\label{derived_form}
\mathbf{K}(x, y) = \left( \begin{matrix}
                         S(x,y) + K(x, y) & -D_2 K(x, y)  \\
                         -D_1 K(x, y) & D_1 D_2 K(x, y)
                                             \end{matrix}
\right)  
\end{equation}
 where $D_i K$ denotes  the derivative of $K$ in the $i$-th co-ordinate.
The symmetry condition \eqref{kernel_antisymmetry} then implies that $K$ is anti-symmetric (that is $K(x,y)=-K(y,x)$ 
for all $x,y$). All our applications  are Pfaffian point processes with a kernel of the form $p\mathbf{K}$ for some
$p \in [0,1]$, where $\mathbf{K}$ is in the above form. 

In order to apply probabilistic methods in our analysis, we will assume that $K$ is given in terms of a
probability density function $\rho$, that is $\rho \geq 0, \int_{\mathbb{R}} \rho(x) dx = 1$.
In our translationally invariant examples $K$ will be given 
as
\begin{equation} \label{TIform}
K(x,y) = -2 \int^{y-x}_0 \rho(z) dz, \qquad  \mbox{for  even } \rho \in C^1(\mathbb{R}) \cap L^{\infty}(\mathbb{R})
\end{equation}
(the symmetry condition \eqref{kernel_antisymmetry} requires that $\rho$ must be an even function). 

In our non translationally invariant examples $K$ will be given in terms of a probability density function 
$\rho \in C^1(\mathbb{R}) \cap  H^1(\mathbb{R}) \cap L^{\infty}(\mathbb{R})$ as 
\begin{equation} \label{NTIform}
K(x,y) = \int_{-\infty}^0  \left| \!\!\! \begin{array}{cc}
\int^{x-z}_{-\infty} \rho(w) dw & \int^{y-z}_{-\infty} \rho(w) dw \\
\rho(x-z) & \rho(y-z) 
\end{array} \!\!\!\right| dz
\end{equation}
(where $|A|$ stands for the $2 \times 2$ determinant). The required symmetry for $K$ holds without any symmetry requirement on 
$\rho$. Note that
\begin{equation} \label{tilderho}
\lim_{c \to -\infty} K(x+c,y+c) =  -2 \int^{y-x}_0 \tilde{\rho}(z) dz  \quad \mbox{where} \quad
\tilde{\rho}(z) = \int_{-\infty}^{\infty} \rho(w) \rho(w-z) dw
\end{equation}
so that the kernel is close to the translation invariant form (\ref{TIform}) near $-\infty$. Indeed our examples
that fit this non translationally framework are from (i) random matrices that are studied near the right hand edge of their spectrum
and (ii) particle systems that are started with `half-space' initial conditions, that is where particles are initially spread over the
half-space $(-\infty,0]$, and in both these cases the kernels approach the `bulk' form far from the origin.

We remark, see  \cite{anderson2010} equation (3.9.32), that the kernel for limiting ($N=\infty$) GOE 
Pfaffian point process of eigenvalues in the bulk is in the derived form $\frac12 \mathbf{K}$, with $K$ in the translationally 
invariant form  (\ref{TIform})  for 
$\rho(z) = \frac{1}{\pi} \frac{\sin(z)}{z}$ (which, while not non-negative, at least satisfies 
$\int_{\mathbb{R}} \rho = 1$). Moreover, see  \cite{anderson2010} equation (3.9.41), the limiting ($N=\infty$) edge kernel for GOE eigenvalues  is in the derived form $\frac12 \mathbf{K}$, with $K$ in the translationally 
non-invariant form  (\ref{NTIform})  for 
$\rho(z) = A(z)$ the Airy functions (which again satisfies 
$\int_{\mathbb{R}} \rho = 1$). 
%
\subsection{Asymptotics for Fredholm Pfaffians: translationally invariant kernels} \label{s2.1.1}
\begin{theorem}
\label{pfaffian_czego}
Let $\mathbf{K}$ be in the derived form \eqref{derived_form} using a scalar kernel $K$ in the form \eqref{TIform}
for a density function $\rho$. Then, for $0 < p < 1$ and under the moment assumptions given below, the asymptotic 
\begin{equation*}
 \log \Pf_{[0,L]}(\mathbf{J}- p \mathbf{K}) = -\kappa_1(p) L + \kappa_2(p) + o(1) \quad \text{ as } L \rightarrow \infty
\end{equation*}
holds, where (writing $\rho^{*n}$ for the $n$-fold convolution of $\rho$ with itself) 
\begin{enumerate}[(i)]
 \item for $0<p < 1/2$, supposing  $\int_\R |x| \rho(x) dx < \infty$,
 \begin{eqnarray*}
\kappa_1(p) &=&  \sum_{n=1}^{\infty} \frac{(4p(1-p))^n}{2n} \rho^{*n} (0),\\
\kappa_2(p) 
 & = &  \log\left(\frac{\sqrt{1-2p}}{1-p}\right) + 
 \int_0^{\infty} \frac{x}{2} \left( \sum_{n=1}^{\infty} \frac{(4p(1-p))^n  \rho^{*n}(x)}{n} \right)^2 dx;
\end{eqnarray*}
\item for $p = 1/2$, supposing  $\int_\R |x|^4 \rho(x) dx < \infty$,
 \begin{eqnarray*}
\kappa_1(1/2) &=&  \sum_{n=1}^{\infty} \frac{1}{2n} \rho^{*n} (0),\\
\kappa_2(1/2) & = & \log 2 - \frac14
+ \frac{1}{2}\sum_{n=2}^{\infty} \left( \sum_{k=1}^{n-1}
\int_0^{\infty} x \frac{\rho^{*k}(x) 
\rho^{*(n-k)}(x)}{k(n-k)} dx - \frac{1}{2n} \right);
\end{eqnarray*}
\item for $1/2 < p < 1$, supposing  there exists $\phi_p>0$ so that $4 p(1-p) \int_\R e^{\phi_p x} \rho(x)dx = 1$
and for which $\int_\R |x| e^{\phi_p x} \rho(x)dx < \infty$,
 \begin{eqnarray*}
\kappa_1(p) &=&  \phi_p + \sum_{n=1}^{\infty} \frac{(4p(1-p))^n}{2n} \rho^{*n}(0),\\
\kappa_2(p)  & = &   \log \left( \frac{\sqrt{2p-1}}{8p(1-p)^2} \right)  
+  \int_0^{\infty} \frac{x}{2} \left( \sum_{n=1}^{\infty} \frac{(4p(1-p))^n  \rho^{*n}(x)}{n} \right)^2 dx \\
&& - \log \left( \phi_p \int_{\mathbb{R}} x e^{\phi_p x}\rho(x) dx \right) - 2  \sum_{n=1}^{\infty} 
\frac{(4p(1-p))^n}{n} \int_{-\infty}^0  e^{\phi_p x} \rho^{*n}(x) dx.
\end{eqnarray*}
\end{enumerate}
\end{theorem}

\noindent
\textbf{Remark 1.} Motivated by our applications, implementation of this theorem for the 
densities $\rho(x) = \pi^{-1} \sech(x)$ and $\rho(x) = (4 \pi t)^{-1/2} \exp(-x^2/4t)$ will be 
done in Corollary \ref{sechkernel} and \ref{gaussiankernel}, where simpler expressions for $\kappa_1, \kappa_2$ are calculated. 
In both cases the function $p \to \kappa_1(p)$ for $p \in (0,1)$ turns out to be analytic. We do not know whether 
this is the general case for, say, analytic $\rho$. 

\vspace{.1in} 

\noindent
\textbf{Remark 2.} The case $p=1$ is not covered by our theorem. In this case the Fredholm Pfaffian reduces 
to the (square root) of a $2 \times 2$ determinant (this is evident from the proof in section \ref{s3.1}). Then 
the asymptotics are easier and they need not be exponential (see Remark 2. in Section \ref{s2.2.2}).

\vspace{.1in} 

\noindent
\textbf{Remark 3.} For numerical or theoretical analysis, it may be useful to rewrite the
infinite sums in these formulae in alternate ways. For example, the probabilistic representations 
(see (\ref{probrep1}) for $p \in (0,\frac12)$ and (\ref{probrep2}) for 
$p \in (\frac12,1)$) give formulae for $\kappa_i(p)$ when $p \neq \frac12$ as expectations.
We may also, for suitably good $\rho$, re-write
some terms in the constants $\kappa_i(p)$ usefully in terms of 
Fourier transforms. We use the conventions $ \hat{\rho}(k) = \int \exp(i k x) \rho(x) dx $ with inversion
(when applicable) $\rho(x) = (2 \pi)^{-1} \int \exp(-i k x) \hat{\rho}(k) dk $. The exponents
$\kappa_i(p)$ can be expressed using the function
\begin{equation} \label{Ldefn}
L_{\rho}(p,x) = \frac{1}{2 \pi} \int_{\mathbb{R}}   e^{-ikx} \log \left(1 - 4p(1-p) \hat{\rho}(k) \right) dk.
\end{equation} 
We do not look for good sufficient conditions, but suppose for the formulae (\ref{k1FT}),(\ref{k2FT})
below that $\rho$ is in Schwartz class.
When $p=\frac12$ we assume 
further that there exists $\mu>0$ so that $\int_\R e^{2 \mu |x|} \rho(x) dx < \infty$ which
justifies certain contour manipulations in the proof. 
\begin{eqnarray} 
\kappa_1(p) & = & \left\{ \begin{array}{ll}  
- \frac12 L_{\rho}(p,0) & \mbox{for $p \in (0,\frac12]$,} \\
- \frac12 L_{\rho}(p,0) + \phi_p & \mbox{for $p \in (\frac12,1)$,}
 \end{array} \right. \label{k1FT} \\
&& \nonumber\\
\kappa_2(p) &=&  \left\{ \begin{array}{ll}  
\log\left(\frac{\sqrt{1-2p}}{1-p}\right) + \frac12
 \int_0^{\infty} x L_{\rho}^2(p,x) dx
 & \mbox{for $p \in (0,\frac12)$,} \\
\frac14 \log \left( 2 \sigma^2\right) - \frac{\gamma}{2} 
- \frac12 \int_0^{\infty} \log(x) (x^2 L^2_{\rho}(\frac12,x))' dx
 & \mbox{for $p = \frac12$,} \\
 \log \left( \frac{\sqrt{2p-1}}{8p(1-p)^2} \right)  - \Gamma_p  + \frac12
 \int_0^{\infty} x L_{\rho}^2(p,x) dx
 & \mbox{for $p \in (\frac12,1)$,}
 \end{array} \right.   \label{k2FT}
\end{eqnarray}
where 
\begin{equation} \label{Gammap}
\Gamma_p = \log \left( \phi_p \int_{\mathbb{R}} x e^{\phi_p x}\rho(x) dx \right)  -  \frac{1}{\pi}  \int_\R \frac{\phi_p}{\phi^2_p+k^2}
\log( 1- 4p(1-p) \hat{\rho}(k)) dk
\end{equation}
and where
$\sigma^2 = \int_\R x^2 \rho(x) dx$ and $\gamma$ is the Euler Mascheroni constant.
When $p=\frac12$ the formula for $\kappa_2$ encodes the slightly arbitrary 
compensation by $-\frac{1}{2n}$ used in Theorem \ref{pfaffian_czego} in, perhaps, a more natural way using the transform.
It is tempting to integrate by parts in the term $\int_0^{\infty} \log x (x^2 L^2_{\rho}(\frac12,x))' dx$ in 
order to obtain a form closer to those when $p \neq \frac12$, but this is not justified
(for example $x L_{\rho}(\frac12,x) \to -1$ as $x \to \infty$). 
The proofs of these alternative formulae are in the subsection \ref{s6.1}.
Notice that the leading order term $-\fr L_{\rho}(p,0)L$ for $p<\fr$ is equal to one half
of the leading term in the classical Akhiezer-Kac formula for the Fredholm determinant
of the operator with the kernel $\rho$. This is due to the Tracy-Widom map between Fredholm
Pfaffians and determinants discussed in section \ref{s3.1}.
\subsection{Asymptotics for Fredholm Pfaffians: non translationally invariant kernels}  \label{s2.1.2}
We start with a result on Fredholm determinants, with a kernel in the form that arises in our applications. 
Indeed an operator in the form (\ref{hypothesis3}) below arises immediately in the 
analysis of the Fredholm Pfaffian of a derived form kernel $\mathbf{K} $ in the special 
non translationally invariant form \eqref{NTIform} (see (\ref{3010})). 
\begin{theorem}

\label{nonTIszego}
Suppose that $\rho \in C(\mathbb{R}) \cap L^2(\mathbb{R}) \cap L^{\infty}(\mathbb{R}) $ is a probability density function. 
Define
\begin{equation} \label{hypothesis3}
T(x, y)  =  \int_{-\infty}^0 \rho(x-z) \rho(y-z) dz
\end{equation}
and let $\tilde{\rho}(z) = \int_{\infty}^{\infty} \rho(w) \rho(w-z) dw$.

\vspace{.1in}
\noindent
For $\beta \in [0,1)$, supposing $\int_\R |x| \rho(x) dx < \infty$, 
\begin{equation*}
 \log \Det_{[-L,\infty)}(I- \beta T) = - \kappa_1(\beta) L + \kappa_2(\beta) + o(1) \quad \text{ as } L \rightarrow \infty
\end{equation*}
where
\[
\kappa_1(\beta) = \sum_{n=1}^{\infty} \frac{\beta^n}{n}  
\tilde{\rho}^{*n}(0)
\]
and 
\[
 \kappa_2(\beta) = \sum_{n=1}^{\infty}
  \frac{\beta^n}{n^2} \int_{-\infty}^{\infty} x (\rho^{*n}(x))^2 dx +
 \int_0^{\infty} x \left( \sum_{n=1}^{\infty} \frac{\beta^n  \tilde{\rho}^{*n}(x)}{n} \right)^2 dx. 
\]
When $\beta=1$,  supposing $\int_\R x^4 \rho(x) dx < \infty$,  
\[
 \log \Det_{[-L,\infty)}(I- T)  =  - \sum_{n=1}^{\infty} \frac{1}{n}  
\tilde{\rho}^{*n}(0) L +  \log L + \kappa_2 
+ o(1) \quad \text{ as } L \rightarrow \infty
\]
where, setting $\tilde{\sigma}^2 = \int_\R x^2 \tilde{\rho}(x) dx$,
\begin{eqnarray*}
\kappa_2 &=& \frac{3}{2} \log 2 -  \frac12 - \log \tilde{\sigma} 
+ \sum_{n=1}^{\infty} \frac{1}{n^2} \int_{-\infty}^{\infty} x \left(\rho^{*n}(x)\right)^2 dx \\
& & \hspace{.2in} + \sum_{n=2}^{\infty} \left( \sum_{k=1}^{n-1}
\int_0^{\infty} x \frac{\tilde{\rho}^{*k}(x) 
\tilde{\rho}^{*(n-k)}(x)}{k(n-k)} dx - \frac{1}{2n} \right) 
\end{eqnarray*}
\end{theorem}
We do not discuss any direct applications of Theorem \ref{nonTIszego} in this paper.
However, note that Fredholm determinants of integral operators with kernels of the form (\ref{hypothesis3}) given by a
Hankel convolution have recently been linked to integrable hierarchies of partial differential equations, such as 
the non-linear Schrodinger equation, see \cite{krajenbrink2020painleve}. Moreover, if $\rho$ is Gaussian, the corresponding Fredholm
determinant appears in the weak noise theory of Kardar-Parisi-Zhang equation, see \cite{krajenbrink2021inverse}.

Our main result for non translationally invariant Fredholm Pfaffians is as follows.  
\begin{theorem}
\label{pfaffian_szego_edge}
Let $\mathbf{K}$ be in the derived form \eqref{derived_form} using a kernel in the form
\eqref{NTIform} for a probability density function 
$\rho \in C^1(\mathbb{R}) \cap  H^1(\mathbb{R}) \cap L^{\infty}(\mathbb{R})$.
 Define $\tilde{\rho}(z) = \int_{\mathbb{R}} \rho(w) \rho(w-z) dw$.
Then, for $0 < p < 1$ and under the moment assumptions given below, the asymptotic 
\begin{equation*}
 \log \Pf_{[-L,\infty)}(\mathbf{J}- p \mathbf{K}) = - \kappa_1(p) L + \kappa_2(p) + o(1) \quad \text{ as } L \rightarrow \infty
\end{equation*}
holds, where 
\begin{enumerate}[(i)]
 \item for $0<p < 1/2$, supposing  $\int_\R |x| \rho(x) dx < \infty$, 
 \begin{eqnarray*}
\kappa_1(p) &=&  \sum_{n=1}^{\infty} \frac{(4p(1-p))^n}{2n} \tilde{\rho}^{*n} (0),\\
\kappa_2(p) &=& \frac12 \log\left(\frac{1-2p}{1-p} \right) +  \frac12  \sum_{n=1}^{\infty}
  \frac{(4p(1-p))^n}{n^2} \int_{-\infty}^{\infty} \!x\, (\rho^{*n}(x))^2 dx \\
  && \hspace{.4in} + 
 \int_0^{\infty}  \frac{x}{2} \Big( \sum_{n=1}^{\infty} \frac{(4p(1-p))^n  \tilde{\rho}^{*n}(x)}{n} \Big)^2 dx;
\end{eqnarray*}
\item for $p = 1/2$, supposing  $\int_\R |x|^4 \rho(x) dx < \infty$, 
 \begin{eqnarray*}
\kappa_1(1/2) &=&  \sum_{n=1}^{\infty} \frac{1}{2n} \tilde{\rho}^{*n} (0),\\
\kappa_2(1/2) & = & \frac12 \log 2 - \frac14 +  \frac12  \sum_{n=1}^{\infty}
  \frac{1}{n^2} \int_{-\infty}^{\infty} \!x\, (\rho^{*n}(x))^2 dx \\
  && + \frac{1}{2}\sum_{n=2}^{\infty} \left( \sum_{k=1}^{n-1}
\int_0^{\infty} x \frac{\tilde{\rho}^{*k}(x) 
\tilde{\rho}^{*(n-k)}(x)}{k(n-k)} dx - \frac{1}{2n} \right);
\end{eqnarray*}
\item for $1/2 < p < 1$, supposing there exists $\phi_p$ so that  $4 p(1-p) \int_\R e^{\phi_p x} \tilde{\rho}(x)dx = 1$
and that the integrals $\int_\R |x| e^{\phi_p x} \tilde{\rho}(x)dx$ and $\int_\R e^{\phi_p |x|} \rho(x)dx$ are finite,
 \begin{eqnarray*}
\kappa_1(p) &=&    \phi_p + \sum_{n=1}^{\infty} \frac{(4p(1-p))^n}{2n} \tilde{\rho}^{*n}(0),\\
\kappa_2(p)  & = & \log\left( \frac{\sqrt{2p-1}}{16 p^{3/2}(1-p)^2} \right) +  \frac12  \sum_{n=1}^{\infty}
  \frac{(4p(1-p))^n}{n^2} \int_{-\infty}^{\infty} \!x\, (\rho^{*n}(x))^2 dx \\
  && + 
 \int_0^{\infty} \! \frac{x}{2} \, \Big( \sum_{n=1}^{\infty} \frac{(4p(1-p))^n  \tilde{\rho}^{*n}(x)}{n} \Big)^2 dx  
 - \log \left( \phi_p\int_{\mathbb{R}} x e^{\phi_p x} \tilde{\rho}(x) dx \right) \\
&& - \log \left(\int_{\mathbb{R}} e^{\phi_p x} \rho(x) dx \right) - 2  \sum_{n=1}^{\infty} 
\frac{(4p(1-p))^n}{n} \int_{-\infty}^0  e^{\phi_p x} \tilde{\rho}^{*n}(x) dx. 
\end{eqnarray*}
\end{enumerate}
\end{theorem}

\noindent \textbf{Remark.} Implementation of this theorem for the 
densities $\rho(x) = (2 \pi t)^{-1/2} \exp(-x^2/4t)$ and $\rho(x) = \frac{2}{\sqrt{\pi}} \exp(x - e^{2x})$, both arising 
from applications, is done in Sections \ref{s2.2.3} and \ref{s2.2.4}
where simpler expressions for $\kappa_1, \kappa_2$ are calculated.

\section{Applications} \label{s2.2}
We repeatedly use the following lemma.
\begin{lemma} \label{littlelemma}
Suppose $X$ is a Pfaffian point process on an interval $A \subseteq \mathbb{R}$ with a kernel $\mathbf{K}$ in the 
derived form \eqref{derived_form}, with an underlying scalar kernel $K$. Suppose $\phi:A \to \mathbb{R}$ is 
$C^1$ and strictly increasing. Then the push forward point process $X'$ of $X$
under $\phi$, given by $X'(\cdot)=X(\phi^{-1}(\cdot))$, is still a Pfaffian point process on the
interval $\phi(A)$ with a kernel $\mathbf{K}'$ in the 
derived form \eqref{derived_form}, with the underlying scalar kernel $K'(x,y) = K(\phi^{-1}(x),\phi^{-1}(y))$.
\end{lemma}

\noindent
\textbf{Proof.}
The intensities $\rho_n'$ for $X'$ are given by 
\[
\rho_n'(x_1,\ldots,x_n) = \rho_n(\phi^{-1}(x_1),\ldots, \phi^{-1}(x_n)) \prod_{k=1}^n \alpha_k
\quad \mbox{where $\alpha_k = (\phi^{-1})'(x_k)$.}
\]
Also 
\begin{eqnarray*} 
&& \hspace{-.4in}
\pf \left( \mathbf{K}'(x_i,x_j): 1 \leq i,j \leq n \right) \\
& = &
\pf_{i,j \leq n}  \left( \begin{matrix}
                         S(x_i,x_j) + K'(x_i, x_j) & -D_2 K'(x_i, x_j)  \\
                         -D_1 K'(x_i, x_j) & D_1 D_2 K'(x_i, x_j)
                                             \end{matrix} \right)  \\    
& = &
\pf_{i,j \leq n} \left( \begin{matrix}
                         S(x_i,x_j) + K(\phi^{-1}(x_i), \phi^{-1}(x_j)) & -D_2 K(\phi^{-1}(x_i), \phi^{-1}(x_j)) \alpha_j \\
                         -D_1 K(\phi^{-1}(x_i), \phi^{-1}(x_j)) \alpha_i & D_1 D_2 K(\phi^{-1}(x_i), \phi^{-1}(x_j))
                         \alpha_i \alpha_j
                                             \end{matrix} \right)                          
\end{eqnarray*}
and the factors of $\alpha_k$ can be extracted as this is the conjugation with a block diagonal matrix $D$ with blocks
$\left( \begin{matrix}
                         1 & 0 \\
                         0& \alpha_i 
                                             \end{matrix} \right) $ for which $\det(D) =  \prod_{k=1}^n \alpha_k$. \qed
\subsection{Zeros of Gaussian Power Series}  \label{s2.2.1}
%
Let $(a_k)_{k\geq 0}$ be an independent  collection of real $N(0, 1)$ random variables and define the 
{\it Gaussian power series}
$  f(z) = \sum_{k=0}^{\infty} a_k z^k$. The series converges almost surely to a continuous function 
on $|z|<1$ and we consider the real zeros of $f$ as a point process $X$ on $(-1,1)$. Forrester \cite{Forrester} (see 
Theorem 2.1 of  Matsumoto and Shirai \cite{Matsumoto_Shirai}) showed that $X$ is a Pfaffian point process with 
kernel $\frac12 \mathbf{K}$, with $\mathbf{K}$ in derived form \eqref{derived_form} with the choice 
\[
K(x,y) = \frac{2}{\pi} \sin^{-1}\left( \frac{\sqrt{(1-x^2)(1-y^2)}}{1-xy}\right) - 1 \qquad \mbox{for $x<y$.}
\]
Using Lemma \ref{littlelemma},
the push forward of the process $X$ under the mapping $x \mapsto  \Phi(x) := \frac12 \log(\frac{1+x}{1-x})$ 
is a Pfaffian point process on $\mathbb{R}$, still in the derived 
form, with the choice (note that $\Phi^{-1}(x) = \tanh(x)$) 
\[
K(\Phi^{-1}(x), \Phi^{-1}(y)) =
\frac{2}{\pi} \sin^{-1} (\sech(y-x)) - 1 \qquad \mbox{for $x<y$.}
\]
The problem has now become translationally invariant and the kernel is in the form (\ref{TIform})
with $\rho(z) = \pi^{-1} \sech(z)$. 
Theorem \ref{pfaffian_czego} with $p=\frac12$ leads to (see Corollary \ref{sechkernel} below)  
\bea\label{ref_needed}
\log \Pp[X(a,b) = 0] = -\frac{3}{8} (\Phi(b) - \Phi(a))  + \kappa_2(1/2) + o(1) 
\eea
where the term $o(1)$ converges to zero whenever $b \uparrow 1$ or $a \downarrow -1$ and $\kappa_2(1/2)$ 
is given by 
\[
\frac14 \log \Big(\frac{\pi^2}{2}\Big) - \frac{\gamma}{2} 
- \frac14 
 \int_0^{\infty} \log (x) (\tanh(x) + \tanh(x/2)) (\sech^2(x) + \mbox{$\frac12$} \sech^2(x/2)) dx.
 \]
In particular (\ref{ref_needed}) implies that
\[
\lim_{\epsilon \downarrow 0} 
\epsilon^{-\frac{3}{16}}  \, \Pp[X(0, 1- 2 \epsilon) = 0]  = 
\lim_{\epsilon \downarrow 0} 
\epsilon^{-\frac{3}{8}}  \, \Pp[X(-1+2 \epsilon, 1- 2 \epsilon) = 0] = e^{\kappa_2(\fr)}.
\]
It is possible that the remaining integral in (\ref{sechk2.2})
can be expressed in terms of special functions,
but it is also not hard to calculate it numerically which gives $\kappa_2(1/2) \approx 0.0247$.

As an application of the results obtained in this section let us prove the 
following statement which completes the theorem of Dembo, Poonen, Shao and Zeitouni
\cite{Dembo_Poonen_Shao_Zeitouni} concerning the zeros of random polynomials.
\begin{proposition}
Let $f_n: x\mapsto \sum_{i=0}^{n-1} a_i x^i$ be a random polynomial on $\R$,
where $(a_i)_{i \geq 0}$ is a sequence of i.i.d. random variables with zero mean, unit variance such that moments of all orders exist. Let $p_n = \Pp[f_n(x) > 0 \quad \forall x \in \mathbb{R}]$ be the probability that $f_n$ stays positive 
('persistence probability'). Then
\bea
\lim_{n \rightarrow \infty} \frac{\log p_{2n+1}}{\log n} = -\frac{3}{4}.
\eea
\end{proposition}
\begin{proof} All the hard work is done in \cite{Dembo_Poonen_Shao_Zeitouni}, where strong  approximations are used to show the asymptotic will follow from the Gaussian case, and the 
approximation of the Gaussian polynomial by the Gaussian power series is controlled.
We are left with an easy task: by Theorem 1.1 of \cite{Dembo_Poonen_Shao_Zeitouni}, the limit
\bea\label{willrem1}
b:=-\lim_{n \rightarrow \infty} \frac{\log p_{2n+1}}{\log n} 
\eea
exists and can be characterised in terms of the Gaussian power series as follows. Let $(Y_t)_{t\in \R}$ be a centered
stationary continuous Gaussian process with the correlation 
$R(t):=\E[Y_{0}Y_{t}]/\E[Y_{0}^2]=\sech(t/2)$. Then
\begin{equation}\label{dzchar}
b = -4 \lim_{T \rightarrow \infty} \frac{1}{T} \log \Pp[\sup_{0\leq t\leq T}Y_t\leq 0].
\end{equation} 
In other words the constant $b$ is universal.

The process $Y$ can be realised as a rescaling of the Gaussian power series $f$
followed by pushing it forward to a process on $\R$ by the function $2\Phi^{-1} =2 \tanh^{-1}$,
\beast
\tilde{Y}_t:=\frac{f(\tanh(t/2))}{\left(\E[f^2(\tanh(t/2))]\right)^{1/2}}, \quad ~t \in \R.
\eeast
Indeed, $(\tilde{Y}_t)_{t \in \R}$ is a continuous centered Gaussian process on $\R$ with 
with the correlation function
\beast
\E[\tilde{Y}_0\tilde{Y}_t]/\E[\tilde{Y}_0^2]=\E[\tilde{Y}_0\tilde{Y}_t]=\sech(t/2),
\eeast
which also implies the stationarity of $\tilde{Y}$.
Therefore the law of $Y$ coincides with the law of a constant multiple
of $\tilde{Y}$, meaning that the laws of real zeros of $Y$ and $\tilde{Y}$
coincide. It follows from the theorem of Forrester above that
the law of real zeros of $Y$ is a translationally invariant
Pfaffian point process with $\rho(\cdot)=\frac{\sech(\fr \cdot)}{2\pi}$,
where the factors of $2$ appear because $Y$ is the pushforward
of the Gaussian power series of by $2\Phi^{-1}$ rather than $\Phi^{-1}$. As a consequence
of the Fourier formula (\ref{k1FT}) for $\kappa_1$ and the $p=1/2$ statement of 
Corollary \ref{sechkernel} below,
\bea\label{sechkappa12}
\log \Pp \left[Y_t \neq 0 \quad \forall 0\leq t\leq T \right]=-\frac{\kappa_1}{2}T+o(T)
=-\frac{3}{16}T+o(T).
\eea
Therefore, using (\ref{dzchar}),
\beast
b &=&-4 \lim_{T \rightarrow \infty} \frac{1}{T} \log \Pp[\sup_{0\leq t\leq T}Y_t\leq 0]\\
& = &
-4 \lim_{T \rightarrow \infty} \frac{1}{T} \log \Pp[\sup_{0\leq t\leq T}Y_t< 0]\\
&=&-4\lim_{T \rightarrow \infty} \frac{1}{T} \log \Pp[Y_t\neq 0 \quad \forall 0\leq t\leq T]
=2\kappa_1=\frac{3}{4}.
\eeast
The second equality is due to the fact that zeros of $Y$ are almost surely simple; the 
third is due to the reflection symmetry of the process $Y$, the fourth
uses  (\ref{sechkappa12}).
\end{proof}

\noindent
\textbf{Remark.} It is worth stressing that
Theorem 1.1 is just one of the universality results presented in \cite{Dembo_Poonen_Shao_Zeitouni}: the case where $\E[a_i]\neq 0$ was also treated; 
the probability that random polynomials have exactly $k$ real zeros or the number of real zeros is $o(\log n/\log\log n)$ were analysed. For all the cases formulae analogous to (\ref{willrem1}) were proved (with $b\rightarrow b/2$ when means are non-zero). However, the value of the limit $b$ could only be calculated numerically as $b = 0.76 \pm 0.03$ and bounded rigorously as $0.4\leq b \leq 2$. For all these 
statements the unknown constant $b$ can now be replaced with $3/4$. 

\vspace{.1in}

We record now
the concrete application of Theorem \ref{pfaffian_czego} for the specific kernel based on
$\rho(x) = \pi^{-1} \sech(x)$ for all $p \in (0,1)$. 
The case $p=1/2$ yields the above application to Gaussian power series. The case
$p \in (0, 1/2)$ would correspond to a gap probability for the thinning of the point process formed 
by the zeros of a Gaussian power series (should this ever be needed). However, the sech kernel arises completely independently 
(as far as we know) in a later application in section \ref{s2.2.3}, 
where the problem of a system of coalescing/annihilating particles on $\R$ never crossing the origin by 
time $t$ is studied. That probability is related to a Fredholm Pfaffian with a 
non-translationally invariant kernel, but which asymptotically agrees with the kernel based on 
$\rho(z) = \pi^{-1} \sech(z)$. The corollary below then becomes needed for all $p \in (0,1)$.
It is also our first chance to study the regularity of $p \to \kappa_i(p)$. 
\begin{corollary}  \label{sechkernel}
Let $\mathbf{K}$ be a derived form kernel, in the translationally invariant form 
(\ref{TIform})  with $\rho(x) = \pi^{-1} \sech(x)$. Then for $p \in [0,1)$
\[
\log \Pf_{[0,L]}(\mathbf{J} - p \mathbf{K}) = -\kappa_1(p) L + \kappa_2(p) + o(1) \text{ as } L \rightarrow \infty 
\]
where 
\begin{equation} \label{k1sech}
\kappa_1(p) =   
\frac{2}{\pi^2} \left( \cos^{-1} \frac{1-2p}{\sqrt{2}} \right)^2    - \frac18 
\end{equation}
is real analytic for $p \in [0,1)$ and $\kappa_2(p)$ is given by
\begin{align}
&  \frac12  \int_0^{\infty} x L_{\rho}^2(p,x) dx + \log\left(\frac{\sqrt{1-2p}}{1-p}\right)
&& \mbox{$p < \frac12$}, \label{sechk2.1} \\
&  \frac14 \log \Big(\frac{\pi^2}{2}\Big) - \frac{\gamma}{2} 
- \frac18 
 \int_0^{\infty} \log (x) \left((\tanh(x) + \tanh(x/2))^2\right)^{'} dx 
 && \mbox{$p = \frac12$,} \label{sechk2.2}  \\ 
&
\frac12  \int_0^{\infty} x L_{\rho}^2(p,x) dx  
- \log \left( \cos^{-1}(4p(1-p)) \right) && \nonumber \\
& \hspace{.15in}  - \log \left( \frac{\sqrt{(2p-1)(1+4p-4p^2)}}{2p} \right)
&& \nonumber \\
& \hspace{.3in} + \frac{1}{\pi} \int_{\mathbb{R}} \frac{1}{1+k^2} \log \left(1- 4p(1-p) \sech(\cos^{-1}(4p(1-p))k)\right) dk
&& \mbox{$p> \frac12$},  \label{sechk2.3}
\end{align}
where, for $p \neq \frac12$ and $x \neq 0$, 
\begin{equation}  \label{sechL}
L_{\rho}(p,x) = \frac{\cosh(x) - \cosh \Big( \frac{4}{\pi} \cos^{-1} (\frac{|2p-1|}{\sqrt{2}}) \, x \Big)}{2x\sinh(x) \cosh(x)}.
\end{equation}
\end{corollary}

\noindent
\textbf{Proof}. We calculate $\kappa_1(p), \kappa_2(p)$ from the Fourier transform representations
(\ref{k1FT},\ref{k2FT}), as follows.
The Fourier transform of $\rho(x) = \pi^{-1} \sech(x)$ is $\hat{\rho}(k) = \sech(k\pi/2)$. 
 For $|\phi|<1$ the exponential moments are given by $ \int_\R e^{\phi x} \rho(x)dx =  \sec( \pi \phi/2)$, so that the 
 solution $\phi_p$ 
to $4 p(1-p) \int_\R e^{\phi x} \rho(x)dx = 1$ is given by $\phi_p = (2/\pi) \cos^{-1}(4p(1-p))$. 
The Fourier integral formula (\ref{Ldefn}) can be evaluated using the 
 integral
\begin{equation} \label{sechint}
I_\lambda (x) := \frac{1}{2\pi} \int_{\mathbb{R}}  e^{-ikx} \log (1-\lambda \sech(k\pi /2)) dk,
\quad \mbox{for $x \in \mathbb{R}$, $\lambda \in (0,1]$.}
\end{equation}
Note $I_\lambda (x)$ is continuous and even in $x$. Integrating by parts we find 
\[
I_\lambda (x)  = \frac{1}{2\pi i x} \int_{\mathbb{R}}  e^{- 2 ikx/\pi}  \left(\frac{\sinh (k)}{\cosh (k)}-\frac{\sinh(k)}{\cosh(k)-\lambda} \right) dk.
\]
For $x>0$ this integral can be computed by closing the integration contour in the lower half plane and
applying Cauchy's residue theorem. The only singularities of the first term of the integrand are the first order
 poles at $z^{(1)}_m=(\frac{\pi}{2}+\pi m)i$ for $ m \in \mathbb{Z}$. For $\lambda \in (0,1)$ 
 the singularities of the second term are also first order poles 
 at $z^{(2)}_m=(\alpha+2\pi m)i$ and $z^{(3)}_m=(-\alpha+2\pi m)i$ for $m \in \mathbb{Z}$, 
 where $\alpha \in (0,\pi)$ satisfies  $\cos\alpha = \lambda, \sin\alpha=\sqrt{1-\lambda^2}$. 
The corresponding residues are $e^{- 2 i x z^{(k)}_m/\pi}$.   Summing up the three 
resulting geometric progressions of residues one finds, for $x>0$,
\[
I_\lambda (x) =\frac{1}{2x}\frac{\cosh(x)-\cosh(2x(1-\frac{\alpha}{\pi}))}{\sinh(x)\cosh(x)}.
\]  
Note also that
\begin{equation} \label{silly}
1-\frac{\alpha}{\pi} = 1- \frac{\cos^{-1}(\lambda)}{\pi} =  \frac{2}{\pi} \cos^{-1} \sqrt{ \frac{1-\lambda}{2}}.
\end{equation}
The value at $\lambda=1$ can be rewritten, for $x>0$, as
\[
I_1 (x) = \frac{1}{2x}\frac{\cosh(x)-\cosh(2x)}{\sinh(x)\cosh(x)} 
= -\frac{1}{2x} \left(\tanh (x)+\tanh (x/2) \right).
\]
The value at $x=0$, by continuity, is 
\[
 I_\lambda (0) = \frac14 - \left( 1- \frac{\cos^{-1}(\lambda)}{\pi} \right)^2 
 = \frac14 - \frac{4}{\pi^2} \left(\cos^{-1} \sqrt{ \frac{1-\lambda}{2}} \right)^2.
\]
Using $L_{\rho}(p,x) = I_{4p(1-p)}(x)$, the expression (\ref{k1FT}) 
leads to the following formulae for $\kappa_1$
\[
\kappa_1(p) =   
\frac{2}{\pi^2} \left( \cos^{-1} \frac{|2p-1|}{\sqrt{2}} \right)^2    - \frac18 
+ \I(p>1/2) \frac{2}{\pi} \cos^{-1}(4p(1-p)).
\]
We can remove the indicator to reveal the analyticity of this formula; indeed 
we use $\cos^{-1}(x) = \pi-\cos^{-1}(-x)$ and 
$\cos^{-1}(x) = \frac{1}{2}\cos^{-1}(2x^2 - 1)$ to see, for $p \in (1/2,1)$, 
\begin{align*}
\kappa_1(p) 
& = \frac{2}{\pi^2}\left( \cos^{-1} \frac{2p-1}{\sqrt{2}}\right)^2 + \frac{2}{\pi} \cos^{-1}(4p(1-p)) - \frac{1}{8} \\
& =  \frac{2}{\pi^2} \left(  \cos^{-1} \frac{2p-1}{\sqrt{2}}\right)^2 + 2 - \frac{2}{\pi} \cos^{-1}(-4p(1-p))  - \frac{1}{8} \\
& =  \frac{2}{\pi^2} \left( \cos^{-1} \frac{2p-1}{\sqrt{2}}\right)^2 + 2 - \frac{4}{\pi} \cos^{-1}\left(\frac{2p-1}{\sqrt{2}}\right)  - \frac{1}{8} \\
& =  \frac{2}{\pi^2} \left(\pi - \cos^{-1}\frac{2p-1}{\sqrt{2}}\right)^2  - \frac{1}{8} = \frac{2}{\pi^2}
\left( \cos^{-1} \frac{1-2p}{\sqrt{2}}\right)^2  - \frac{1}{8}
\end{align*}
agreeing with the expression for $\kappa_1(p)$ when $p \in (0,1/2)$.
Using the exponential moments we find for $p \geq 1/2$
\[
\int_{\mathbb{R}} x e^{\phi_p x}\rho(x) dx = \frac{\pi}{2} \frac{\sin(\pi \phi_p/2)}{\cos^2(\pi \phi_p/2)}
= \frac{(2p-1) \pi \sqrt{1+4p-4p^2}}{32 p^2 (1-p)^2}
\]
and the formula for $\kappa_2$ then follows from (\ref{k2FT}). \qed
%

\noindent 
\textbf{Remark.} We do not investigate the regularity of $\kappa_2$, but the numerics in Figure \ref{fig:sech}
suggest that it is at least in $C^1$.  

\begin{figure}
  \includegraphics[width=\linewidth]{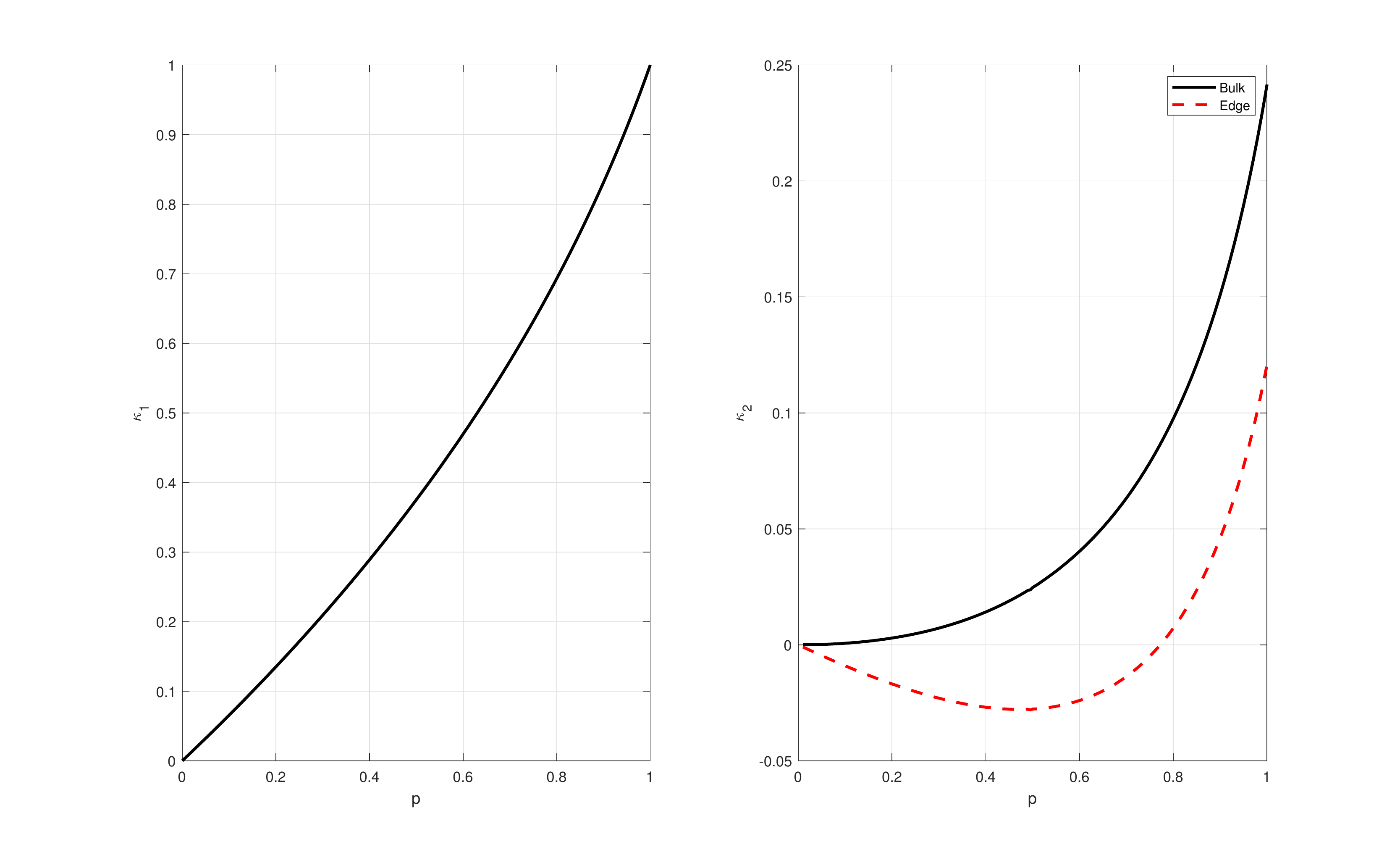}
  \caption{Left pane: the leading coefficient $p \mapsto \kappa_1(p)$ for the sech kernel from Corollary 
  \ref{sechkernel} in Section \ref{s2.2.1}.  
  Right pane: the sub-leading coefficient $p\mapsto \kappa_2(p)$ for the
same kernel (labelled `bulk') and for the non translationally invariant kernel based on
$\rho(x) = \frac{2}{\sqrt{\pi}} \exp(x - e^{2x})$ discussed in Corollary \ref{persistence} in Section \ref{s2.2.4} (labelled `edge').}
\label{fig:sech}
\end{figure}

\subsection{Gap probabilities for Coalescing/Annihilating Brownian Motions}  \label{s2.2.2}
%
This section discusses the result that arose in Derrida and Zeitak
\cite{Derrida_Zeitak} in their study of domain sizes for Potts models.
Consider an infinite system of reacting Brownian motions on $\mathbb{R}$,
 where each colliding pair instantly annihilates with probability $\theta$ or instantly coalesces with probability $1-\theta$
 (independently at each collision). We will refer to this system as CABM($\theta$). Suppose
the  initial positions form a Poisson point process with bounded intensity $\lambda(x) dx$. 
The positions of the particles at time $t>0$ form a Pfaffian point process $X_t$
with a  kernel $(1+\theta)^{-1} \mathbf{K}$ in the derived form (\ref{derived_form}) 
where
\begin{equation} \label{generalPoissonkernel}
K(x,y) = \int_{-\infty}^\infty \int_{-\infty}^{y'}  
\left( e^{-(1+\theta) \int^{y'}_{x'} \lambda(z) dz} -1 \right)  
\left|  \begin{array}{cc} 
p_t(x,x')  &  p_t(x,y') \\
p_t(y,x') &   p_t(y,y') 
\end{array}  \right| dx' dy' 
\end{equation}
 where $p_t(x,x')$ is the transition density for Brownian motion on $\mathbb{R}$.
When $\lambda$ is constant this reduces to 
\begin{equation} \label{Poissonkernel}
K(x,y) = \int^{\infty}_0 ( e^{-\lambda(1+\theta) z} -1 ) \frac{1}{\sqrt{4 \pi t}} 
\left(e^{-(z-y+x)^2/4t} - e^{-(z+y-x)^2/4t}\right) dz.
\end{equation}
This scalar $K(x,y)$ is in the translationally invariant form (\ref{TIform}) with
\begin{equation} \label{Poissonrho}
\rho(x) = 
\int_{\mathbb{R}} \frac{\lambda(1+\theta)}{2} e^{-\lambda(1+\theta) |z|} \frac{1}{\sqrt{4\pi t}}e^{-\frac{|x-z|^2}{4t}} dz, 
\end{equation}
that is the density for the convolution of a Gaussian N$(0,2t)$ variable with a two sided 
Exponential($\lambda(1+\theta)$) variable. One may also let $\lambda \uparrow \infty$, 
starting the process as an entrance law (which we informally call the maximal entrance law), 
and where $\rho$ becomes just Gaussian N$(0,2t)$ density. 

A derivation of the kernel (\ref{generalPoissonkernel}) is not quite in the literature. The 
maximal entrance law and its kernel are derived in \cite{TZ} for annihilating or coalescing Brownian motions.
Discrete analogues of CABM($\theta$) are discussed in \cite{GPTZ}, together with the kernels for
continuum limits, but for deterministic initial conditions. We go through all the (analogous) 
steps when deriving the kernel for the novel case of exit measures in Section \ref{s5}. 

Our interest here is to 
explore the gap probability asymptotics. For constant intensity Poisson($\lambda$) initial conditions, we may apply Theorem  \ref{pfaffian_czego} to deduce for $\theta >0$ that, setting 
$p_{\theta} = (1+\theta)^{-1}$, 
\[
\log \Pp[X_t(0,L) = 0] = -\kappa_1(p_{\theta}) L + \kappa_2(p_{\theta}) + o(1) \text{ as } L \rightarrow \infty 
\]
where $\kappa_1(p),\kappa_2(p)$ are given in (\ref{k1FT}) and (\ref{k2FT}) using 
$\hat{\rho}(k) = \frac{\lambda^2}{\lambda^2 + k^2} \exp(-T k^2)$. 
For the maximal entrance law, that is where $\lambda \uparrow \infty$, the underlying 
density $\rho$ is Gaussian and  the formulae
for $\kappa_1$ and $\kappa_2$ become more tractable, as shown in the upcoming corollary.

Note that as $\theta$ ranges over $(0,1]$ the value $p_{\theta}$ ranges over $[1/2,1)$. 
However the kernel $p \mathbf{K}$ for $p \in (0,\frac12)$ also has a use for the study of massive coalescing 
particles - see Lemma \ref{MCBM}. Therefore we now examine the behaviour of $\kappa_i(p)$ for all 
$p \in (0,1)$. Brownian scaling would reduce the two parameters $t,L$ in $\rho$ to 
one, but we leave both parameters so we can align our results with those in \cite{Derrida_Zeitak}.
\begin{corollary}  \label{gaussiankernel}
Let $\mathbf{K}$ be a derived form kernel, in the translationally invariant form 
(\ref{TIform})  with $\rho(x) = (4 \pi t)^{-1/2} \exp(-x^2/4t)$. Then for $p \in [0,1)$
\[
\log \Pf_{[0,L]}(\mathbf{J} - p \mathbf{K}) = -\kappa_1(p) L + \kappa_2(p) + o(1) \text{ as } L \rightarrow \infty 
\]
where $\kappa_1(p)$ is given by 
\begin{equation} \label{k1C1}
\kappa_1(p) 
= \frac{1}{4\sqrt{\pi t}} \mbox{Li}_{3/2}(4p(1-p))
+  \I(p >1/2) \left( - t^{-1} \log 4p(1-p)\right)^{1/2} 
\end{equation}
using the poly-logarithm function $\mbox{Li}_s(x) = \sum_{n\geq 1} x^n/n^{s}$, and $\kappa_2(p)$ is given by
\begin{align}
& \log \Bigg(\frac{\sqrt{1-2p}}{1-p} \Bigg) + \frac{1}{4 \pi} \sum_{n=2}^{\infty} \frac{(4p(1-p))^n}{n} 
 \sum_{k=1}^{n-1} \frac{1}{\sqrt{k(n-k)}}
&& \mbox{  for $p \in (0,\frac12)$}, \label{DZ50} \\
& \log 2 - \frac14 + \frac{1}{4 \pi}\sum_{n=2}^{\infty} \frac1n \left( \sum_{k=1}^{n-1} \frac{1}{\sqrt{k(n-k)}}
 - \pi \right) && \mbox{  for $p = 1/2$},   \label{DZ51} \\ 
& \frac12 \log \left( \frac{2p-1}{16(1-p)^2} \right) 
+
\frac{1}{4 \pi}\sum_{n=2}^{\infty} \frac{(4p(1-p))^n}{n} \sum_{k=1}^{n-1} \frac{1}{\sqrt{k(n-k)}} &&
 \nonumber  \\
 & \hspace{.2in} - \log (- \log(4p(1-p)))  -   \sum_{n=1}^{\infty} 
\frac{1}{n} \mathrm{erfc}(\sqrt{-n \log(4p(1-p))}) && \mbox{  for $p \in (\frac12,1)$}.  \label{DZ53}
\end{align}
The function $\kappa_1$ is analytic and the function $\kappa_2$ is $C^1$ for $p \in (0,1)$. 
\end{corollary}
\vspace{.1in}

\noindent
\textbf{Proof.} We use (\ref{k1FT}) to calculate $\kappa_1(p)$. We have 
$\hat{\rho}(k) =  \exp(- tk^2)$ so that 
\[
L_{\rho}(p,0) = \frac{1}{2 \pi} \int_{\mathbb{R}} \log \left(1 - 4p(1-p) e^{-t k^2} \right) dk = 
- \frac{1}{\sqrt{4\pi t}} \mbox{Li}_{3/2}(4p(1-p)).
\]
The factor $\phi_p = (- t^{-1} \log(4p(1-p)))^{1/2}$ and (\ref{k1C1}) follows from (\ref{k1FT}). 
 For $\kappa_2(p)$ we use the expressions in Theorem \ref{pfaffian_czego}, where all the 
 integrals can be evaluated using the explicit Gaussian densities $\rho^{*n}$.
 
 The regularity of $\kappa_1,\kappa_2$ is not immediately evident from these expressions, but follows 
 after some manipulation which we detail in the subsection \ref{s6.2}. \qed
 
  \vspace{.1in} 
 
 \noindent
\textbf{Remark 1.}
The formulae for $\kappa_2(p)$ are independent of $t$: 
(\ref{DZ50}) agrees with Derrida and Zeitak  \cite{Derrida_Zeitak} equation (50);
(\ref{DZ51}) agrees with \cite{Derrida_Zeitak} equation (51); 
(\ref{DZ53}) agrees with \cite{Derrida_Zeitak} equation (53). 
The formulae for $\kappa_1(p)$ depend on $t$; with the choice $t=p^2/\pi$ we find
(\ref{k1C1}) agrees with  \cite{Derrida_Zeitak} equations (49) and (52). This choice of $t$ is also 
consistent with space scaling used in \cite{Derrida_Zeitak}, as it makes the one-point 
density take the constant value $1$.
 
 \vspace{.1in} 
 
\noindent
\textbf{Remark 2.} Figure \ref{fig:gauss} plots $p \to \kappa_1(p), \kappa_2(p)$ from Corollary 
 \ref{gaussiankernel} at $t=\frac12$. As expected, $\kappa_1(p)$  increases with $p$, which corresponds 
 to weaker thinning.
Note that $\kappa_1(p) \to \infty$ as $p \uparrow 1$ (indeed $\kappa_1(p)=(-2\log(4(1-p)))^{\fr}+O(1-p)$.)
This is good sense, since at $p=1$ we 
are studying coalescing Brownian motions where gap probability have Gaussian tails not exponential tails.
Indeed gap probabilities for $p=1$ can be read off from the Brownian web in terms of a single pair of dual Brownian motions
(see Section 2 of \cite{TZ}). This simplicity corresponds in the analytic approach to the fact that the 
Fredholm Pfaffian reduces (see Proposition \ref{TWmanip}) to a $2 \times 2$ determinant.
\begin{figure}
  \includegraphics[width=\linewidth]{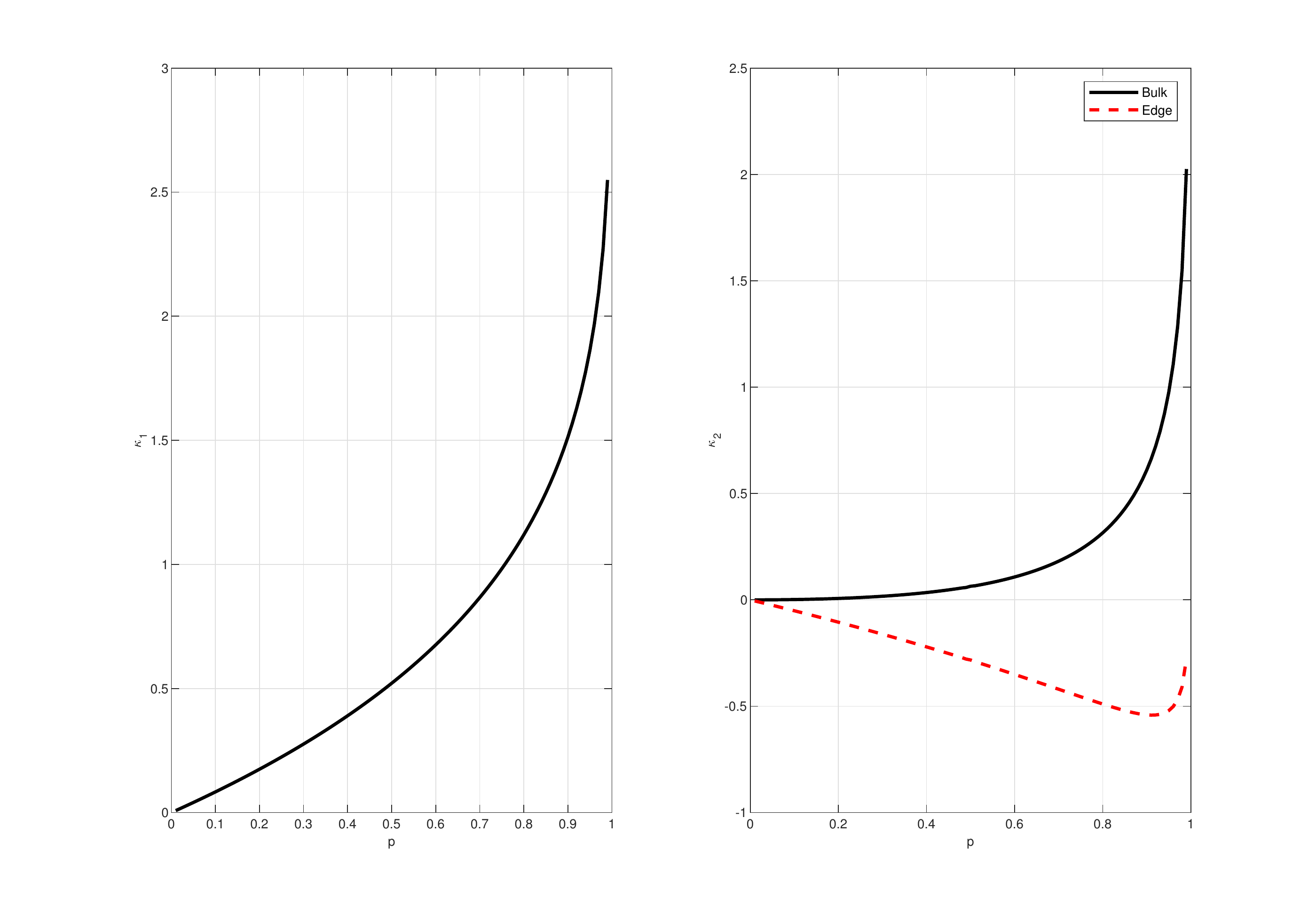}
  \caption{Left pane: the leading coefficient $p \mapsto \kappa_1(p)$ for the Gaussian kernel from Corollary 
  \ref{gaussiankernel} in Section \ref{s2.2.2} at $t=\fr$.  
  Right pane: the sub-leading coefficient $p\mapsto \kappa_2(p)$ for the
same kernel (labelled `bulk') and for the non translationally invariant kernel discussed in Section \ref{s2.2.3} (labelled `edge').}
  \label{fig:gauss}
\end{figure}
%
\subsection{Half-space initial conditions for Coalescing/Annihilating Brownian Motions} \label{s2.2.3}

Consider the same system CABM($\theta$) of reacting Brownian motions on $\mathbb{R}$ as in 
 Section \ref{s2.2.2}, but with a 
 a `maximal' entrance law on $(-\infty,0]$, defined as the limit of Poisson($\mu$) initial conditions 
 on $(-\infty,0]$ as $\mu \to \infty$. This example fits into the framework for Theorem \ref{pfaffian_szego_edge}. Indeed 
the positions of the particles at time $t>0$ form a Pfaffian point process $X_t$
with a  kernel $(1+\theta)^{-1} \mathbf{K}$ in the derived form (\ref{derived_form}). 
Taking $\lambda(z) = \lambda_0 \I(z \leq 0)$ in (\ref{generalPoissonkernel}) 
and then letting $\lambda_0 \uparrow \infty$ we find the 
underlying scalar kernel $K(x,y)$ is given by  
\begin{equation} \label{halfspacekernel}
K(x,y) = \int \int \displaylimits_{\hspace{-.2in} x'<y'}  
\I(x' \leq 0)
\left| \!\! \begin{array}{cc} 
p_t(x,x')  &  p_t(x,y') \\
p_t(y,x') &   p_t(y,y') 
\end{array} \!\! \right| dx' dy' 
\end{equation}
This scalar $K(x,y)$ is in the non translationally invariant form (\ref{NTIform}) with
$ \rho(x) = p_t(x)$.
Note that $\tilde{\rho}(z) = \int_{\mathbb{R}} \rho(w) \rho(w-z) dw = p_{2t}(z)$. 
Thus, as expected, the kernel for the half-space initial condition converges to the kernel for the full space
initial condition near $-\infty$. We therefore compare the answers given by Theorem \ref{pfaffian_szego_edge}
for the half-space maximal initial condition:
\begin{equation} \label{edge}
\log \Pp[X_t (-L,\infty) =0]  = - \kappa^{\mbox{\scriptsize{edge}}}_1 (p_{\theta}) L + \kappa^{\mbox{\scriptsize{edge}}}_2(p_{\theta}) + o(1) 
\mbox{ as $L \to - \infty$,}
\end{equation}
with those for the full space maximal initial condition given by Theorem \ref{pfaffian_czego}
in Section \ref{s2.2.2}:
\begin{equation} \label{bulk}
\log \Pp[X_t(-L,0) = 0] = -\kappa^{\mbox{\scriptsize{bulk}}}_1(p_{\theta}) L + \kappa^{\mbox{\scriptsize{bulk}}}_2(p_{\theta}) + o(1) 
\text{ as } L \rightarrow \infty. 
\end{equation}
(using the random matrix terminology for analogous problems on random spectra). 
The expression for $\kappa_1$ in Theorems \ref{pfaffian_czego} and \ref{pfaffian_szego_edge} show, as expected, 
that $\kappa^{\mbox{\scriptsize{edge}}}_1(p) = \kappa^{\mbox{\scriptsize{bulk}}}_1(p)$. The change in the O(1) constant $\kappa_2$ can be 
evaluated exactly for this Gaussian kernel and we find
\[
\kappa^{\mbox{\scriptsize{edge}}}_2(p) = \kappa^{\mbox{\scriptsize{bulk}}}_2(p) + \frac12 \log(1-p) \qquad \mbox{for all $p \in (0,1)$.}
\]
Thus the regularity properties of $\kappa_1,\kappa_2$ for $p \in (0,1)$ are unchanged
when switching from the bulk to edge case. 
Figure \ref{fig:gauss} plots $\kappa^{\mbox{\scriptsize{edge}}}_2(p)$ and $\kappa^{\mbox{\scriptsize{bulk}}}_2(p)$.
According to Corollary $6$, they are at 
least $C^1$ functions on $[0,1)$, which is consistent with the shape of the presented graphs. Near $p=1$,
$\kappa_2^{\mbox{\scriptsize{bulk}}}(p)=-\log(1-p)-\log(-\log(1-p))+O((1-p)^0)$, $\kappa_2^{\mbox{\scriptsize{edge}}}(p)=-\fr \log(1-p)-\log(-\log(1-p))+O((1-p)^0)$,
so each coefficient approaches $+\infty$ as $p\rightarrow 1-$.  

\vspace{.1in}

\noindent
\textbf{Remark 1.} As already mentioned in the introduction, the answers for $\kappa_1, \kappa_2$ for $p \in [0,1/2)$
correspond to the thinning of the real Ginibre ensemble with the thinning parameter $\gamma=2p$ investigated in
 \cite{baik2020edge}.
Under this substitution, the answer for the constant term given in Lemma 1.14 of the cited paper coincides with the answers
presented above.

\vspace{.1in}

\noindent
\textbf{Remark 2.}
For half-space initial condition it is natural to write the results in terms of the rightmost particle.
Let $R_t$ denote the position of the rightmost particle alive at time $t \geq 0$ 
so that $ \Pp[R_t \leq -L]  =  \Pp[X_t (-L,\infty) =0]$. 
The limit as $L \to \infty$ involves events where there are large numbers of annihilations by time $t$. 
The easier asymptotic probability $\Pp[R_t \geq L]$ as $L \to \infty$ involves a particle moving a large distance 
by time $t$. Indeed, using $\I(X \geq 1) = X - (X-1)_+$ and $\rho_1$, it is straightforward to see that
\[
\log \Pp[R_t \geq L] = \log \int^{\infty}_{L} \rho_1(x) dx + o(1) \quad \mbox{as $L \to \infty$.}
\]
\subsection{Non-crossing probabilities for Coalescing/Annihilating Brownian motions}  \label{s2.2.4}
%
Here, and in Section \ref{s2.3}, we study the problem by Derrida, Hakim and Pasquier
\cite{Derrida_Hakim_Pasquier} which arose in their study of persistence for Potts models, as discussed in the introduction.
We again consider the system CABM($\theta$) of reacting Brownian motions on $\mathbb{R}$ as in 
Section \ref{s2.2.3}, started from the `maximal' entrance law on $[0,\infty)$. 
We denote the position of the leftmost particle by  $(L_t: t \geq 0)$. The non-crossing probability
\[
\Pp \left[L_t > - a, \; \mbox{$ \forall t \in [0,T]$} \right] 
\]
turn out to be exactly given by a Fredholm Pfaffian. Indeed we believe the entire law of $(L_t: t \geq 0)$ should be determined
by Fredholm Pfaffians. This is explained and proved in Section \ref{s2.3} where we show 
that the particles that reach the line $x=-a$ form an exit measure point process that is Pfaffian. Its kernel fits into the 
hypotheses for Theorem \ref{pfaffian_szego_edge}  and we will deduce, for all $a>0$ and $\theta \in [0,1]$, that
 \begin{equation} \label{persistenceevent}
\log \Pp \left[L_t > - a, \; \mbox{$ \forall t \in [0,T]$} \right]  = - \frac12 \kappa_1(p_{\theta}) 
\log (2T/a^2)  
+ \kappa_2(p_{\theta}) + o(1) 
\end{equation}
as $T/a^2 \to \infty$ (again Brownian scaling shows that this probability depends only on the combination $T/a^2$). 
Here  $p_{\theta} = (1+\theta)^{-1}$ and the coefficient $\kappa_1(p), \kappa_2(p)$ are given below in 
Corollary \ref{persistence}. Using an initial condition that is  `maximal' entrance law on $(-\infty,a] \cup [a,\infty)$, the probability
that no particle crosses the origin is the square of the probability in  (\ref{persistenceevent}), since on this event the particles to the right and to the left of the origin evolve independently. This confirms the result (\ref{derrida2}) described in the 
introduction that is closest to those in \cite{Derrida_Hakim_Pasquier}.

Corollary \ref{persistence} below is a direct application
of Theorem \ref{pfaffian_szego_edge} to the non translationally invariant kernel based on 
 $\rho(x) = \frac{2}{\sqrt{\pi}} \exp(x - e^{2x})$.
Note that $\tilde{\rho}(z) = \int_{\mathbb{R}} \rho(w) \rho(w-z) dw = \frac{1}{\pi} \sech(z)$ so 
that the leading coefficient $\kappa_1(p)$ agrees with that in Corollary \ref{sechkernel} for the sech kernel. 
It is, at the moment, a coincidence that the sech kernel arises in this problem and also for Gaussian power series. 
The corollary is proved in Section \ref{s2.3}.
\begin{corollary} \label{persistence}
Let $\mathbf{K}$ be a derived form kernel, in the non translationally invariant form (\ref{NTIform}) 
based on the probability density  $\rho(x) = \frac{2}{\sqrt{\pi}} \exp(x - e^{2x})$.
Then for $p \in [0,1]$
\[
\log \Pf_{[-L,\infty)}(\mathbf{J} - p \mathbf{K}) = - \kappa_1(p) L + \kappa_2(p) + o(1) \quad \text{ as } L \rightarrow \infty 
\]
where $\kappa_1(p)$ is given by 
\begin{equation} \label{persistencekappa1}
\kappa_1(p) =   
\frac{2}{\pi^2} \left( \cos^{-1} \frac{|2p-1|}{\sqrt{2}} \right)^2    - \frac18
+ \frac{2}{\pi} \cos^{-1}(4p(1-p)) \I(p > 1/2);
\end{equation}
for $p \in [0,1/2)$
\begin{eqnarray*}
\kappa_2(p) & = & \frac12 \log \left( \frac{1-2p}{1-p} \right) +  
\frac12  \int_0^{\infty} x L_{\rho}^2(p,x) dx  
 \nonumber \\
&&  \hspace{.3in}  + \frac{1}{8 \pi } \int_{-\infty}^{\infty} \psi^{(0)}((1+ik)/2)  \log \left(1 - 4p(1-p) \sech( k \pi/2) \right) dk;
\end{eqnarray*}
\begin{eqnarray*}
\kappa_2(1/2) & = &   \frac14 \log \Big(\frac{\pi^2}{8}\Big) - \frac{\gamma}{2} 
- \frac18 \int_0^{\infty} \log (x) \left((\tanh(x) + \tanh(x/2))^2\right)^{'} dx \\
&& \hspace{.3in} + 
 \frac{1}{8 \pi } \int_{-\infty}^{\infty} \psi^{(0)}((1+ik)/2)  \log \left(1 - \sech( k \pi/2) \right) dk;
\end{eqnarray*}
and for $p \in (1/2,1]$, using $\phi_p = \frac{2}{\pi} \cos^{-1}(4p(1-p))$, 
\begin{eqnarray*}
\kappa_2(p) & = &  
\frac12  \int_0^{\infty} x L_{\rho}^2(p,x) dx  - \log \left( \cos^{-1}(4p(1-p)) \right)
 \nonumber \\
&& - \log \left( \sqrt{\frac{(2p-1)(1+4p-4p^2)}{\pi p}} \Gamma ((1+ \phi_p)/2) 
\right)  \nonumber \\
&& \hspace{.15in} + \frac{1}{\pi} \int_{\mathbb{R}} \frac{1}{1+k^2} \log \left(1- 4p(1-p) \sech( \pi \phi_p k/2) \right) dk \\
&&  \hspace{.3in}  + \frac{1}{8 \pi } \int_{-\infty}^{\infty} \psi^{(0)}((1+ik)/2)  \log \left(1 - 4p(1-p) \sech( k \pi/2) \right) dk
\end{eqnarray*}
where $L_{\rho}(p,x)$ is given in (\ref{sechL}) and $\psi^{(0)}(z)$ is the digamma function
\end{corollary}
%
%
\noindent
\textbf{Remark.} Figure \ref{fig:sech} plots $p \to \kappa_1(p), \kappa_2(p)$ from Corollary 
 \ref{persistence}.  
 When $p=1$ the exponents correspond to coalescing Brownian motions 
and take the values $\kappa_1(1) =1$ and $\kappa_2(1) = \log(2/\sqrt{\pi})$ giving, 
in (\ref{persistenceevent}) 
that
\[
 \Pp \left[L_t > - a, \; \mbox{$ \forall t \in [0,T]$} \right]  = \sqrt{\frac{2a^2}{\pi T}} (1+ o(1)).
 \]
The leftmost particle is just a Brownian motion started at $0$ and the result is then consistent  
with the exact formula found from the reflection principle.  

Figure \ref{fig:sech} also allows a comparison between the coefficients 
$\kappa^{\mbox{\scriptsize{edge}}}_2$ from Corollary \ref{persistence} 
with $\kappa_2^{\mbox{\scriptsize{bulk}}}$ from Corollary \ref{sechkernel} for the sech kernel. 
A exact computation shows that 
$\kappa_2^{\mbox{\scriptsize{bulk}}}(1)=2  \log(2/\sqrt{\pi}) = 2 \kappa^{\mbox{\scriptsize{edge}}}_2 (1)$, a relation
that requires an independent  derivation. 
%
%
\subsection{Real Eigenvalues for Real Ginibre Matrices} \label{s2.2.5}

This example is treated  in \cite{fitzgerald2020sharp} using the techniques that are generalised in this paper. Moreover
it coincides exactly with examples discussed above by considering the 
purely annihilating case ($\theta = 1$, $p_{\theta}=\frac12$) in Sections \ref{s2.2.2} and \ref{s2.2.3}.  
However, we record the results here again, as examples of both 
Theorem \ref{pfaffian_czego} and Theorem \ref{pfaffian_szego_edge} that are 
of interest to the random matrix community.

A real Ginibre ensemble matrix $M_N$ has i.i.d. real Gaussian $N(0,1)$ entries. 
Let $X_N$ be the point process created by the positions of the real eigenvalues of $M_N$. 
Then $X_N$ converges to a limit point process 
$X$ on $\mathbb{R}$ as $N \to \infty$. Also the shifted point  process  $\tilde{X}_N(dx) = X_N(N^{1/2} + dx)$ 
(that is shifted to the position of the the right hand edge of the spectrum) also converge to a 
limit $\tilde{X}$ on $\mathbb{R}$ as $N \to \infty$. Then 
 \[
\log \Pp[X(0,L) =0] =  - \frac{1}{\sqrt{8\pi}} \zeta(3/2) L + \kappa_2^{\mbox{\scriptsize{bulk}}} + o(1) \mbox{ as $L \to \infty$},
\]
and 
\[
\log \Pp[\tilde{X} (-L, \infty) =0] =- \frac{1}{\sqrt{8\pi}} \zeta(3/2) L + \kappa_2^{\mbox{\scriptsize{edge}}} + o(1) \mbox{ as $L \to \infty$,}
\]
where 
\[
 \kappa_2^{\mbox{\scriptsize{bulk}}} = \log 2 + \frac{1}{4 \pi} \sum_{n=1}^{\infty} \left( -\pi + \sum_{m=1}^{n-1} \frac{1}{\sqrt{m(n-m)}} \right) 
 =  \kappa_2^{\mbox{\scriptsize{edge}}} + \frac12 \log 2. 
\] 
The point is that the Pfaffian kernels for the bulk limit (respectively the edge limit) for the real eigenvalues 
in the Real Ginibre ensemble coincide with those for annihilating Brownian motions at time $t = \frac12$ started 
from  the maximal initial condition (respectively the half-space maximal initial condition).  
\subsection{Exit measures for particle systems} \label{s2.3} 
\subsubsection{Exit kernels} \label{s2.3.1}
To reach the applications above to persistence problems we will study exit measures for particle systems. 
We consider particle systems evolving in a region $D \subseteq \mathbb{R} \times [0,\infty)$ where whenever 
a particle hits the boundary $\partial D$ it is frozen at its exit position and plays no further role in the evolution. 
This leads to a collection of frozen particles on the boundary $\partial D$ which we call the exit measure.
Such exit measures have been used commonly in the study of branching systems, but they 
are also straightforward to construct for our coalescing and annihilating systems (first for finite systems and then 
by approximation for certain infinite systems - see the discussion in section \ref{s5.2}). We
use only the special example of the exit measure from a half-space.
\begin{theorem} \label{Exittheorem}
Let $X_e$ be the exit measure for the domain $D= (0,\infty) \times [0,\infty)$ for a system CABM($\theta$) of 
coalescing/annihilating particles as described in example \ref{s2.2.2}, started from $\mu$ a (deterministic) locally finite
simple point measure on $(0,\infty)$. Then the exit measure $X_e$ on $\{0\} \times [0,\infty)$ 
is a Pfaffian point process with kernel 
$(1+\theta)^{-1}\mathbf{K}$, where $\mathbf{K}$ is in the derived form (\ref{NTIform}) given by, when $s<t$,
\begin{equation} 
  K((0,s),(0,t))
= \int \int \displaylimits_{\hspace{-.2in} 0<y_1< y_2}  \left((-\theta)^{\mu(y_1,y_2)} -1 \right)  
\left|  \begin{array}{cc} 
p^R_s(0,y_1)  &  p^R_t(0,y_1) \\
p^R_s(0,y_2) &   p^R_t(0,y_2) 
\end{array}  \right| dy_1dy_2 \label{exitkernel}
\end{equation}
 where $p^R_t(x,y)$ is the transition density for reflected Brownian motion on $[0,\infty)$. 
 When the initial condition is 
Poisson with a bounded intensity $\lambda:(0,\infty) \to \mathbb{R}$ 
the exit measure $X_e$ remains a Pfaffian point process as above with
\begin{equation} 
  K((0,s),(0,t))
= \int \int \displaylimits_{\hspace{-.2in} 0<y_1< y_2}  \left(e^{-\int^{y_2}_{y_1} \lambda(x) dx} -1 \right)  
\left|  \begin{array}{cc} 
p^R_s(0,y_1)  &  p^R_t(0,y_1) \\
p^R_s(0,y_2) &   p^R_t(0,y_2) 
\end{array}  \right| dy_1dy_2. \label{poissonexitkernel}
\end{equation}
\end{theorem}
\noindent
\textbf{Remark 1.}
Note that kernel (\ref{poissonexitkernel}) can be obtained by averaging
the kernel (\ref{exitkernel}) for deterministic initial conditions, considering $\mu$ as Poisson. 
However this is not true for all random initial conditions and the 
Pfaffian point process structure does not hold in general. 

\vspace{.1in}

\noindent 
\textbf{Remark 2.} We believe that the Pfaffian property holds also for the exit measures
on more general regions $D$ and we explain this informally here. 
Consider the domain $D = \{(x,s): x > g(s), s \in [0,t]\}$ for some continuous $g \in C^1([0,t], \mathbb{R})$. 
We now allow particles that hit $\partial D$ to continue with constant negative drift $\mu$. 
Choosing $\mu <  - \|g'\|_{\infty}$ the particles can 
never re-enter the region $D$. This  yields a new reacting system on $\mathbb{R} \times [0,t]$ with spatially
inhomogeneous motion, where particles to the left of the graph of $g$ move deterministically 
and never re-enter $D$ nor ever again collide (see Figure \ref{willsfigure}). For coalescing/annihilating spatially inhomogeneous systems on 
$\mathbb{R}$ we believe the 
particles at a fixed time $t>0$ will form a Pfaffian point process (started from suitable initial conditions). 
Indeed in \cite{GPTZ} a class of interacting particle systems on the lattice $\mathbb{Z}$ are shown to be Pfaffian 
point processes at fixed times $t \geq 0$. These include spatially inhomogeneous coalescing and annihilating 
random walks where the right and left jump rates and the coalescence and annihilation parameters may be site dependent. 
By a continuum approximation one 
expects that the analogous systems of continuous diffusions should retain the Pfaffian property. 
The point process formed at time $t$ by the particles alive on the half-line $(-\infty,g(t)]$
can be mapped (by a deterministic bijection) onto the exit measure of the original system 
on $\partial D$, and so this exit measure should itself be a Pfaffian point process. 
The gap probabilities for this exit measure will coincide (see the discussion in the next section) 
with $\Pp[L_s > g_s, s \leq t]$, and by varying the function $g$ will determine the law of the 
leftmost particle $\{L_t: t \geq 0\}$ for a process on $\mathbb{R}$.
\pgfmathsetseed{1240}
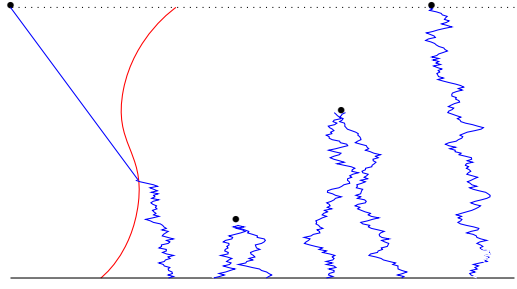
\begin{figure}
\centering
\begin{tikzpicture}[rotate = 90]
\draw [red] plot [smooth, tension=1]
coordinates { (0, 4) (1,3.5) (2.5,3.7) (3.6, 3)};
\draw(0, -1.5) -- (0, 5.2);
 \Brownian{180}{0.02}{-0.1}{blue}{}{0}{-1}
  \Brownian{20}{0.02}{-0.1}{white}{}{0}{-1}
  \Brownian{110}{0.02}{-0.1}{blue}{}{0}{0}
\Brownian{110}{0.02}{-0.1}{blue}{}{0}{1}
\Brownian{35}{0.02}{-0.1}{blue}{}{0}{2.5}
\Brownian{25}{0.02}{-0.1}{white}{}{0}{2.5}
\Brownian{35}{0.02}{-0.1}{blue}{}{0}{1.8}
\Brownian{65}{0.02}{-0.1}{blue}{}{0}{3.1}
 \node at (2.2,0.8) {\textbullet};
 \node at (0.75, 2.2){\textbullet};
  \node at (3.6, 5.2) {\textbullet};
   \node at (3.6 , -0.4) {\textbullet};
 \draw[blue] (1.3, 3.5) -- (3.6, 5.2); 
 \draw[dotted](3.6, -1.5) -- (3.6, 5.2);
 \end{tikzpicture}
 \caption{The exit measure can be transported by ballistic motion to the halfline
 $(-\infty, g(t)) \times \{t\}$.}
  \label{willsfigure}
\end{figure}
\subsubsection{Non-crossing probability} \label{s2.3.2}
In this section we prove  Corollary \ref{persistence}. 
For specific choices of initial condition the underlying scalar kernels $K((0,s),(0,t))$ in
Theorem \ref{Exittheorem} can be computed more 
precisely. They become most tractable for the entrance laws constructed as the limit of 
Poisson initial conditions with increasing intensities. The existence of these entrance laws is discussed in 
Section \ref{s5.2}. Starting with constant Poisson($\lambda$) initial conditions on $(0,\infty)$ and taking the 
limit as $\lambda \uparrow \infty$, the CABM($\theta$) starts according to a  `maximal' entrance law.
The exit measure $X_e$ remains Pfaffian with a kernel, as expected, that is the limit of the 
corresponding kernels for finite Poisson intensity (this can be checked by passing to the limit
in the duality identity (\ref{poissonduality})). Taking this limit in (\ref{poissonexitkernel}), we find
the kernel for $X_e$ under the maximal entrance law on $(0,\infty)$ has underlying scalar kernel
 \begin{eqnarray} 
K^{(\infty)}((0,s),(0,t)) & =  & - \int \int \displaylimits_{\hspace{-.2in} 0<y_1< y_2}  
\left|  \begin{array}{cc} 
p^R_s(0,y_1)  &  p^R_t(0,y_1) \\
p^R_s(0,y_2) &   p^R_t(0,y_2) 
\end{array}  \right| dy_1dy_2 \nonumber \\
& = & -  \int_0^{\infty}  \int_0^{\infty}  \sgn(y_2-y_1) p^R_s(0,y_1) p^R_t(0,y_2) dy_1dy_2 dy_1dy_2 \nonumber \\
& = & -  \frac{2}{\pi} \int_0^{\infty}  \int_0^{\infty}
\sgn(\sqrt{t} y_2- \sqrt{s} y_1) e^{-(y_1^2+y_2^2)/2} dy_1dy_2 \nonumber \\
& = & -  \frac{2}{\pi} \int_0^{\pi/2} \sgn(\sqrt{t} \sin \theta - \sqrt{s} \cos \theta)  d\theta \nonumber \\
& = & \frac{4}{\pi} \tan^{-1}\sqrt{\frac{s}{t}} -1 \qquad \mbox{for $s<t$} \label{Maximalexitkernel}
\end{eqnarray}
where in the third equality we have used
$p^R_t(0,y) = \sqrt{2/\pi t} \exp(-y^2/2t)$, and in the fourth polar coordinates.
Under the map $(0,t) \to \frac12 \log t$ the exit measure $X_e$ is pushed forward to a 
translation invariant point process on $\mathbb{R}$, and the kernel (\ref{Maximalexitkernel}) is mapped to
a translationally invariant kernel in the form (\ref{TIform}) with $\rho(z) = \pi^{-1} \sech(z)$. 
This is exactly the kernel for the zeros of the Gaussian power series in example \ref{s2.2.1}, showing 
that the zeros of the random power series agree in law, after a change of variable, with the exit measure of
annihilating Brownian motions. The asymptotics for the probability $\Pr(X_e(\{0\} \times (s,t))=0)$
(as $s \downarrow 0$ or $t \uparrow \infty$) can be read off from Corollary \ref{sechkernel}. Note, however, that 
 the exit measure $X_{e}$ gives infinite mass to any interval $\{0\} \times [0,\delta)$ if $\delta>0$.
 
 To study the persistence problem from Section \ref{s2.2.4} we choose $a>0$ and 
 start the process from Poisson($\lambda_0 \I(a,\infty)$), then let 
$\lambda_0 \uparrow \infty$, to obtained another entrance law Poisson($\infty \I(a,\infty)$), that is `maximal on $(a,\infty)$'.
Then
\begin{equation} \label{persistenceidentity}
\{X_e(\{0\} \times [0,T]) =0\} = \{L_t > 0, \; \mbox{for $t \leq T$}\}
\end{equation}
where $L_t$ denotes the position of the leftmost particle at time $t$, and the probability of this event 
agrees with the event (\ref{persistenceevent}) in Corollary \ref{persistence} (by translating by $a$). 
The parameters $a,T$ are linked and we choose $a=\sqrt{2}$ and will 
restore the final answers by Brownian scaling
\begin{eqnarray*}
&& \hspace{-.2in} \Pp[L_t > 0, \; \mbox{$ \forall t \leq T$}] \quad \mbox{under Poisson($\infty \I(a,\infty)$)} \\
& = &
\Pp[L_t > 0, \; \mbox{$\forall t \leq 2T/a^2$}] \quad \mbox{under Poisson($\infty \I(\sqrt{2},\infty)$).}
\end{eqnarray*}
Choosing $\lambda(x) = \lambda_0 \I(x > \sqrt{2})$ in (\ref{poissonexitkernel}) and letting 
$\lambda_0 \uparrow \infty$ we find the kernel 
under the entrance law Poisson($\infty \I(a,\infty)$) is 
\begin{eqnarray*}
  K((0,s),(0,t))
& = & - \int \int \displaylimits_{\hspace{-.2in} 0<y_1< y_2} \I(y_2 > \sqrt{2})
\left|  \begin{array}{cc} 
p^R_s(0,y_1)  &  p^R_t(0,y_1) \\
p^R_s(0,y_2) &   p^R_t(0,y_2) 
\end{array}  \right| dy_1dy_2 \\
& = &  \int_{\sqrt{2}}^{\infty} \int_0^{y_2}   \frac{2}{\pi \sqrt{st}} 
\left(e^{-\frac{y_2^2}{2s}-\frac{y_1^2}{2t}} - e^{-\frac{y_1^2}{2s}-\frac{y_2^2}{2t}} \right) dy_1 dy_2 \\
& = & \int_{-\infty}^0 \int_{-\infty}^{-z_2}  \frac{4 e^{z_1-z_2}}{\pi \sqrt{st}} 
\left(e^{-\frac{e^{-2z_2}}{s}-\frac{e^{2z_1}}{t}} -e^{-\frac{e^{2 z_1}}{s}-\frac{e^{-2 z_2}}{t}}  \right) dz_1 dz_2 
\end{eqnarray*}
using the substitutions $y_1 = \sqrt{2} \exp(z_1)$ and $y_2 = \sqrt{2} \exp(-z_2)$. 
Under the map $(0,t) \to -\frac12 \log t$ the exit measure $X_e$ is pushed forward to a 
point process $\tilde{X}$ on $\mathbb{R}$. The new kernel for $\tilde{X}$ is $\tilde{K}(x_1,x_2) =
K((0,e^{-2x_2}),(0,e^{-2x_1}))$ for $x_1<x_2$, which becomes  
\begin{eqnarray*}
&& \hspace{-.3in} 
\int_{-\infty}^0 \int_0^{-z_2}  \frac{4}{\pi} e^{z_1-z_2+x_1+x_2} 
\left(e^{-e^{-2(z_2-x_2)}- e^{2(z_1+x_1)}} -e^{-e^{2 (z_1+x_2)}-e^{-2 (z_2-x_1)}}  \right) dz_1 dz_2 \\
& = & \int_{-\infty}^0  \left|  \begin{array}{cc}
\int^{x_1-z}_{-\infty} \rho(w) dw & \int^{x_2-z}_{-\infty} \rho(w) dw \\
\rho(x_1-z) & \rho(x_2-z) 
\end{array} \right| dz
\end{eqnarray*}
for the probability kernel $\rho(x) = \frac{2}{\sqrt{\pi}} \exp(x - e^{2x})$.
This is in the non-translationally invariant form (\ref{NTIform}) so that we may apply  
Theorem \ref{pfaffian_szego_edge} which gives, for the initial condition Poisson($\infty \I(a,\infty)$), 
\begin{eqnarray} \label{2.2.4.1}
\log \Pp \left[L_t >0, \; \forall t \leq T \right] & = & \log \Pp \left[ \tilde{X}(-\log(2T/a^2)/2,\infty) = 0 \right] \nonumber \\
& = & - \kappa_1(p_{\theta}) \frac12 \log(2T/a^2) + \kappa_2(p_{\theta}) + o(1),
\end{eqnarray}
as $T \to \infty$, where $p_{\theta} = (1+\theta)^{-1} $ and $\kappa(p), \kappa_2(p)$ are given by Theorem \ref{pfaffian_szego_edge}
using the density $\rho(x)$. 
To evaluate $\kappa(p), \kappa_2(p)$ we first calculate  
$\tilde{\rho}(z) = \int_{\mathbb{R}} \rho(w) \rho(w-z) dw = \frac{1}{\pi} \sech(z)$.
This shows that the leading order asymptotics, that is $\kappa_1(p)$, will coincide with those for the 
translationally invariant sech kernel in Corollary \ref{sechkernel}. This immediately gives the value of 
$\kappa_1(p)$ in (\ref{persistencekappa1}) (we need to consider only $p \in [1/2,1]$ since this is the range of 
$p_{\theta}$ for $\theta \in [0,1]$). 
The terms in the formulae for $\kappa_2(p)$ in Theorem \ref{pfaffian_szego_edge}
that involve $\tilde{\rho}$ have been rewritten using Fourier transforms in  (\ref{k2FT}). 
We continue this with the terms that involve $\rho$, so 
we will use $\hat{\rho}(k) = \frac{1}{\sqrt{\pi}} \Gamma((1+ik)/2)$. Expressing $\rho^{*n}$ via the Fourier inversion formula  
and then performing the integral in $x$ we find
\begin{eqnarray*}
 && \hspace{-.3in} \frac12  \sum_{n=1}^{\infty}
  \frac{(4p(1-p))^n}{n^2} \int_{-\infty}^{\infty} \!x\, (\rho^{*n}(x))^2 dx \\
  &=&  - \frac{1}{4 \pi i}  \sum_{n=1}^{\infty}
  \frac{(4p(1-p))^n}{n^2} \int_{-\infty}^{\infty}  n (\hat{\rho}(k))^{n-1} \hat{\rho}'(k) (\hat{\rho}(-k))^n  dk \\
 & = & \frac{1}{4 \pi i} \int_{-\infty}^{\infty}  \frac{\hat{\rho}'(k) }{\hat{\rho}(k)} \log \left(1 - 4p(1-p) \sech( k \pi/2) \right) dk \\
 & = & \frac{1}{8 \pi } \int_{-\infty}^{\infty} \psi^{(0)}((1+ik)/2)  \log \left(1 - 4p(1-p) \sech( k \pi/2) \right) dk
\end{eqnarray*}
using $\hat{\rho}(k) \hat{\rho}(-k) = \hat{\tilde{\rho}}(k) = \sech(k \pi /2)$ and 
\[
\frac{\hat{\rho}'(k) }{\hat{\rho}(k)} = \frac{i}{2} \frac{\Gamma'((1+ik)/2)}{\Gamma((1+ik)/2)} = \frac{i}{2} \psi^{(0)}((1+ik)/2)
\]
where $\psi^{(0)}(z)$ is the digamma function. Finally 
\[
\int_{\mathbb{R}} e^{\phi_p x} \rho(x) dx = \hat{\rho}(-i \phi_p) = \frac{1}{\sqrt{\pi}} \Gamma ((1+\phi_p)/2)
\]
which completes all the terms contributing to $\kappa_2(p)$ for $p > \frac12$. 
\section{The proof of Theorem \ref{pfaffian_czego}}\label{s3}
In this section we will derive the asymptotic expressions
for Fredholm Pfaffians stated in the translationally invariant case. 
The proofs consists of the following
steps: (i) represent the square of the
Fredholm Pfaffian at hand as a product of a Fredholm determinant and a finite
dimensional determinant; (ii) interpret each factor as an expectation of a function
of a random walk with the transition density determined by the Pfaffian kernel;
(iii) Calculate each expectation using general theory of random walks. 
\subsection{Operator manipulation} \label{s3.1}
The first step is a calculation that was used by Tracy and Widom \cite{Tracy_Widom}  in their analysis of the Pfaffian 
kernels for GOE and GSE. It exploits the special derived form (\ref{derived_form}) of the Pfaffian kernel.
\begin{proposition} \label{TWmanip}
Let $\mathbf{K}$ be a kernel in the derived form (\ref{derived_form}) based on kernel $K \in C^2[a,b]$ 
for a finite interval $[a,b]$. We suppose that the operator $I + 2p(1-p) D_2 K$  on $L^2[a,b]$ has an inverse 
$R = (I + 2p(1-p) D_2 K)^{-1}$, for which $R-I$ itself has a $C^1$ kernel. Then 
\begin{equation} \label{TWprop}
\left( \Pf_{[a,b]} (\mathbf{J} - p \mathbf{K}) \right)^2 = \Det_{[a,b]} (I + 2p(1-p) D_2K) \; \det^{a,b}_2(K)
\end{equation}
where $\det^{a,b}_2(K)$ is the $2 \times 2$ determinant 
$\det \left( \begin{matrix} 1 + k^{(1)}(a) & k^{(1)}(b) \\ k^{(2)}(a) & 1+ k^{(2)}(b) \end{matrix}
\right)$ with entries given in terms of the functions
\begin{equation} \label{smalldetfns}
k^{(1)}  = (p-p^2) R K(\cdot,a) + p^2 R K(\cdot,b), \quad k^{(2)} = (p^2-p) R K(\cdot,b) - p^2 R K(\cdot,a). 
\end{equation}
\end{proposition}
This proposition does two things. It represents the square of the Pfaffian in terms of determinants. 
However, the main point is to exploit the derived form as follows. In the 
finite Pfaffians that define a Fredholm Pfaffian, 
there are integrals over $[a,b]$ of products of $K,D_1K,D_2K,D_{12}K$.  Each occurrence of a term $K(x_i,x_j)$ 
can be paired with a term $D_{12}K(x_j,x_k)$ and then integration by parts yields terms that only involve $D_1K$ or $D_2K$. 
Moreover $D_1K$ and $D_2K$ are related by the symmetry conditions. Repeated integration by parts leaves an 
expression that is mostly expressible only in terms of $D_2K$. This is all best done at the operator level.
Since this is a key starting point for this paper (as it was also for the study
for the specific case of the real Ginibre ensemble in \cite{Rider_Sinclair} and \cite{fitzgerald2020sharp}), and since
we will also need a modification when we treat the non translationally invariant case, we include a proof.

\noindent
\textbf{Proof.} 
The proof exploits results on determinants for trace class operators. 
We may consider a kernel $(K(x,y): x,y \in [a,b])$ as an operator on $L^2[a,b]$ via the map 
$K(f) = \int^b_a K(x,y)f(y) dy$ (we need only finite intervals).
The references \cite{GGK}, \cite{Lax} contain most of the results that we need, in particular that 
the Fredholm determinant $\Det_{[a,b]}(I+K)$  agrees with the
trace class determinant $\Det_{L^2[a,b]}(I+K)$ whenever $K: L^2[a,b] \to L^2[a,b]$
 is a trace class operator and that $K$ will be trace class
 if it is sufficiently smooth. 
 
 We need to consider operators $A \in L(H_1,H_2)$ between two 
different Hilbert spaces. In particular, an operator $A \in L(H_1,H_2)$ is called trace class if it 
satisfies $\|A\|_{tr} := \sum_n s_n < \infty$ where $(s_n)$ are the singular values of $A$, that is the 
eigenvalues of  $\sqrt{A^*A}: H_1 \to H_1$.  For $A \in  L(H_1,H_2)$ and $B \in L(H_2,H_1)$, with operator norms
$\|A\|,\|B\|$, we have
\[
\| A B \|_{tr} \leq \|A\| \|B\|_{tr} \quad \mbox{and}  \quad  \| A B \|_{tr} \leq \|A\|_{tr} \|B\|. 
\]
Thus if one of the operators $A$ or $B$ is trace class then the compositions $AB$ and $BA$ are trace class.
 Moreover the Sylvester identity
\begin{equation} \label{sylvester}
\Det_{H_2} (1 + AB) = \Det_{H_1} (1+BA)
\end{equation}
(see \cite{GGK} Chapter 4 for the case $H_1=H_2$ where $A,B$ are both trace class) also holds in the case
where $A$ is bounded and $B$ is trace class, a result which can, as in \cite{GGK}, be checked by approximating 
by finite rank operators. 

The finite interval $[a,b]$ is fixed throughout this proof. We first suppose that $K$ is smooth. 
The discontinuity $S(x,y)$ in our kernels means that the entries in $\mathbf{K}$, as in (\ref{derived_form}), 
are not trace class operators and, as in Tracy and Widom \cite{Tracy_Widom}, 
we first make a smooth approximation. We may choose smooth anti-symmetric approximations $S^{(\epsilon)}(x,y)$ that converge pointwise as $\epsilon \to 0$ to
$S(x,y)$ and are uniformly bounded by $1$.  
Then $\mathbf{K}^{\epsilon}$, defined as in (\ref{derived_form}) with $S$ is replaced by $S^{(\epsilon)}$,
can be considered a trace class operator on $(L^2_{[a,b]})^2 \to (L^2_{[a,b]})^2$ (and we do this without changing the notation). 
 
For finite dimensional matrices we have $(\Pf(J-K))^2 = \det(J-K) = \det(I+JK)$, where $J$ is block diagonal matrix made from
blocks $\tiny{ \left( \begin{matrix} 0 & 1 \\ -1 & 0 \end{matrix} \right)}$ (so that $J^2 = - I$ and $\det(J) = 1$). The analogue for us is
the relation 
\begin{equation} \label{starttw}
\left(\Pf_{[a,b]}(\mathbf{J} - p \mathbf{K}^{(\epsilon)}) \right)^2 = \Det_{(L^2_{[a,b]})^2}(I + p \mathbf{J} \mathbf{K}^{(\epsilon)}) 
\end{equation}
where the left hand side is the Fredholm Pfaffian given by the infinite series (\ref{FPdefn}) and the right hand side
is the trace class determinant on $(L^2[a,b])^2$  and $\mathbf{J}$ is the bounded operator defined by 
$\mathbf{J}(f,g) = (g,-f)$.  To derive the identity (\ref{starttw}) it is natural to argue by finite rank approximations. 
Indeed $K$ can be approximated by a polynomial $K_N(x,y) = \sum_{n,m \leq N} c_{n,m} x^n y^m$ so that
$K_N$ converges both uniformly over $[a,b]^2$ and also in trace norm as operators. For the operator 
$K_N$ the identity reduces to the finite dimensional result. 

Tracy and Widom then exploit block manipulations in the operator determinant. 
Write, in block operator notation,
\begin{eqnarray*}
\mathbf{J} \mathbf{K}^{(\epsilon)} 
&=&  \left( \begin{matrix} 0 & I \\ -I & 0 \end{matrix} \right) 
\left(  \begin{matrix} S^{(\epsilon)} + K & - D_2 K
 \\ - D_1 K &  D_{12} K  \end{matrix}\right)  \\
& = & \left( \begin{matrix} - D_1 K &  D_{12} K  \\ - S^{(\epsilon)} - K  &  D_{2} K \end{matrix} \right) \\
& = & \left( \begin{matrix} 0 & \partial  \\ -I  & I  \end{matrix} \right) 
\left( \begin{matrix}  S^{(\epsilon)} & 0  \\  -K & D_2 K \end{matrix} \right) 
\end{eqnarray*}
where $\partial : H^1_{[a,b]} \to L^2_{[a,b]}$ is the derivative operator $\partial(f) = Df$. This expresses $\mathbf{J} \mathbf{K}^{(\epsilon)}$
as the composition of two operators $\mathbf{A} \mathbf{B}$ where $\mathbf{A}: (H^1_{[a,b]})^2 \to (L^2_{[a,b]})^2$ and
$\mathbf{B}: (L^2_{[a,b]})^2 \to (H^1_{[a,b]})^2$. Moreover $A$ is bounded and 
$B$ is trace class, again by the smoothness of the kernels, and hence the 
compositions $\mathbf{A} \mathbf{B}$ and $ \mathbf{B} \mathbf{A}$ are trace class. 
Now we apply the Sylvester identity (\ref{sylvester}) to find
\begin{eqnarray}
\Det_{(L^2_{[a,b]})^2}(I + p \mathbf{J} \mathbf{K}^{(\epsilon)})  
& = & \Det_{(H^1_{[a,b]})^2}(I + p \mathbf{B} \mathbf{A}) \nonumber  \\
& = & \Det_{(H^1_{[a,b]})^2} \left( \left( \begin{matrix} I & 0 \\0  & I  \end{matrix} \right)
 + p \left( \begin{matrix} 0 & S^{(\epsilon)} \, \partial \nonumber \\
-D_2 K  & -K  \, \partial + D_2 K  \end{matrix} \right) \right)  \nonumber \\
& = & \Det_{H^1_{[a,b]}} ( I - p K \, \partial + p D_2 K + p^2 D_2 K S^{(\epsilon)} \, \partial ) \label{rhs101}
\end{eqnarray}
where the last step uses a simple manipulation for determinants of block operators.  

Now we let $\epsilon \to 0$. On the left hand side of (\ref{starttw}) the absolute convergence of the series for the
Fredholm Pfaffian allows us to check that 
$\Pf_{[a,b]}(\mathbf{J} - p \mathbf{K}^{(\epsilon)})  \to \Pf_{[a,b]}(\mathbf{J} - p \mathbf{K}) $. 
On the right hand side of (\ref{rhs101}) we rewrite various terms. We have 
\[
K \partial (f) (x) = \int^b_a K(x,y) f'(y) dy = K(x,b) f(b) - K(x,a) f(a) + \int^b_a D_2 K(x,y) f(y) dy
\]
so that $ K \partial = - D_2K + K(\cdot,b) \otimes \delta_b - K(\cdot,a) \otimes \delta_a$, as 
an operator mapping $H^1_{[a,b]} \to L^2_{[a,b]}$, where a tensor 
operator $h \otimes \delta_a$, for $h \in L^2$, acts via $ h \otimes \delta_a (f) = f(a) h$.
Similarly, again using integration by parts, $S \partial = 1 \otimes (\delta_a+\delta_b) - 2 I$. 
Then 
\begin{eqnarray*}
\| S^{(\epsilon)} \partial f + 2f - (f(a)+f(b)) \|_{L^2}^2 & = & \| \int^b_a (S^{(\epsilon)}(\cdot,y)-S(\cdot,y)) f'(y) dy \|_{L^2}^2 \\
& \leq & \|f\|_{H^1}^2 \int^b_a \int^b_a (S^{(\epsilon)}(x,y)-S(x,y))^2 dx dy
 \end{eqnarray*}
showing  the convergence $ S^{(\epsilon)} \, \partial \to 1 \otimes (\delta_a+\delta_b) - 2 I $
in operator norm from $H^1_{[a,b]} \to L^2_{[a,b]}$.
Hence the composition $ D_2 K S^{(\epsilon)} \partial $ converges in trace norm from
$H^1_{[a,b]} \to L^2_{[a,b]}$.
Using the continuity of the determinant with respect to the trace norm, the right hand side of (\ref{rhs101})  
converges and we reach
\begin{equation} \label{4010}
\left( \Pf_{[a,b]}(\mathbf{J} - p \mathbf{K})\right)^2 =  \Det_{H^1_{[a,b]}} (I + 2p(1-p) D_2K  + F)
\end{equation}
where $F: H^1_{[a,b]} \to H^1_{[a,b]}$ is the finite rank operator 
\begin{eqnarray} 
F & = & p K(\cdot,a) \otimes \delta_a - p K(\cdot,b) \otimes \delta_b + p^2 D_2K(1) \otimes (\delta_a + \delta_b) \nonumber \\
& = &  \left((p -p^2) K(\cdot,a) + p^2 K(\cdot,b)\right) \otimes \delta_a + \left((p^2-p) K(\cdot,b)-p^2 K(\cdot,a)\right) \otimes \delta_b \label{4020}
\end{eqnarray}
(using $D_2K(1)(x) = \int^b_a D_2K(x,z) dz = K(x,b) - K(x,a)$).
The assumption on the resolvent $R = (I + 2p(1-p) D_2K)^{-1}$ now allows us to split this as the product 
\begin{eqnarray*}
\Det_{H^1_{[a,b]}} ( I + 2p(1-p) D_2K  + F ) &=& \Det_{H^1_{[a,b]}} ( I + 2p(1-p) D_2K) \Det_{H^1_{[a,b]}} ( I + RF) \\
&=& \Det_{[a,b]} ( I + 2p(1-p) D_2K) \det^{a,b}_2(K)
\end{eqnarray*}
where the finite rank determinant $\Det_{H^1_{[a,b]}} ( I + RF)$ is evaluated as a $2 \times 2$ determinant 
$\det^{a,b}_2(K)$ by examining the operator $RF$ on its two-dimensional range.

Finally if $K$ is only $C^2$ we approximate by smooth anti-symmetric kernels $K_{\epsilon}$ so that the first two 
derivatives converge uniformly. If  $I + 2p(1-p) D_2 K$ is invertible and $I - R$ has a $C^1$ kernel then 
the same is true for $K_{\epsilon}$ for small $\epsilon$ and one may 
conclude  by passing to the limit in the conclusion (\ref{TWprop}) for $K_{\epsilon}$. 
\qed
%
\subsection{Probabilistic representation} \label{s3.2}
Throughout this section we suppose a kernel $\mathbf{K}$ is in derived form and has the special translationally invariant form 
(\ref{TIform}), based on a probability density $\rho$. 
We aim to apply Proposition \ref{TWmanip} to the kernel $\mathbf{K}$ on an interval $[a,b]$. 

\vspace{.1in}
\noindent
\textbf{Notation.} In this subsection only, we write $T$ for the convolution operator on $L^\infty(\R)$ 
with kernel $\rho(y-x)$ and we write $T_{a,b}$ for the convolution operator restricted to $L^2[a,b]$, that is
$T_{a,b}(f)(x) = \int^b_a \rho(y-x) f(y) dy$.

\vspace{.1in}

\noindent
Note, from (\ref{TIform}),  that 
\[
I + 2p(1-p) D_2K  = I- \beta_p T  \quad \mbox{where} \quad \beta_p := 4p(1-p).
\]
We first check the resolvent hypothesis for Proposition \ref{TWmanip}.
Since $\rho$ is a probability density we have $\gamma_0 := \sup_{x \in [a,b]} \int_{[a,b]} \rho(y-x) dy \leq 1$ 
and when $\beta_p <1$ (that is when $p \neq \frac12$) or when $\gamma_0<1$ the series
\[
\left| \sum_{n=1}^{\infty} \beta_p^n T_{a,b}^n(x,y) \right| \leq \sum_{n=1}^{\infty} \beta_p^n \gamma^{n-1}_0 
\|\rho\|_{\infty}
\]
 is uniformly convergent and hence the operator $I-\beta_p T_{a,b}$ has the inverse
$R = I +\sum_{k=1}^{\infty}  \beta_p^k T_{a,b}^{k}$.
Similarly, since $\rho \in C^1$, the series for the first derivatives of $(R - I)(x,y)$ also converge uniformly
implying that $R-I$ has a $C^1$ kernel.
In the case $p=\frac12$ we may choose $n_0 \geq 1$ so that 
\begin{equation}
\gamma_1 := \sup_{x \in [a,b]} \int_{[a,b]} T^{n_0}_{a,b}(x,y) dy < 1 \label{n0eqn}.
\end{equation}
Repeating the arguments above for 
 $\sum_{k=1}^{\infty}  T_{a,b}^{kn_0}$ we see that 
 $R = I + (I+T_{a,b} + \ldots + T_{a,b}^{n_0-1}) \sum_{k=1}^{\infty}  T_{a,b}^{kn_0}$
so that we may apply Proposition \ref{TWmanip}.

\vspace{.1in}
\noindent
\textbf{Notation}

\vspace{.1in}
Let $S=(S_n: n \geq 0)$ be a random walk, with increments distributed according to the law with density 
$\rho(x)dx$, and started at $x \in \mathbb{R}$ under the probability $\mathbb{P}_x$.  

\vspace{.1in}

We write $M_n$ for the running maximum $M_n = \max_{1 \leq k \leq n} S_k$.

\vspace{.1in}

We write $\tau_A = \inf\{ n \geq 1: S_n \in A\} $ for the positive hitting time of $A \subseteq \mathbb{R}$. 

\vspace{.1in}

We write $\tau_{a-}$ as shorthand for $\tau_{(-\infty,a]}$ and $\tau_{a+}$ as shorthand for $\tau_{[a,\infty)}$.
%
%
%
\vspace{.1in}

\noindent
We will now rewrite the Fredholm determinant and small determinant $\det^{a,b}_2(K)$
from Proposition \ref{TWmanip} as expectations for this random walk.
First we follow Kac's probabilistic representation for the Fredholm determinant from \cite{Kac} (where the 
result is established for small $\beta$ only).
\begin{lemma} \label{TIPR1}
For all $\beta \in [0,1]$ 
\begin{equation} \label{TIFDPR}
\log \Det_{[a,b]}(I - \beta T)  =  - \E_{a} \left[ \beta^{\tau_{a-}} \,\delta_{a}(S_{\tau_{a-}}) (b- M_{\tau_{a-}})_+  \right],
 \end{equation}
 where $(z)_{+}=\max(z,0)$ and $\delta_a$ stands for the Dirac delta function concentrated at $a$.
\end{lemma}

\noindent
\textbf{Proof.}
The trace-log formula (sometimes called the Plemelj-Smithies formula)
\begin{eqnarray}
&& \log \Det_{[a,b]}(I - \beta T)  = - \sum_{n=1}^{\infty} \frac{\beta^n}{n} \Tr (T_{a,b}^n) \nonumber \\
&&  \hspace{.4in}  =  - \sum_{n=1}^{\infty} \frac{\beta^n}{n} \int_{[a,b]^n} \rho(x_2-x_1) \ldots 
\rho(x_n-x_{n-1}) \rho(x_1-x_n) dx_1 \ldots dx_n \label{log-trace}
\end{eqnarray}
always holds for $|\beta|>0$ small (see \cite{GGK} Theorem 3.1). Also the Fredholm determinant $\Det_{[a,b]}(I - \beta T)$ 
is a real-analytic function of $\beta \in \mathbb{R}$. We now show the trace-log expansion
is also real-analytic for $\beta \in [0,1]$ by estimating the growth of the traces. Indeed the estimate (\ref{n0eqn})
implies that $|\Tr (T_{a,b}^{kn_0 + j})| \leq (b-a)^j \|\rho\|_{\infty}^j \gamma_1^k$ implying the series is analytic for 
$|\beta| < \gamma_1^{-1}$, so that we may apply (\ref{log-trace}) for all $\beta \in [0,1]$.

The derivative below has $n$ equal contributions:
\begin{eqnarray*}
\frac{d}{da} \Tr (T_{a,b}^n)  & = & - n \int_{[a,b]^{n-1}} \rho(x_2-a) \ldots \rho(x_n-x_{n-1}) \rho(a-x_n) dx_2 \ldots dx_{n} \\
& = & - n \, \E_a [ \delta_a(S_{n}); \tau_{b+} > \tau_{a-} =n].
\end{eqnarray*}
Subsitituting this into (\ref{log-trace}) we find 
\begin{eqnarray*}
\frac{d}{da} \log \Det_{[a,b]}(I - \beta T) & = &  \E_a[ \beta^{\tau_{a-}}\delta_a(S_{\tau_{a-}}); \tau_{b+} > \tau_{a-}] \\
& = & \E_a \left[ \beta^{\tau_{a-}}\delta_a(S_{\tau_{a-}}); M_{\tau_{a-}} <b \right].
\end{eqnarray*}
Integrating this equality over $[a,b]$ gives
\begin{eqnarray*}
 \log \Det_{[a,b]}(I - \beta T) & = & - \int^b_a  \frac{d}{dc}  \log \Det_{[c,b]}(I - \beta T) \, dc \\
 & = & - \E_{a} \left[ \beta^{\tau_{a-}} \,\delta_{a}(S_{\tau_{a-}}) (b- M_{\tau_{a-}})_+  \right].
\end{eqnarray*} \qed
%
%
%
\begin{lemma} \label{TIPR2}
When $K$  has the translationally invariant form (\ref{TIform}), based on a (symmetric) probability density $\rho$,
the factor $\det^{a,b}_2(K)$ from Proposition \ref{TWmanip} has the following probabilistic representation, 
recalling $\beta_p = 4 p(1-p)$:  when $p \neq \frac12 \; \mbox{or} \; 1$
\begin{eqnarray*}
 \det_2^{a,b}(K) 
&=&  \Bigl( 1 + \frac{2p}{2p-1} \E_a [
\beta_p^{\tau_{(a,b)^c} -1} -1]\Bigr) \\
&& \hspace{.2in} 
\left( 1 +  \frac{1}{2(1-p)} 
\left( \E_a [\beta_p^{\tau_{b+}}; \tau_{b+}< \tau_{a-}] -   \E_a [\beta_p^{\tau_{a-}}; \tau_{a-}< \tau_{b+}] \right) \right). \\
\end{eqnarray*}
Also $\det^{a,b}_2(K) = 2 \Pp_a[ \tau_{b+}< \tau_{a-}]$ when $p = \frac12$ and 
 $\det^{a,b}_2(K) = 4 \Pp_a[ \tau_{b+} =1]^2$ when $p=1$. 
\end{lemma}

\noindent
\textbf{Proof.}
We rewrite the functions $k^{(1)},k^{(2)}$ that define $\det^{a,b}_2(K)$ in terms of the kernel 
$T(x,y) = \phi(y-x)$. Using the form  (\ref{TIform}) and the symmetry of $\rho$,
\begin{eqnarray*}
K(x,a) & = &  -2 \int^{a-x}_0 \rho(z) dz \\
& = &  \int^{\infty}_a \rho(z-x) dz - \int^a_{-\infty} \rho(z-x) dz \\
& = &  T \I_{(a,b)}(x) +T \I_{[b,\infty)}(x)  - T \I_{(-\infty,a]}(x).
\end{eqnarray*}
Similarly $K(x,b) = - T \I_{(-\infty,a]}(x) - T \I_{(a,b)}(x)  + T\I_{[b,\infty)}(x)$. 
Also, for $n \geq 0$, 
\[
T_{a,b}^n T\I_{(-\infty,a]}(x) = \Pp_x[ \tau_{(a,b)^c} = n+1, S_{n+1} < a]
\]
so that, using $R = \sum_{n=0}^{\infty} \beta_p^n T_{a,b}^n$, 
\[
R T\I_{(-\infty,a]}(x)   = \E_x [ \beta_p^{ \tau_{(a,b)^c}-1}; S_{(a,b)^c} < a]
= \E_x [ \beta_p^{\tau_{a-}-1}; \tau_{a-}< \tau_{b+}].
\]
Similarly $R T \I_{[b,\infty)}(x)   = \E_x [ \beta_p^{\tau_{b+}-1}; \tau_{b+}< \tau_{a-}]$. Also
$T_{a,b}^n T \I_{[a,b]}(x) = \Pp_x[ \tau_{(a,b)^c} \geq n+2]$ so that 
\[
\mbox{$R T \I_{[a,b]}(x) = \frac{1}{\beta_p-1} \E_x [ \beta_p^{\tau_{(a,b)^c} -1}-1]$ when $p \neq \frac12$.}
\]
The symmetry of $\rho$ allows us to rewrite
\[
\E_b [ \beta_p^{\tau_{a-}}; \tau_{a-}< \tau_{b+}] =  \E_a [ \beta_p^{\tau_{b+}}; \tau_{b+} < \tau_{a-} ], \quad
\E_b [ \beta_p^{\tau_{b+}}; \tau_{b+}< \tau_{a-}] = \E_a [ \beta_p^{\tau_{a-}}; \tau_{a-} < \tau_{b+} ],
\]
as well as $ \E_a [ \beta_p^{\tau_{(a,b)^c}}] = \E_b [ \beta_p^{\tau_{(a,b)^c}}]$. 
The lemma follows, after some manipulation, by substituting the above representations
into the expressions given in Proposition \ref{TWmanip} for $k^{(1)},k^{(2)}$ and then $\det_2^{a,b}(K)$.
\qed

\subsection{Asymptotics} \label{s3.3}
We will derive the asymptotics in Theorem \ref{pfaffian_czego}. These rely on some 
classical results
about general random walks 
which we recall here. We include some derivations since we will need slight variants
in section \ref{s4} for the non translationally invariant results. The identities below hold for 
walks whose steps have a density $\rho$; we state explicitly when in addition they require 
symmetry and/or continuity of $\rho$. 
\subsubsection{Random Walk Results} \label{s3.3.1}
\paragraph*{Overshoots} 
Many of the classical results we need follow from the fact that, when the walk starts from
the  origin, the joint law of $(\tau_{0+},S_{\tau_{0+}})$ can be calculated 
in terms of the step distribution.
Indeed, supposing only that the step distribution has a density $\rho$, 
\begin{equation} \label{WH}
1- \E_0 \left[ \beta^{\tau_{0+}} e^{ik S_{\tau_{0+}}} \right] = 
\exp \big( - \sum_{n=1}^{\infty} \frac{\beta^n}{n} \int^{\infty}_0  e^{ikx} \rho^{*n}(x) dx\big) 
\quad \mbox{for $k \in \mathbb{R}, 0 \leq \beta <1$,}
\end{equation}
(see Lemma 1 of XVIII.3 from Feller \cite{feller}).  

We use various consequences of this joint law. Choosing $k=0$ and letting $\beta \uparrow 1$ yields
the entrance probability
\begin{equation} \label{entranceprob}
\Pp_0 [ \tau_{0+} = \infty] = \exp \big( - \sum_{n=1}^{\infty} \frac{1}{n}\Pp_0 [ S_n >0] \big). 
\end{equation}
When $\rho$ is in addition symmetric, one has 
\textbf{Sparre Andersen's formula} (Theorem 1 in Section XII.7 of \cite{feller})
\begin{equation}
\E_0 [\beta^{\tau_{0+}}] = 1 - \sqrt{1-\beta} \quad \mbox{for $\beta \in [0,1]$.}  \label{andersen}
\end{equation}
When $\E_0[S_1]>0$ and $\E_0[S_1^2] < \infty$ one has
\begin{equation} \label{overlapmean}
\E_0[S_{\tau_{0+}}] = \E_0[S_1]\exp \big( \sum_{n=1}^{\infty} \frac{1}{n}\Pp_0 [ S_n < 0] \big)
= \E_0[S_1]/\Pp_0[\tau_{0-} = \infty].
\end{equation}
When $\rho$ is symmetric, this is replaced by 
\textbf{Spitzer's formula} (Theorem 1  in Section XVIII.5 of \cite{feller}):
if $\sigma^2  = \E_0[S_1^2]  < \infty$ then
\begin{equation}
\E_0 [S_{\tau_{0+}}] = \frac{\sigma}{\sqrt{2}}. \label{spitzer}
\end{equation}
We give a derivation of (\ref{overlapmean}) since we do not find it in \cite{feller}. We can rewrite (\ref{WH})
as
\begin{eqnarray*}
1- \E_0 \left[ \beta^{\tau_{0+}} e^{ik S_{\tau_{0+}}} \right] &=&
\exp \big( - \sum_{n=1}^{\infty} \frac{\beta^n}{n} \int^{\infty}_{-\infty}  e^{ikx} \rho^{*n}(x) dx
+ \sum_{n=1}^{\infty} \frac{\beta^n}{n} \int^{0}_{-\infty}  e^{ikx} \rho^{*n}(x) dx \big) \\
&=&  (1- \beta \E_0[e^{ik S_1}]) \exp \big(\sum_{n=1}^{\infty} \frac{\beta^n}{n} \int^{0}_{-\infty}  e^{ikx} \rho^{*n}(x) dx \big).
\end{eqnarray*}
Differentiating in $k$ and then setting $k=0$ yields
\[
\E_0 \left[ \beta^{\tau_{0+}} S_{\tau_{0+}} \right]
= \exp \big(\sum_{n=1}^{\infty} \frac{\beta^n}{n} \Pp_0[S_n < 0] \big)
\big( \beta \E_0[S_1] - (1-\beta) \sum_{n=1}^{\infty} \frac{\beta^n}{n}
\int^0_{-\infty} x \rho^{*n}(x) dx \big).
\]
The positive mean $\E_0[S_1]>0$ and finite variance imply that
$\frac{1}{n} \int^0_{-\infty} x \rho^{*n}(x) dx \to 0$ and letting $\beta \uparrow 1$ leads to (\ref{overlapmean}).

\paragraph*{Cyclic symmetry} We use several formulae whose 
proofs exploit cyclic symmetry of the increments of the walk. 
For these, we suppose $\rho$ is both symmetric and continuous. 
The first (which can also be derived from (\ref{WH})) is 
\begin{equation}
 \E_0 [\delta_0(S_n); \tau_{0+} = n]  = \frac{1}{n} \, \E_0 [\delta_0(S_{n})].
 \label{CS1} 
 \end{equation}
We give a direct proof using cyclic symmetry since we apply this technique on other similar identities.
Let $(\mathcal{X}_0,\ldots, \mathcal{X}_{n-1})$ be the first $n$ increments of the walk, that is 
$S_k = \sum_{j=1}^k \mathcal{X}_{j-1}$. Let $S^{(p)}$, for $p=0,1,\ldots,n-1$, be the $n$-step random walk constructed from the 
same increments $(\mathcal{X}_0,\ldots, \mathcal{X}_{n-1})$ but with a cyclical permutation of the increments: 
that is $S^{(p)}_0 = 0$ and
\begin{equation*}
 S_k^{(p)} = \sum_{j=1}^k \mathcal{X}_{p \oplus (j -1)} \quad \mbox{for $1 \leq k \leq n$}
\end{equation*}
where $p \oplus (j-1)$ is addition modulo $n$. Note that  $(S^{(0)}_k)$ coincides with the original walk $(S_k)$. 
Moreover $S^{(p)}_n = S_n$ is independent of $p$. Furthermore 
\begin{equation} \label{cyclickey}
S^{(p)}_k = S_{p \oplus k} - S_{p} \quad \mbox{for all $k,p$ whenever $S_n =0$.} 
\end{equation}
Let $\tau^{(p)}_{0+} = \inf\{k \geq 1:  S_k^{(p)} > 0\}$. 
The law of each of the $(S_k^{(p)})_{0 \leq k \leq n}$ is identical so that
\begin{eqnarray*}
 \E_0 [\I(\tau_{0+} = n) \delta_0(S_n)] 
& = & 
 \frac{1}{n} \sum_{p = 0}^{n-1}  \E_0 [\I(\tau^{(p)}_{0+} = n) \delta_0(S^{(p)}_n)]  \\
 & = &   \frac{1}{n} \sum_{p = 0}^{n-1}  \E_0 [\I(\tau^{(p)}_{0+} = n) \delta_0(S^{(0)}_n)].
\end{eqnarray*}
The proof of (\ref{CS1}) is completed by noting that the sum $\sum_{p=0}^{n-1} \I(\tau^{(p)}_{0+} = n)= 1$ 
almost surely; this follows from (\ref{cyclickey}) because $\{\tau_{0+}^{(p)} = n\}$ holds if and only if 
the index $p$ is chosen such that $S_p$ is the global maximum of the random walk $(S_0, S_1, \ldots, S_{n-1})$
(this global maximum is almost surely unique since the increments have a continuous density).

Recall that $M_n := \max\{S_k: 1 \leq k \leq n\}$.  A lemma from \cite{Kac} (or
see the short proof, also based on cyclic symmetry, in the appendix of \cite{fitzgerald2020sharp}) states that 
for all $n$
 \begin{equation} \label{kac_eq_3}
\E_0 [M_{n}  \delta_0(S_{n})] =  \Kac_{\rho}(n) : = \frac{n}{2} \int_0^{\infty} x 
 \sum_{k=1}^{n-1} \frac{\rho^{*k}(x)\rho^{*(n-k)}(x)}{k(n-k)} dx
\end{equation}
where the right hand side is taken as zero for $n = 0$ or $n=1$. We call this 
\textbf{Kac}'s formula, as it was originally derived in Kac's work on Fredholm determinants
and it enters all our asymptotics. 

One final consequence of cyclic symmetry:  let $m_n := \min\{S_k: 1 \leq k \leq n\}$, then 
 \begin{equation}
 \E_0 [\min \{L, M_{n}\} \delta_0(S_n); \tau_{0-}= n] =
 \frac{1}{n} \, \E_0 [\min \{L, M_{n} - m_{n}\} \delta_0(S_{n})] \label{CS2}
\end{equation}
and therefore, letting $L \uparrow \infty$ and using the symmetry of $\rho$, 
 \begin{equation}
 \E_0 [M_{n} \delta_0(S_n); \tau_{0-}= n] =
 \frac{2}{n} \, \E_0 [M_{n} \delta_0(S_{n})]. \label{CS3}
\end{equation}
To prove (\ref{CS2}) note that 
\begin{eqnarray*}
 && \hspace{-.4in} \E_0 [\min \{L, M_{n}\} \delta_0(S_n) \I(\tau_{0-} = n)] \\
 & = & 
 \E_0 [\min\{L, M_{n}-m_n\} \delta_0(S_n) \I(\tau_{0-} = n)] \\
 & = & \frac{1}{n} 
 \sum_{p=0}^{n-1} \E_0 [\min \{L, M_{n}^{(p)}-m_n^{(p)}\} \I(\tau^{(p)}_{0-} = n)
 \delta_0(S_n) ]
\end{eqnarray*}
When $S_n = 0$ then, using (\ref{cyclickey}), 
\[
M_n^{(p)} - m_n^{(p)} = \max_{1 \leq k \leq n}  (S_{p \oplus k} - S_{p}) - \min_{1 \leq k \leq n}  (S_{p \oplus k} - S_{p})
= M_n - m_n
\]
is constant in $p$ and as above that there is exactly one value of $p$ with 
$\{\tau_{0-}^{(p)} = n\}$, establishing (\ref{CS2}). 
\subsubsection{Asymptotics for $p \in (0,\frac12)$} \label{s3.3.2}
%
Combining Proposition \ref{TWmanip} and Lemmas \ref{TIPR1} and \ref{TIPR2} for the interval $[0,L]$
we have the probabilistic representation (recall $\beta_p = 4p(1-p)$)
\begin{eqnarray}
&& \hspace{-.3in}  2 \log \Pf_{[0,L]} (\mathbf{J} - p \mathbf{K}) \nonumber \\
&  =  &  \log \Det_{[0,L]} (I + 2p(1-p) D_2K) +  \log \det^{0,L}_2(K)  \nonumber \\
& = &  -  \E_{0} [ \beta_p^{\tau_{0-}} \delta_{0}(S_{\tau_{0-}}) (L- M_{\tau_{0-}})_+ ] 
\nonumber  \\
&& \hspace{.1in} + \log \Bigl( 1 + \frac{2p}{2p-1} \E_0 [
\beta_p^{\tau_{(0,L)^c} -1} -1]\Bigr)   \nonumber \\
&&  \hspace{.1in} + \log \left( 1 +  \frac{1}{2(1-p)} \left(\E_0 [\beta_p^{\tau_{L+}}; \tau_{L+}< \tau_{0-}] -  
\E_0 [\beta_p^{\tau_{0-}}; \tau_{0-}< \tau_{L+}] \right) \right). \label{123}
\end{eqnarray}
We split the first term in (\ref{123}) using the identity $(L-M)_+ = L - \min\{L,M\}$ as follows
\begin{eqnarray}
&& \hspace{-.3in} \E_{0} [ \beta_p^{\tau_{0-}} \delta_{0}(S_{\tau_{0-}}) (L- M_{\tau_{0-}})_+ ] \nonumber \\
& = &  L \E_{0} \left[ \beta_p^{\tau_{0-}} \delta_{0}(S_{\tau_{0-}}) \right]
- \E_{0} [ \beta_p^{\tau_{0-}}  \delta_{0}(S_{\tau_{0-}}) \min\{L,M_{\tau_{0-}} \} ] \nonumber \\
& = &   L \E_{0} [ \beta_p^{\tau_{0-}} \delta_{0}(S_{\tau_{0-}}) ]
- \E_{0} [ \beta_p^{\tau_{0-}}  \delta_{0}(S_{\tau_{0-}}) M_{\tau_{0-}} ] + o(1) \label{FDp<1/2}
\end{eqnarray}
as $L \to \infty$, where the $o(1)$ asymptotic follows by monotone 
convergence, provided that $\E_{0} \left[ \beta_p^{\tau_{0-}}  \delta_{0}(S_{\tau_{0-}}) M_{\tau_{0-}} \right] $
is finite, which follows from the explicit finite formula below. 

For the second and third terms in (\ref{123}) we use Sparre Andersen's formula (\ref{andersen})
to see that $\E_0 [\beta_p^{\tau_{0-}}] = 1 - \sqrt{1-\beta_p} = 2p$ when $p < \frac12$. Thus we can write 
\[
\E_0 [\beta_p^{\tau_{0-}}; \tau_{0-} < \tau_{L+}]
= \E_0 [\beta_p^{\tau_{0-}}] + o(1) = 2p  + o(1).
\]
Similarly $\E_0 [\beta_p^{\tau_{L+}}; \tau_{L+} < \tau_{0-}]  = o(1)$ and $\E_0 [\beta_p^{\tau_{(0,L)^c}}]
= \E_0 [\beta_p^{\tau_{0-}}] + o(1)$. 

Substituting in all these asymptotics into (\ref{123}) we reach
\begin{eqnarray} 
&& \hspace{-.5in} 2 \log \Pf_{[0,L]} (\mathbf{J} - p \mathbf{K}) \nonumber \\
& = & -L \E_{0} [ \beta_p^{\tau_{0-}} \delta_{0}(S_{\tau_{0-}}) ]
+ \E_{0} [ \beta_p^{\tau_{0-}}  \delta_{0}(S_{\tau_{0-}}) M_{\tau_{0-}} ] + \log \left(\frac{1-2p}{(1-p)^2}\right) + 
o(1)  \label{probrep1}
\end{eqnarray}
which gives a probabilistic formula for the constants $\kappa_1(p),\kappa_2(p)$ 
when $p<\frac12$. Also
\begin{equation} \label{ka1form}
\E_{0} [ \beta_p^{\tau_{0-}} \delta_{0}(S_{\tau_{0-}}) ] = \sum_{n=1}^{\infty} \beta_p^n  
\E_{0} \left[ \delta_{0}(S_n); \tau_0 = n \right] =  \sum_{n=1}^{\infty} \frac{\beta_p^n}{n}  
\rho^{*n}(0)
\end{equation}
using the cyclic symmetry formula (\ref{CS1}). This  gives the form for $\kappa_1(p)$ 
stated in Theorem \ref{pfaffian_czego}. For $\kappa_2(p)$ we use 
the cyclic symmetry (\ref{CS3}) and then Kac's formula \eqref{kac_eq_3} as follows: 
\begin{eqnarray}
 \E_0 [ \beta_p^{\tau_{0-}} M_{\tau_{0-}} \delta_0(S_{\tau_{0-}})] 
 &  = & \sum_{n=1}^{\infty} \beta_p^n  \E_0 [M_{n} \delta_0(S_n); \tau_{0-}= n] \nonumber \\
 & = & \sum_{n=1}^{\infty} \frac{2\beta_p^n}{n} \E_0[M_{n} \delta_0(S_{n})] \nonumber \\
 & = & \sum_{n=1}^{\infty} \beta_p^n \int_0^{\infty} x 
 \sum_{k=1}^{n-1} \frac{\rho^{*k}(x)\rho^{*(n-k)}(x)}{k(n-k)} dx \nonumber \\
 & = &  \int_0^{\infty} x \left( \sum_{n=1}^{\infty} \frac{\beta_p^n  \rho^{*n}(x)}{n} \right)^2 dx. \label{<1/2k2}
\end{eqnarray}
This gives the form for $\kappa_2(p)$ stated in Theorem \ref{pfaffian_czego}.
The expression is finite since we assumed $\rho$ was bounded, so that
we may bound $\sup_n \|\rho^{*n}\|_{\infty} \leq \|\rho\|_{\infty} < \infty$, and that $\rho$ has  
first moment so that $\int^{\infty}_0 x \rho^{*n}(x) dx \leq Cn$ for all $n$. 
\subsubsection{Asymptotics for $p \in (\frac12,1)$} \label{s3.3.3}
The identity (\ref{123}) holds for  $p \in (\frac12,1)$, and the asymptotic for the Fredholm determinant
(\ref{FDp<1/2}) still holds. The asymptotic for the small determinant $\det_2^{0,L}(K)$ is more 
complicated, and contributes to the leading term $O(L)$. We use Sparre Andersen's formula (\ref{andersen})
to see that $\E_0 [\beta_p^{\tau_{0-}}] = 1 - \sqrt{1-\beta_p} = 2(1-p)$ when $p > \frac12$. 
This allows us to rewrite the final two terms in (\ref{123}) using
\begin{eqnarray}
&& \hspace{-.4in} 1 +  \frac{1}{2(1-p)} \left(\E_0 [\beta_p^{\tau_{L+}}; \tau_{L+}< \tau_{0-}] -  
\E_0 [\beta_p^{\tau_{0-}}; \tau_{0-}< \tau_{L+}] \right) \nonumber\\
& = &  1 +  \frac{1}{2(1-p)} \left(\E_0 [\beta_p^{\tau_{L+}}; \tau_{L+}< \tau_{0-}] -  
\E_0 [\beta_p^{\tau_{0-}}]  +
\E_0 [\beta_p^{\tau_{0-}}; \tau_{L+}< \tau_{0-}] \right) \nonumber\\
& = & \frac{1}{2(1-p)} \left(\E_0 [\beta_p^{\tau_{L+}}; \tau_{L+}< \tau_{0-}] +
\E_0 [\beta_p^{\tau_{0-}}; \tau_{L+}< \tau_{0-}] \right) \label{2003}
\end{eqnarray}
and 
\begin{eqnarray}
&& \hspace{-.4in}  1 + \frac{2p}{2p-1} \E_a [\beta_p^{\tau_{(0,L)^c} -1} -1]  \nonumber \\
 & = &   \frac{2p}{(2p-1)\beta_p} \left(\E_0 [\beta_p^{\tau_{L+}}; \tau_{L+}< \tau_{0-}] + 
\E_0 [\beta_p^{\tau_{0-}}]  -
\E_0 [\beta_p^{\tau_{0-}}; \tau_{L+}< \tau_{0-}] \right)  
- \frac{1}{2p-1}
\nonumber \\
 & = & \frac{1}{2(2p-1)(1-p)} \Bigl( \E_0 [\beta_p^{\tau_{L+}}; \tau_{L+} < \tau_{0-}]  
 - \E_0 [\beta_p^{\tau_{0-}}; \tau_{L+} < \tau_{0-}] \Bigr).  \label{2004}
\end{eqnarray}
Note that 
\begin{eqnarray}
 0 \leq \E_0 [\beta_p^{\tau_{0-}}; \tau_{L+} < \tau_{0-}] & = &  
 \E_0 \left[\beta_p^{\tau_{L+}} \E_{S_{\tau_L}}[\beta_p^{\tau_{0-}}]; \tau_{L+} < \tau_{0-} \right] \nonumber \\
 & \leq &  \E_0 [\beta_p^{\tau_{L+}}; \tau_{L+} < \tau_{0-} ] \sup_{x \geq L} \E_x [\beta_p^{\tau_{0-}}] \nonumber \\
   & = &  \E_0 [\beta_p^{\tau_{L+}}; \tau_{L+} < \tau_{0-} ] \sup_{x \leq 0} \E_x [\beta_p^{\tau_{L+}}]. \label{littleo}
\end{eqnarray}
An exact calculation shows that $\Pp_x[ \tau_{L+} =k] \to 0$ as $L \to \infty$ uniformly over $x \leq 0$. This implies that
$ \sup_{x \leq 0} \E_x [\beta_p^{\tau_{L+}}] \to 0$ and hence that we need the asymptotics only for one part of the terms
(\ref{2003}) and (\ref{2004}). Using this, (\ref{123}) can be rewritten as
\begin{eqnarray}
&& \hspace{-.3in}  2 \log \Pf_{[0,L]} (\mathbf{J} - p \mathbf{K}) \nonumber \\
& = &  -  \E_{0} [ \beta_p^{\tau_{0-}} \delta_{0}(S_{\tau_{0-}}) (L- M_{\tau_{0-}})_+ ] 
\nonumber  \\
&& \hspace{.1in} + 2 \log \E_0 [\beta_p^{\tau_{L+}}; \tau_{L+} < \tau_{0-}] - \log(4(2p-1)(1-p)^2) + o(1). \label{2005}
\end{eqnarray}
It remains only to find the asymptotics of $\E_0 [\beta_p^{\tau_{L+}}; \tau_{L+} < \tau_{0-}]$, which are as follows. 
\begin{lemma} \label{finalconstant}
Suppose  there exists $\phi_p>0$ so that $\beta_p \int e^{\phi_p x} \rho(x)dx = 1$ and that
$\int |x| e^{\phi_p x} \rho(x)dx < \infty$.
 Let $\Pp^{(p)}_x, \E^{(p)}_x$ be the tilted probability and expectation, where the random walk
$(S_n)$ has i.i.d. increments under the tilted density $\rho^{(p)}(x) = \beta_p \exp(\phi_p x) \rho(x) dx$. 
Then
\begin{equation} \label{olegf}
\lim_{L \to \infty}  e^{\phi_p L} \, \E_0 [\beta_p^{\tau_{L+}}; \tau_{L+} < \tau_{0-} ] = 
 \frac{\sqrt{1-\beta_p}}{\phi_p \E^{(p)}_0[S_1]} 
\left(\Pp^{(p)}_0 [ \tau_{0-} = \infty]\right)^2.
\end{equation}
\end{lemma}
Before the proof, we confirm that we have completed part (iii) of Theorem \ref{pfaffian_czego}.
The lemma, combined with (\ref{FDp<1/2}) and (\ref{2005}), gives 
\begin{eqnarray}
 && \hspace{-.3in} 2 \log \Pf_{[0,L]} (\mathbf{J} - p \mathbf{K}) \nonumber \\
 & = & 
 -  L  \E_{0} [ \beta_p^{\tau_{0-}} \delta_{0}(S_{\tau_{0-}}) ] - 2 \phi_pL 
 + \E_{0} [ \beta_p^{\tau_{0-}}  \delta_{0}(S_{\tau_{0-}}) M_{\tau_{0-}} ]   
  \nonumber \\ 
&&  - \log(4(2p-1)(1-p)^2)+ 
2 \log \left(  \frac{\sqrt{1-\beta_p}}{\phi_p \E^{(p)}_0[S_1]} 
\left(\Pp^{(p)}_0 [ \tau_{0-} = \infty]\right)^2 \right) \label{probrep2}
+ o(1)
\end{eqnarray}
which gives the probabilistic representation of $\kappa_1(p), \kappa_2(p)$ when $p \in (\frac12,1)$.
The expressions in the statement of Theorem  \ref{pfaffian_czego} emerge after 
using (\ref{ka1form}), (\ref{<1/2k2}) and the exact formula for
$\Pp^{(p)}_0 [ \tau_{0-} = \infty]$ given in (\ref{entranceprob}). 

\vspace{.1in}
\noindent
\textbf{Proof of Lemma \ref{finalconstant}.}
The process $X_n = \beta_p^n \exp(\phi_p S_n)$ is a martingale under $\Pp_0$ and 
$d \Pp_0^{(p)}/d \Pp_0 = X_n$ on $\mathcal{F}_n = \sigma(S_1,\ldots,S_n)$. Since
$\Pp_0[ \tau_{L+} <\infty] =1$ this extends to $d \Pp_0^{(p)}/d \Pp_0 = X_{\tau_{L+}}$ on $\mathcal{F}_{\tau_{L+}}$.
Hence
\begin{eqnarray*}
\E^{(p)}_0[ e^{-\phi_p S_{\tau_{L+}}}; \tau_{L+} < \tau_{0-}]
&=& \E_0 [ e^{-\phi_p S_{\tau_L+}} X_{\tau_{L+}}; \tau_{L+} < \tau_{0-}] \\
&=&  \E_0 [ \beta_p^{\tau_{L+}}; \tau_{L+} < \tau_{0-}]
\end{eqnarray*}
so that
\begin{eqnarray}
e^{\phi_p L} \, \E_0 [\beta_p^{\tau_{L+}}; \tau_{L+} < \tau_{0-} ] 
&=&  \E^{(p)}_0[ e^{-\phi_p (S_{\tau_{L+}}-L)}; \tau_{L+} < \tau_{0-}]. \label{p1234}
\end{eqnarray}
By conditioning on the value of $S_{\tau_{0+}}$ we see that $V(L) = \E^{(p)}_0[ \exp(-\phi_p (S_{\tau_{L+}}-L))]$, 
an exponential moment of the overlap at $L$,
satisfies the renewal equation
\[
V(L) = \int^L_0 V(L-z) G(dz) + h(L), \quad \mbox{for $h(x) = \int_x^{\infty} e^{-\phi_p(z-x)} G(dz)$}
\]
where $G(dz)$ is the law of the variable $S_{\tau_{0+}}$ on $[0,\infty)$ under $\Pp_0^{(p)}$.
By the renewal theorem (see \cite{durrett2019probability} Theorem 2.6.12 - $h$ is directly Riemann integrable)
\begin{equation} \label{renewal}
V(L) \to C_h := \frac{\int^{\infty}_0 h(x) dx}{\int_0^{\infty} x G(dx)} =  
\frac{\E^{(p)}_0[1-e^{-\phi_p S_{\tau_{0+}}}]}{\phi_p \E^{(p)}_0[S_{\tau_{0+}}]} \quad \mbox{as $L \to \infty$.}
\end{equation}
Conditioning (\ref{p1234}) on $\sigma(S_{\tau_{0-}})$ we see, as $L \to \infty$,
\begin{eqnarray*}
e^{\phi_p L} \, \E_0 [\beta_p^{\tau_{L+}}; \tau_{L+} < \tau_{0-} ] 
&=& V(L) -  \E^{(p)}_0[ e^{-\phi_p (S_{\tau_{L+}}-L)};  \tau_{0-} \leq \tau_{L+}] \\
& = & V(L) - \E^{(p)}_0[ V(L- S_{\tau_{0-}}); \tau_{0-} \leq \tau_{L+}] \\
& \to & C_h - C_h \Pp^{(p)}_0[ \tau_{0-} < \infty] = C_h \Pp^{(p)}_0[ \tau_{0-} = \infty].
\end{eqnarray*}
This shows the existence of the desired limit, and it remains to re-write this limit in an easier
form. The identity (\ref{WH}) holds for all $k \in \mathbb{C}$ with $\Im(k)>0$, since both sides are analytic there. 
Using this identity for the density $\rho^{(p)}$, 
choosing $k = i \phi_p$ and letting $\beta \uparrow 1$, gives
\begin{eqnarray}
1-\E^{(p)}_0 \left[ e^{- \phi_p S_{\tau_{0+}}} \right] 
&=& \exp \big( - \sum_{n=1}^{\infty} \frac{1}{n} \int^{\infty}_0  e^{- \phi_p x} (\rho^{(p)})^{*n}(x) dx\big) \nonumber \\
& = & \exp \big( - \frac12 \sum_{n=1}^{\infty} \frac{1}{n} \int^{\infty}_{-\infty}  e^{- \phi_p x} (\rho^{(p)})^{*n}(x) dx\big) \nonumber \\
& = &  \exp \big( - \frac12 \sum_{n=1}^{\infty} \frac{1}{n} \E^{(p)}_0 [ e^{-\phi_p S_n}] \big) \nonumber \\
& = & \exp( \frac12 \log(1-\E^{(p)}_0 [ e^{-\phi_p S_1}]  \big) \nonumber \\
& = & \sqrt{1-\beta_p} \label{9078}
\end{eqnarray}
where the second equality follows from an explicit calculation using $\rho^{(p)}(x) = \beta_p \exp(\phi_p x) \rho(x) dx$ 
and the symmetry of $\rho$ that shows
\[
\int^{\infty}_0  e^{- \phi_p x} (\rho^{(p)})^{*n}(x) dx = \int_{-\infty}^0  e^{- \phi_p x} (\rho^{(p)})^{*n}(x) dx.
\]
Together (\ref{9078}), (\ref{overlapmean}) and (\ref{renewal}) lead to the desired form for the limit. 
\qed
\subsubsection{Asymptotics for $p=\frac12$} \label{s3.3.4}
When  $p=\frac12$ we find, again by applying Propositions \ref{TWmanip}, \ref{TIPR1} and 
Lemma \ref{TIPR2} for the interval $[0,L]$ (and noting that the probabilistic representation for the small determinant is different),
\begin{equation}
 2 \log \Pf_{[0,L]} (\mathbf{J} - \frac12 \mathbf{K}) 
= -  \E_{0} \left[ \delta_{0}(S_{\tau_{0-}}) (L- M_{\tau_{0-}})_+  \right] 
 +  \log 2 + \log \Pp_0[ \tau_{L+} < \tau_{0-}]. 
\label{E1}
\end{equation}
The optional stopping theorem is applicable to the martingale $(S_n)_{n \geq 0}$ and the stopping time
$\tau_{0-} \wedge \tau_{L+}$ (see \cite{lawler} Lemma 5.1.3) giving 
\begin{eqnarray*}
0 &=&  \E_0[S_{\tau_{L+} \wedge \tau_{0-}}]  \\
& = & \E_0[S_{\tau_{0-}}; \tau_{0-} < \tau_L] + \E_0[S_{\tau_{L+}}; \tau_{L+} < \tau_{0-}]  \\
& = & \E_0[S_{\tau_{0-}}] - \E_0[S_{\tau_{0-}}; \tau_{L+} < \tau_{0-}]  + \E_0[(S_{\tau_{L+}}-L); \tau_{L+} < \tau_{0-}] 
+ L \Pp_0[\tau_{L+} < \tau_{0-}]. 
\end{eqnarray*}
so that
\begin{equation} \label{gambler}
 \Pp_0[\tau_{L+} < \tau_{0-}] =  \frac{-\E_0[S_{\tau_{0-}}]}{L}  + \frac{\E_0[S_{\tau_{0-}} - (S_{\tau_{L+}}-L); \tau_{L+} < \tau_{0-}]}{L}
 \leq   \frac{-\E_0[S_{\tau_{0-}}]}{L}.
\end{equation}
The overshoots $S_{\tau_{0-}}$ and $S_{\tau_{L+}}-L$ have in general less moments than the underlying step distribution.
However, Lemma 5.1.10  from \cite{lawler} shows that when $\int x^4 \rho(x) dx <\infty$ then the overshoots have finite second moment
(bounded independently of the starting point). Then, for example, 
\[
\E_0[(S_{\tau_{L+}}-L); \tau_{L+} < \tau_{0-}] \leq (\E_0[(S_{\tau_{L+}}-L)^2] )^{1/2} (\Pp_0[ \tau_{L+} < \tau_{0-}])^{1/2}, 
\]
and  we deduce from (\ref{gambler}) that $ \Pp_0[\tau_{L+} < \tau_{0-}] =  -\E_0[S_{\tau_{0-}}]/L + O(L^{-3/2})$.
Hence, using Spitzer's formula (\ref{spitzer}) for  $\E_0[S_{\tau_{0-}}]$,
\begin{equation} \label{E2}
\log \Pp_0[\tau_{L+} < \tau_{0-}] = - \log L + \log \sigma - \frac12 \log 2+ o(1).
\end{equation} 
As before we have
\begin{equation} \label{E3}
\E_{0} \left[ \delta_{0}(S_{\tau_{0-}}) (L- M_{\tau_{0-}})_+  \right] 
= L \E_{0} \left[ \delta_{0}(S_{\tau_{0-}}) \right]
- \E_{0} \left[ \delta_{0}(S_{\tau_{0-}}) \min\{L,M_{\tau_{0-}} \} \right] 
\end{equation}
which, as in the case $p<\frac12$, gives the value of the constant $\kappa_1(1/2)$ in the 
leading order $O(L)$ asymptotic. The following lemma 
shows that $\E_{0} \left[ \delta_{0}(S_{\tau_{0-}}) \min\{L,M_{\tau_{0-}} \} \right]$
is of form $\log L + C_0 + o(1)$ leading to the cancelation of the $-\log L$ term in (\ref{E2}).
This lemma, together with (\ref{E1}), (\ref{E2}), (\ref{E3}) completes the $p=\frac12$ case
of Theorem \ref{pfaffian_czego}.
\begin{lemma} \label{12lemma}
\begin{eqnarray*}
&& \hspace{-.4in} \E_{0} \left[ \delta_{0}(S_{\tau_{0-}}) \min\{L,M_{\tau_{0-}} \} \right] \\
& = &  \log L + \frac{3}{2} \log 2 - \log \sigma +
\sum_{n=1}^{\infty} \left( \sum_{k=1}^{n-1}
\int_0^{\infty} x \frac{\rho^{*k}(x) 
\rho^{*(n-k)}(x)}{k(n-k)} dx - \frac{1}{2n} \right) + o(1).
\end{eqnarray*}
\end{lemma}

\noindent
\textbf{Proof.}
We follow the strategy used in \cite{fitzgerald2020sharp} which considers a walk with Gaussian increments;
in this general case the asymptotics differ only in the part of 
the constant term arising from Kac's formula (\ref{kac_eq_3}).
Write 
\begin{equation}
 \E_{0} \left[ \delta_{0}(S_{\tau_{0-}}) \min\{L,M_{\tau_{0-}} \} \right] = \sum_{n=1}^{\infty} p(n,L)
 \label{1005}
\end{equation}
where, using the cyclic symmetry technique (\ref{CS2}), 
\begin{eqnarray*}
p(n,L) & = &  \E_{0} \left[ \delta_{0}(S_n) \min\{L,M_n\} \I(\tau_{0-} = n) \right] \\
& = & \frac1n \E_{0} \left[ \delta_{0}(S_n) \min\{L,M_n - m_n\}  \right].
\end{eqnarray*}
While $n \leq L^{2 - \epsilon}$ (we will soon choose $\epsilon \in (\frac12,2)$) the walk is unlikely to have reached $L$ and we will approximate 
$\min\{L,M_n - m_n\} \approx M_n - m_n$. For $n \geq L^{2 - \epsilon}$ we will
use a Brownian approximation, using a Brownian motion $(W(t): t \geq 0)$ run at speed $\sigma^2$
(that is $[W](t) = \sigma^2t$) and the 
running extrema $W^*(t) = \sup_{s \leq t} W(s)$ and $W_*(t) = \inf_{s \leq t} W(s)$.
These approximations lead to 
\begin{eqnarray}
p(n,L) & = & \frac1n \E_{0} \left[ \delta_{0}(S_n) (M_n - m_n)  \right] + E^{(1)}(n,L), \label{P1} \\
p(n,L) & = & \frac1n \E_{0} \left[ \delta_{0}(W(n)) \min\{L,W^*(n) - W_*(n)\}  
\right] + E^{(2)}(n,L), \label{P2}
\end{eqnarray}
where, for some $\eta>0, \,C< \infty$
\begin{eqnarray}
&& E^{(1)}(n,L) \leq  C n L^{-3} \quad \mbox{for $n \leq L^{2-\epsilon}$,} \label{E10} \\
&& E^{(2)}(n,L) \leq C n^{-1-\eta} \quad \mbox{for $n \geq L^{2-\epsilon}$.} \label{E20}
\end{eqnarray}
We delay the detailed proof for the error bounds (\ref{E10}),(\ref{E20}) to 
the subsection \ref{s6.3}.

Kac's formula (\ref{kac_eq_3}) and the symmetry of $\rho$ give 
\begin{equation} \label{1010}
 \frac1n \E_{0} \left[ \delta_{0}(S_n) (M_n - m_n)  \right]  = 
 \frac2n \E_{0} \left[ \delta_{0}(S_n) M_n \right] = \frac2n \Kac_{\rho}(n).
\end{equation}
The asymptotic
\begin{eqnarray*}
\E_{0} \left[ \delta_{0}(S_n) M_n \right] & = & \frac{1}{2 \pi} \int_{\mathbb{R}} \E_{0} \left[  e^{i \theta S_n} M_n \right] d \theta \\
& = & \frac{1}{2 \pi} \int_{\mathbb{R}} \E_{0} \left[  e^{i \theta S_n/\sqrt{n}} \frac{M_n}{\sqrt{n}} \right] d \theta \\
& \to & \E_{0} \left[ \delta_0(W(1)) W^*(1) \right] = \frac14
\end{eqnarray*} 
(using an explicit calculation with the joint density for $(W^*(t),W(t))$) can be quantified, 
using the local central limit theorem - see the details in the subsection \ref{s6.3}, to give
\begin{equation} \label{1011}
\Kac_{\rho}(n) = \frac14 + O(n^{-1/4}).
\end{equation}
Thus the series $\sum \frac2n \Kac_{\rho}(n)$ is divergent and we will need to compensate the 
terms to gain a convergent series. 
This completes all parts of this lemma that are different from the Gaussian case in 
\cite{fitzgerald2020sharp}, and we now refer the reader to that paper for some of the subsequent calculations.

The Brownian expectation (\ref{P2}) can be calculated (see section 3.2.2 of \cite{fitzgerald2020sharp}) using the joint distribution 
of $(W(t), W^*(t),W_*(t))$, yielding 
\begin{equation} \label{1020}
\E_{0} \left[ \delta_{0}(W(n)) \min\{L,W^*(n) - W_*(n)\}  
\right] 
= \frac{1}{2} - \sqrt{\frac{2L^2}{\pi \sigma^2 n}} \Omega \left( \frac{2L^2}{\pi n \sigma^2} \right)
\end{equation} 
where $\Omega(t) = \sum_{k\geq 1} \exp(- \pi k^2 t)$ is a special function 
(related to Jacobi's $\theta$-function). 
From (\ref{P1}), (\ref{P2}) we find, as in \cite{fitzgerald2020sharp}, the asymptotic for $L \to \infty$:
\begin{equation} \label{1056}
\sum_{n=1}^{\infty} p(n,L) = \sum_{n \leq L^{2-\epsilon}}  \frac2n \Kac_{\rho}(n) 
+ \sum_{n > L^{2-\epsilon}} \frac{1}{2n} - \sqrt{\frac{2L^2}{\pi \sigma^2 n^3}} 
\Omega \left( \frac{2L^2}{\pi n \sigma^2} \right) + o(1).
\end{equation}
The error terms $E^{(i)}(n,L)$ contribute only to the $o(1)$ term, but for this we 
must choose $\epsilon \in (\frac12,2)$ (in \cite{fitzgerald2020sharp} the error term $E^{(1)}(n,L)$ was exponentially small
as we dealt with a walk with Gaussian increments - here we suppose only fourth moments). 

We compensate, using (\ref{1011}), 
\begin{eqnarray*}
\sum_{n \leq L^{2-\epsilon}}  \frac2n \Kac_{\rho}(n) &=& \sum_{n \leq L^{2-\epsilon}}  \frac2n \left(\Kac_{\rho}(n) - \frac14\right) 
+ \sum_{n \leq L^{2-\epsilon}} \frac{1}{2n} \\
& = & \sum_{n \geq 1}  \frac2n\left(\Kac_{\rho}(n) - \frac14\right)  +  \sum_{n \leq L^{2-\epsilon}} \frac{1}{2n} + o(1) \\
& = & \sum_{n \geq 1}  \frac2n \left(\Kac_{\rho}(n) - \frac14\right)   + \frac12 \log L^{2-\epsilon} + \frac{\gamma}{2} + o(1)
\end{eqnarray*}
using the asymptotic $\sum_{n \leq N} \frac1n = \log N + \gamma + O(N^{-1})$ where $\gamma$ is the 
Euler-Mascheroni constant. An analysis of the error in a Riemann block approximation implies 
\begin{eqnarray*} 
&& \hspace{-.4in} 
 \sum_{n > L^{2-\epsilon}} \frac{1}{2n} - \sqrt{\frac{2L^2}{\pi \sigma^2 n^3}} 
\Omega \left( \frac{2L^2}{\pi n \sigma^2} \right) \\
& = & \int_{L^{-\epsilon}}^{\infty} \left(\frac{1}{2x} - \sqrt{\frac{2L^2}{\pi \sigma^2 x^3}} 
\Omega \left( \frac{2L^2}{\pi x \sigma^2} \right)\right) dx + o(1) \\
& = &  \int_{1}^{\infty} \left(\frac{1}{2x} - \sqrt{\frac{2L^2}{\pi \sigma^2 x^3}} 
\Omega \left( \frac{2L^2}{\pi x \sigma^2} \right)\right) dx 
- \int_{0}^{1} \sqrt{\frac{2L^2}{\pi \sigma^2 x^3}} 
\Omega \left( \frac{2L^2}{\pi x \sigma^2} \right) dx + \frac{\epsilon}{2} \log L + o(1) 
\end{eqnarray*}
The asymptotics $\Omega(t) = \frac{1}{2 \sqrt{t}} - \frac12 + o(1)$ as $t \downarrow 0$ 
and $\Omega(t) \sim \exp(-\pi t)$ as $t \to \infty$  justify the convergence of these integrals. 
Substituting these into (\ref{1056}) we find the dependence on $\epsilon$ vanishes and we reach
\[
\sum_{n=1}^{\infty} p(n,L) = \log L + \sum_{n \geq 1}  \frac2n \left(\Kac_{\rho}(n) - \frac14\right)  + C_0 + o(1)
\]
where 
\[
C_0 = \frac{\gamma}{2} 
+ \int_{1}^{\infty} \left(\frac{1}{2x} - \sqrt{\frac{2L^2}{\pi \sigma^2 x^3}} 
\Omega \left( \frac{2L^2}{\pi x \sigma^2} \right)\right) dx 
- \int_{0}^{1} \sqrt{\frac{2L^2}{\pi \sigma^2 x^3}} 
\Omega \left( \frac{2L^2}{\pi x \sigma^2} \right) dx.
\]
Amazingly, certain identities for $\gamma$ and the function $\Omega(t)$ imply that $C_0 = \frac{3}{2} \log 2 - \log \sigma$ 
(see Section 2 of \cite{fitzgerald2020sharp}) and this completes the proof. \qed
\section{The proof of Theorem \ref{pfaffian_szego_edge}}   \label{s4}
Throughout this section we suppose we have a kernel $\mathbf{K}$ in the derived form (\ref{derived_form})
based on a scalar kernel in the form (\ref{NTIform}), that is 
\begin{equation} \label{repeat}
K(x,y) = \int_{-\infty}^0  \left|  \begin{array}{cc}
\int^{x-z}_{-\infty} \rho(w) dw & \int^{y-z}_{-\infty} \rho(w) dw \\
\rho(x-z) & \rho(y-z) 
\end{array} \right| dz
\end{equation}
for a probability density $\rho \in C^1(\mathbb{R}) \cap  H^1(\mathbb{R}) \cap L^{\infty}(\mathbb{R})$.
This is sufficient to allow 
differentiation under the integral and integration by parts showing that $K,D_1K,D_2K,D_{12}K$ are all
continuous and in particular
\[
D_2K(x,y) = -2 \int^0_{-\infty} \rho(x-z) \rho(y-z) dz - \rho(y) \int^x_{-\infty} \rho(w) dw.
\]
We now write $T$ for the integral operator on $L^{\infty}(\mathbb{R})$ with kernel 
\begin{equation} \label{recallT}
T(x,y) = \int^0_{-\infty} \rho(x-z) \rho(y-z) dz.
\end{equation}
Then, recalling $\beta_p=4p(1-p)$,  
\begin{equation} \label{3010}
2p(1-p) D_2 K(x,y) = - \beta_p T(x,y) - 2p(1-p)  \rho(y) \int^x_{-\infty} \rho(w) dw.
\end{equation}
Note the last term in (\ref{3010}) is a finite rank kernel. The operator $T$ still can be interpreted in probabilistic terms
using a random two-step walk, where pairs of increments have the density $\rho(-x)$ and $\rho(x)$. 
To our surprise, it is possible to follow fairly closely the strategy used in the translationally invariant case, namely 
(i) represent the Fredholm Pfaffian in terms of determinants; (ii) represent these in terms of the random two-step walk; 
(iii) derive asymptotics from probabilistic results for a two-step walk. 
Each of these steps requires slight modifications
(and becomes slightly messy) due to the different operator $T$ and due to the extra finite rank 
term above. (The thesis \cite{FWR} contains an exploration of more general 
kernels where $T(x,y) = \int^0_{-\infty} \rho^{(1)}(x-z) \rho^{(2)}(y-z) dz$ for two probability densities $\rho^{(1)},\rho^{(2)}$, 
and shows that many of the steps above go through; however we have yet to find applications). 
\subsection{Operator manipulation}  \label{s4.1}
We again write $T_{a,b}$ for the integral operator restricted to $L^2[a,b]$, that is
$T_{a,b}f(x) = \int_a^b T(x,y) f(y) dy$. Our first aim is an 
analogue of the Tracy Widom manipulations in Proposition \ref{TWmanip} for this non
translationally invariant  setting. The reasoning at the start of Section \ref{s3.2} extends 
to this case to show that $1-\beta_p T_{a,b}$ has the inverse
$R = I +\sum_{k=1}^{\infty}  \beta_p^k T_{a,b}^{k}$ and that 
$R-I$ has a $C^1$ kernel.
\begin{lemma} \label{TWmanip2}
With the above notation we have
\begin{equation} \label{TWprop2}
\left( \Pf_{[a,b]} (\mathbf{J} - p \mathbf{K}) \right)^2 = \Det_{[a,b]} (I - \beta_p T) \; \det^{a,b}_3(K)
\end{equation}
where $\det^{a,b}_3(K)$ is the $3 \times 3$ determinant with entries $\delta_{ij} + (Re_i,f_j)$, where
$(e,f)$ is the dual pairing between $H^1_{[a,b]}$ and its dual, for the elements
\begin{eqnarray*}
& f_1 = \beta_p \rho, \quad  &  e_1 = - \frac12 \Phi_{\rho}, \\
 & f_2 = \delta_a - \delta_b, \quad &  e_2 = p T \I_{[b,\infty)} - p T \I_{(-\infty,a]} 
+ \frac{p}{2}(2-\Phi_{\rho}(b)-\Phi_{\rho}(a)) \Phi_{\rho}, \\
&f_3 = \delta_a + \delta_b, \quad & e_3 = -p(2p-1) T \I_{(a,b)} - \frac{p(2p-1)}{2} 
(\Phi_{\rho}(b)-\Phi_{\rho}(a)) \Phi_{\rho}, 
\end{eqnarray*}
where $\Phi_{\rho}(x) = \int_{-\infty}^x \rho(x)$ is the distribution function for $\rho$.
\end{lemma}

\noindent
\textbf{Proof.} We follow the proof of Proposition  \ref{TWmanip} up to (\ref{4010}), (\ref{4020}). Then using (\ref{3010})
we have
\[
\left( \Pf_{[a,b]}(\mathbf{J} - p \mathbf{K})\right)^2 =  \Det_{H^1_{[a,b]}} (I - \beta_p T  + F)
=  \Det_{H^1_{[a,b]}} (I - \beta_p T)  \Det_{H^1_{[a,b]}} (I  + RF)
\]
where $F$ is the finite rank operator
\begin{eqnarray*}
&& \hspace{-.3in} - \frac{\beta_p}{2} \Phi_{\rho} \otimes \rho + 
\left((p -p^2) K(\cdot,a) + p^2 K(\cdot,b)\right) \otimes \delta_a + \left((p^2-p) K(\cdot,b)-p^2 K(\cdot,a)\right) 
\otimes \delta_b\\
& = & - \frac{\beta_p}{2} \Phi_{\rho} \otimes \rho + 
\frac{p}{2} 
\left(K(\cdot,b) + K(\cdot,a)\right) \otimes (\delta_a-\delta_b) + \frac{p(2p-1)}{2} \left(K(\cdot,b)-K(\cdot,a)\right) \otimes 
(\delta_a+\delta_b).
\end{eqnarray*}
This gives the rank $3$ form for $F$ and the values of $f_1,f_2,f_3,e_1$ as stated. 
Again using (\ref{3010}) we have
\[
K(\cdot,b)-K(\cdot,a) = \int^b_a D_2 K(\cdot,z) dz =
- 2 T \I_{(a,b)} - \int^b_a \rho(w) dw \, \Phi_{\rho}
\]
giving the value of $e_3$. The limits $
K(x,-\infty) = 1$ and $K(x,\infty) = \Phi_{\rho} (x)-1 $
follow from the definition (\ref{repeat}). Then
\begin{eqnarray*}
&& \hspace{-.4in} K(\cdot,b) + K(\cdot,a) \\
& = & \left( K(\cdot,a) - K(\cdot,-\infty) \right) - \left( K(\cdot,\infty) - K(\cdot,b) \right)
+ K(\cdot, \infty) + K(\cdot,-\infty) \\
& = & -2 T \I_{(-\infty,a]} - \int_{-\infty}^a \rho(w) dw \, \Phi_{\rho}
+2 T \I_{[b,\infty)} + \int_b^{\infty} \rho(w) dw \, \Phi_{\rho} + \Phi_{\rho}
\end{eqnarray*}
which gives the value of $e_2$. \qed
\subsection{Probabilistic representation}  \label{s4.2}
We need two different two-step random walks. These have independent increments, but alternate between
a step distributed as $\rho(-x)dx$ and then as $\rho(x)dx$, as follows. 

\vspace{.1in}
\noindent
\textbf{Notation}

\vspace{.1in}

\noindent
Let
$(\mathcal{X}_k)$ and $(\mathcal{Y}_k)$ be two independent families of i.i.d. variables where 
$\mathcal{X}_k$ have density $\rho(-x)dx$ and $\mathcal{Y}_k$ have density $\rho(x)dx$. 

\vspace{.1in}

\noindent
Under $\Pp_x$ the variables $(S_n: n \geq 0)$ and $(\tilde{S}_n: n \geq 1)$ are defined by $S_0=x$ and
\[
\tilde{S}_k = S_{k-1} + \mathcal{X}_k, \quad S_k = \tilde{S}_k + \mathcal{Y}_k \quad \mbox{for $k = 1,2, \ldots$}.
\]
\noindent
We write $\tau_A = \inf\{n \geq 1: S_n \in A\}, \;  \tilde{\tau}_A = \inf\{n \geq 1: \tilde{S}_n \in A\}$ and the 
special cases 
\begin{eqnarray*}
&\tau_{a+} = \inf\{n \geq 1: S_n > a\}, \quad \tau_{a-} = \inf\{n \geq 1: S_n <a\}, & \\
&\tilde{\tau}_{a+} = \inf\{n \geq 1: \tilde{S}_n > a\}, \quad \tilde{\tau}_{a-} = \inf\{n \geq 1: \tilde{S}_n <a\}, &
\end{eqnarray*}
and the running maxima and minima
\begin{eqnarray*}
&M_n = \max\{S_k: 1 \leq k \leq n\}, \quad m_n = \min\{S_k: 1 \leq k \leq n\}  & \\
&\tilde{M}_n = \max\{\tilde{S}_k: 1 \leq k \leq n\} \quad \tilde{m}_n = \min\{\tilde{S}_k: 1 \leq k \leq n\}. &
\end{eqnarray*}

\vspace{.05in} 
\noindent
Note that, under $\Pp_x$, the process $(S_n)$ is a random walk whose increments
have the density $\tilde{\rho}(z) = \int_{\mathbb{R}} \rho(w) \rho(w-z) dw$.

\begin{lemma} \label{NTIFdetrep}
For $\beta \in [0,1]$, when $T(x,y) = \int^0_{-\infty} \rho(x-z) \rho(y-z) dz$
for a probability density  $\rho \in C(\mathbb{R}) \cap L^2(\mathbb{R}) \cap L^{\infty}(\mathbb{R}) $,
\begin{equation} \label{NTIFDPR}
\log \Det_{[-L,\infty)}(I - \beta T)  = - \E_{0} [ \beta^{\tau_{0-}} \,\delta_{0}(S_{\tau_{0-}}) (L- \tilde{M}_{\tau_{0-}})_+ ].
 \end{equation}
\end{lemma}

\noindent
\textbf{Proof.}
Arguing as in Lemma \ref{TIPR1} the log-trace formula formula 
\begin{eqnarray}
&& \log \Det_{[a,b]}(I - \beta T)  = - \sum_{n=1}^{\infty} \frac{\beta^n}{n} \Tr (T_{a,b}^n) \nonumber \\
&&  \hspace{.4in}  =  - \sum_{n=1}^{\infty} \frac{\beta^n}{n} \int_{[a,b]^n} T(x_1,x_2) \ldots 
T(x_{n-1},x_n) T(x_n,x_1) dx_1 \ldots dx_n \label{log-trace2}
\end{eqnarray}
holds for all $\beta \in [0,1]$. The derivative
\begin{eqnarray*}
\frac{d}{da} \Tr (T_{a,b}^n)  & = & - n \int_{[a,b]^{n-1}}  T(a,x_2) T(x_2,x_3) \ldots 
T(x_n,a) dx_2 \ldots dx_{n} \\
& = & - n \int_{[a,b]^{n-1}} dx_2 \ldots dx_{n}   \int_{(-\infty,0]^n} dz_1 \ldots dz_n \\
&& \hspace{.3in} 
\rho(a-z_1) \rho(x_2-z_1) \rho(x_2-z_2)\rho(x_3-z_2) \ldots \rho(x_n-z_n) \rho(a-z_n)\\
& = & - n \, \Pp_a[\tilde{S}_1<0, S_1 \in (a,b), \ldots, \tilde{S}_n <0, S_n \in da] \\
& = & - n \, \E_a [ \delta_a(S_{n}); \tau_{(a,b)^c} = n, \tilde{M}_n <0].
\end{eqnarray*}
Subsitituting this into (\ref{log-trace2}) we find 
\[
\frac{d}{da} \log \Det_{[a,b]}(I - \beta T) =  \E_a[ \beta^{\tau_{a-}}\delta_a(S_{\tau_{a-}}); 
\tau_{b+} > \tau_{a-}, \tilde{M}_{\tau_{a-}} < 0]. 
\]
Integrating this equality over $[a,b]$ gives
\[
 \log \Det_{[a,b]}(I - \beta T) = - \int^b_a  
 \E_c[ \beta^{\tau_{c-}}\delta_c(S_{\tau_{c-}}); \tau_{b+} > \tau_{c-},  \tilde{M}_{\tau_{c-}} < 0]  \, dc.
 \]
Both side of this identity are decreasing in $b$. Setting $a=-L$ and letting $b \to \infty$ we reach
\begin{eqnarray*}
 \log \Det_{[-L,\infty)}(I - \beta T) &=& - \int^{\infty}_{-L}  
 \E_c[ \beta^{\tau_{c-}}\delta_c(S_{\tau_{c-}});  \tilde{M}_{\tau_{c-}} < 0]  \, dc \\
 & = &  - \int^{\infty}_{-L} 
 \E_0[ \beta^{\tau_{0-}}\delta_0(S_{\tau_{0-}});  \tilde{M}_{\tau_{0-}} < -c]  \, dc \\
 & = & - \E_{0} [ \beta^{\tau_{0-}} \,\delta_{0}(S_{\tau_{0-}}) (L- \tilde{M}_{\tau_{0-}})_+ ].
 \end{eqnarray*}
 \qed
 \begin{lemma} \label{NTIFk3rep}
The limit $ \det^{a,\infty}_3(K) = \lim_{b \to \infty} \det^{a,b}_3(K)$ of the finite rank determinant from Lemma \ref{TWmanip2} exists 
and is given, when $p \neq \frac12$ , by 
\begin{eqnarray*} 
\det^{a,\infty}_3(K) = \frac{2}{1-2p}
 \left| \begin{matrix} 1-p - \frac12  \E_a [ \beta_p^{\tau_{a-}}; \tilde{\tau}_{0+} > \tau_{a-}]
 &   - \frac12 \E_{a}[ \beta_p^{\tilde{\tau}_{0+}-1}; \tau_{a-} \geq \tilde{\tau}_{0+}] \\  
 - p \E_a [ \beta_p^{\tilde{\tau}_{0+}-1}; \tilde{\tau}_{0+} \leq \tau_{a-}] 
  &  \frac12 - p \E_{a} [ \beta_p^{\tau_{a-}-1}; \tau_{a-} < \tilde{\tau}_{0+}]
  \end{matrix} \right|
\end{eqnarray*}
and by $\Pp_a[\tilde{\tau}_{0+} \leq \tau_{a-}]$ when $p=\frac12$.
\end{lemma}

\noindent
\textbf{Proof.}
We will represent each term $(Re_i,f_j)$ in terms of the two-step walk. 
All nine terms are somewhat similar, and we detail just a few. For example,
\begin{eqnarray}
&&\hspace{-.3in} (R\Phi_{\rho},\delta_{x_0}) = R\Phi_{\rho}(x_0) = \sum_{n=0}^{\infty} \beta_p^n T^n_{a,b} 
\Phi_{\rho}(x_0) \nonumber \\
& = & \sum_{n=0}^{\infty} \beta_p^n \int_{[a,b]^n} dx_1 \ldots dx_n T(x_0,x_1) \ldots T(x_{n-1},x_n) \Phi_{\rho}(x_n) 
\nonumber \\
& = & \sum_{n=0}^{\infty} \beta_p^n \int_{[a,b]^n} dx_1 \ldots dx_n \int_{(-\infty,0]^n} dy_1 \ldots dy_n \int_{0}^{\infty} dz 
\nonumber \\
&& \; \; \rho(x_0-y_1) \rho(x_1-y_1) \rho(x_1-y_2) \rho(x_2-y_2) \ldots  \rho(x_{n-1}-y_n) \rho(x_n-y_n) \rho(x_n-z)
\nonumber  \\
& = & \sum_{n=0}^{\infty} \beta_p^n  \Pp_{x_0}[ \tilde{S}_1 <0, S_1 \in (a,b), \ldots, \tilde{S}_n <0, S_n \in (a,b),
\tilde{S}_{n+1} >0] 
\nonumber \\
 & = & \sum_{n=0}^{\infty} \beta_p^n  \Pp_{x_0}[ \tilde{\tau}_{0+} = n+1, \tau_{(a,b)^c} \geq n+1]
 \nonumber \\
 &&  =\E_{x_0}[ \beta_p^{\tilde{\tau}_{0+} -1}; \tau_{(a,b)^c} \geq \tilde{\tau}_{0+}]. \label{420}
\end{eqnarray}
A similar exact calculation shows that, for bounded $f$,
\[
(R T f,\delta_{x_0}) = \sum_{n=0}^{\infty} \beta_p^n \E_{x_0} [ f(S_{n+1}); \tilde{\tau}_{0+} \geq n+2, \tau_{(a,b)^c} \geq n+1].
\]
Using this for $f= \I_{[b,\infty)}$ gives
\begin{eqnarray}
(R T \I_{[b,\infty)},\delta_{x_0}) &=& \sum_{n=0}^{\infty} \beta_p^n \Pp_{x_0} [ S_{n+1} > b, \tilde{\tau}_{0+} \geq n+2, \tau_{(a,b)^c} \geq n+1] \nonumber \\
& = &  \sum_{n=0}^{\infty} \beta_p^n \Pp_{x_0} [ \tilde{\tau}_{0+} \geq n+2, \tau_{b+} = n+1, \tau_{b+}< \tau_{a-} ] \nonumber \\
& = & \E_{x_0} [ \beta_p^{\tau_{b+}-1}; \tau_{b+} < \tau_{a-} \wedge \tilde{\tau}_{0+}] \label{421}
\end{eqnarray}
and similarly, using $f= \I_{(-\infty,a]}$,
\begin{equation} \label{422}
(R T \I_{(-\infty,a]},\delta_{x_0}) = \E_{x_0} [ \beta_p^{\tau_{a-}-1}; \tau_{a-} < \tau_{b+} \wedge \tilde{\tau}_{0+}]
\end{equation}
and, using $f = \I_{(a,b)}$, when $p \neq \frac12$,
\begin{eqnarray}
(R T \I_{(a,b)},\delta_{x_0}) &=& \sum_{n=0}^{\infty} \beta_p^n \Pp_{x_0} [ S_{n+1} \in (a,b), \tilde{\tau}_{0+} \geq n+2, \tau_{(a,b)^c} \geq n+1] \nonumber \\
& = &  \sum_{n=0}^{\infty} \beta_p^n \Pp_{x_0} [ \tilde{\tau}_{0+} \geq n+2, \tau_{(a,b)^c} \geq n+2] \nonumber \\
& = & \frac{1}{\beta_p-1}\E_{x_0} [ \beta_p^{(\tau_{(a,b)^c} \wedge \tilde{\tau}_{0+})-1} -1 ]. \label{423}
\end{eqnarray} 
 The entries of the form $(Re_i,f_1) = \beta_p (Re_i,\rho)$ start
with an integral against $\rho(w) dw$ and need a reflection $x \to -x$ to be 
written in terms of the two-step walk which start with an 
increment with density $\rho(-w)dw$. For example
 \begin{eqnarray*}
&& \hspace{-.3in} (RTf,\rho) = \sum_{n=0}^{\infty} \beta_p^n \int^b_a dw \rho(w) T^n_{a,b} Tf(w) \\
& = & \sum_{n=0}^{\infty} \beta_p^n \int_{[a,b]^{n+1}} dw dx_1 \ldots dx_n \rho(w) T(w,x_1) \ldots T(x_{n-1},x_n) Tf(x_n) \\
& = & \sum_{n=0}^{\infty} \beta_p^n \int_{[a,b]^{n+1}} dw dx_1 \ldots dx_n \int_{(-\infty,0]^{n+1}} dy_1 \ldots dy_n dz' 
\int_{-\infty}^{\infty} dz \\
&& \;\;\rho(w) \rho(w-y_1) \rho(x_1-y_1) \ldots  \rho(x_{n-1}-y_n) \rho(x_n-y_n) \rho(x_n-z')  \rho(z-z') f(z)\\
& = & \sum_{n=0}^{\infty} \beta_p^n \int_{[-b,-a]^{n+1}} dw dx_1 \ldots dx_n \int_{[0,\infty)^{n+1}} dy_1 \ldots dy_n dz' 
\int_{-\infty}^{\infty} dz \\
&& \;\;\rho(-w) \rho(y_1-w) \rho(y_1-x_1) \ldots  \rho(y_n-x_{n-1}) \rho(y_n-x_n) \rho(z'-x_n)  \rho(z'-z) f(-z)\\
& = & \sum_{n=0}^{\infty} \beta_p^n  \E_0[ f(-\tilde{S}_{n+2}); \tilde{S}_1 \in(-b,-a) , S_1 >0 , \ldots, 
\tilde{S}_{n+1} \in(-b,-a), S_{n+1} >0] \\
 & = & \sum_{n=0}^{\infty} \beta_p^n  \E_0[ f(-\tilde{S}_{n+2}); \tau_{0-} \geq n+2, \tilde{\tau}_{(-b,-a)^c} \geq n+2] \\
  & = & \sum_{n=0}^{\infty} \beta_p^n  \E_a[ f(a-\tilde{S}_{n+2}); \tau_{a-} \geq n+2, \tilde{\tau}_{(a-b,0)^c} \geq n+2].
\end{eqnarray*}
Using this for $f= \I_{[b,\infty)}$ gives
\begin{eqnarray}
(R T \I_{[b,\infty)},\rho)  &=& \sum_{n=0}^{\infty} \beta_p^n 
\Pp_a [ \tilde{\tau}_{(a-b)-} = n+2, \tilde{\tau}_{0+} > n+2, \tau_{a-} \geq n+2] \nonumber \\
& = & \E_a [ \beta_p^{\tilde{\tau}_{(a-b)-}-2}; \tilde{\tau}_{(a-b)-} < \tilde{\tau}_{0+} \wedge (1+\tau_{a-})] 
- \beta_p^{-1} \Pp_a[ \tilde{\tau}_{(a-b)-} = 1] \nonumber \\
& = & \E_a [ \beta_p^{\tilde{\tau}_{(a-b)-}-2}; \tilde{\tau}_{(a-b)-} < \tilde{\tau}_{0+} \wedge (1+\tau_{a-})] 
- \beta_p^{-1} (1- \Phi_{\rho}(b)) \label{424}
\end{eqnarray}
where the final subtracted term emerges since the sum over $n$ does not include the event $ \{\tilde{\tau}_{(a-b)-} = 1\}$.
Similarly, using $f= \I_{(-\infty,a]}$,
\begin{equation} \label{425}
(R T \I_{(-\infty,a]},\rho) =  \E_a [ \beta_p^{\tilde{\tau}_{0+}-2}; \tilde{\tau}_{0+} < \tilde{\tau}_{(a-b)-} \wedge (1+\tau_{a-})] 
- \beta_p^{-1} \Phi_{\rho}(a)
\end{equation}
and, using $f = \I_{(a,b)}$, when $p \neq \frac12$,
\begin{eqnarray}
(R T \I_{(a,b)},\rho) &=& \sum_{n=0}^{\infty} \beta_p^n 
\Pp_a [ \tilde{\tau}_{(a-b,0)^c} \geq n+3, \tau_{a-} \geq n+2] \nonumber \\
& = &  \frac{1}{\beta_p-1} \E_a [ \beta_p^{(\tilde{\tau}_{(a-b,0)^c}-2) \wedge (\tau_{a-} -1)}-1] 
+ \frac{1}{\beta_p} \Pp_a [ \tilde{\tau}_{(a-b,0)^c} = 1] \nonumber \\
& = &  \frac{1}{\beta_p-1} \E_a [ \beta_p^{(\tilde{\tau}_{(a-b,0)^c}-2) \wedge (\tau_{a-} -1)}-1] 
+ \frac{1}{\beta_p}(1+\Phi_{\rho}(a) - \Phi_{\rho}(b)). \label{426}
\end{eqnarray}
The final representation needed is derived in a  similar manner:
\begin{equation} \label{427}
(R \Phi_{\rho},\rho) = \E_a [ \beta^{\tau_{a-}-1}; \tilde{\tau}_{(a-b,0)^c} > \tau_{a-}].
\end{equation}
The formulae (\ref{420}),...,(\ref{427}) can be substituted into the matrix elements $(R e_i,f_j)$ and each has a limit
 as $b \to \infty$. Before evaluating these limits, it is convenient first to do two row operations to convert $e_1,e_2,e_3$ to
$ \hat{e}_1 =  e_1$, $\hat{e}_2 =  e_2 + p(2-\Phi_{\rho}(b) - \Phi_{\rho}(a))  e_1$, and 
 $\hat{e}_3 =  e_3 - p(2p-1)(\Phi_{\rho}(b) - \Phi_{\rho}(a)) e_1$ so that
\begin{eqnarray*}
&& \hat{e}_1 = - \frac12 \Phi_{\rho}, \quad \hat{e}_2 =  p T \I_{[b,\infty)} - p T \I_{(-\infty,a]}, \quad \hat{e}_3 = -p(2p-1) T \I_{(a,b)}, \\
&& f_1 = \beta_p \rho, \quad f_2 = \delta_a - \delta_b, \quad  f_3 = \delta_a + \delta_b.
\end{eqnarray*}
 Under these operations we change 
 \[
 \det^{a,b}_3(K) = \det \left( I + ((R e_i,f_j): i,j \leq 3)\right) = \det \left( \hat{I} + ((R \hat{e}_i,f_j): i,j \leq 3)\right)
 \] 
 where
 \[
\hat{I} =  \left( \begin{matrix} 1 & 0 & 0\\  p(2-\Phi_{\rho}(b) - \Phi_{\rho}(a))  & 1 & 0\\ 
-p(2p-1)(\Phi_{\rho}(b) - \Phi_{\rho}(a)) & 0 & 1 \end{matrix} \right).
 \]
Using (\ref{420}) we have 
\begin{eqnarray*}
(R \hat{e}_1,f_2) &=& - \frac12 ( R \Phi_{\rho}, \delta_a - \delta_b) \\
& = & \frac12 \E_{b}[ \beta_p^{\tilde{\tau}_{0+}-1}; \tau_{(a,b)^c} \geq \tilde{\tau}_{0+}] - \frac12 \E_{a}[ \beta_p^{\tilde{\tau}_{0+}-1}; \tau_{(a,b)^c} \geq \tilde{\tau}_{0+}] \\
& \to &  
\frac12 - \frac12 \E_{a}[ \beta_p^{\tilde{\tau}_{0+}-1}; \tau_{a-} \geq \tilde{\tau}_{0+}] \quad \mbox{as $b \to \infty$}
\end{eqnarray*}
since $\Pp_a[ \tau_{b+} \to \infty]=1$ and $\Pp_b[ \tilde{\tau}_{0+} =1] \to 1$ as $b \to \infty$. The limiting entry for
and $(R\hat{e}_1,f_3)$ differs only by a sign.  Using (\ref{423}) we have, when $p \neq \frac12$,
\begin{eqnarray*}
(R \hat{e}_3,f_2) &=& -p(2p-1) (RT\I_{(a,b)},\delta_a - \delta_b) \\
& = & \frac{p}{2p-1} \left( \E_{a}[ \beta_p^{(\tau_{(a,b)^c} \wedge \tilde{\tau}_{0+})-1} -1 ] 
- \E_{b} [ \beta_p^{(\tau_{(a,b)^c} \wedge \tilde{\tau}_{0+})-1} -1 ] \right) \\
& \to & \frac{p}{2p-1} \E_{a} [ \beta_p^{(\tau_{a-} \wedge \tilde{\tau}_{0+})-1} -1] \quad \mbox{as $b \to \infty$.}
\end{eqnarray*}
The limiting form for the entry $(R \hat{e}_3,f_3)$ is the same. Using (\ref{421}) and (\ref{422}) we have
\begin{eqnarray*}
(R \hat{e}_2,f_2) &=& p(RT\I_{[b,\infty)} - RT_{(-\infty,a]},\delta_a - \delta_b) \\
& = &  p  \E_{a} [ \beta_p^{\tau_{b+}-1}; \tau_{b+} < \tau_{a-} \wedge \tilde{\tau}_{0+}] - p
\E_{b} [ \beta_p^{\tau_{b+}-1}; \tau_{b+} < \tau_{a-} \wedge \tilde{\tau}_{0+}] \\
&& - p \E_{a} [ \beta_p^{\tau_{a-}-1}; \tau_{a-} < \tau_{b+} \wedge \tilde{\tau}_{0+}]
+ p \E_{b} [ \beta_p^{\tau_{a-}-1}; \tau_{a-} < \tau_{b+} \wedge \tilde{\tau}_{0+}] \\
& \to & - p \E_{a} [ \beta_p^{\tau_{a-}-1}; \tau_{a-} < \tilde{\tau}_{0+}] \quad \mbox{as $b \to \infty$}
\end{eqnarray*}
and the limiting form for $(R \hat{e}_2,f_3)$ is the same.  
Using (\ref{427}), and $\Pp_a [\tilde{\tau}_{b-} \to \infty] = 1$ as $b \to \infty$, we have  
 \[
 (R \hat{e}_1,f_1) = - \frac12 \E_a [ \beta_p^{\tau_{a-}}; \tilde{\tau}_{(a-b,0)^c} > \tau_{a-}]
 \to - \frac12  \E_a [ \beta_p^{\tau_{a-}}; \tilde{\tau}_{0+} > \tau_{a-}].
 \]
Using (\ref{426}) we have, when $p \neq \frac12$,
\begin{eqnarray*}
(R \hat{e}_3,f_1) &=& -p(2p-1) \beta_p (RT\I_{(a,b)},\rho) \\
& = & -\frac{p(2p-1)}{\beta_p-1} \E_a [ \beta_p^{(\tilde{\tau}_{(a-b,0)^c}-1) \wedge \tau_{a-} }-\beta_p]  
-p(2p-1) (1+\Phi_{\rho}(a) - \Phi_{\rho}(b)) \\
& \to &  \frac{p}{2p-1} \E_a [ \beta_p^{(\tilde{\tau}_{0+}-1) \wedge \tau_{a-} }] - \frac{4p^2(1-p)}{2p-1} -
p(2p-1) \Phi_{\rho}(a)
\quad \mbox{as $b \to \infty$.}
\end{eqnarray*}
Using (\ref{424}) and (\ref{425}) we have
\begin{eqnarray*}
(R \hat{e}_2,f_1) &=& p \beta_p (RT\I_{[b,\infty)} - RT\I_{(-\infty,a]},\rho) \\
& = & p \beta_p \E_a [ \beta_p^{\tilde{\tau}_{(a-b)-}-2}; \tilde{\tau}_{(a-b)-} < \tilde{\tau}_{0+} \wedge (1+\tau_{a-})] 
- p (1- \Phi_{\rho}(b))  \\
&& - p \beta_p \E_a [ \beta_p^{\tilde{\tau}_{0+}-2}; \tilde{\tau}_{0+} < \tilde{\tau}_{(a-b)-} \wedge (1+\tau_{a-})] 
+  p \Phi_{\rho}(a) \\
& \to &   - p \E_a [ \beta_p^{\tilde{\tau}_{0+}-1}; \tilde{\tau}_{0+} \leq \tau_{a-}] 
+  p \Phi_{\rho}(a)
\quad \mbox{as $b \to \infty$.}
\end{eqnarray*}
Combining with the entries in $\hat{I}$, this completes the limiting values for all nine terms for $\det^{a,\infty}_3(K)$ as
the determinant, when $p \neq \frac12$,
\[
 \left| \begin{matrix} 1 - \frac12  \E_a [ \beta_p^{\tau_{a-}}; \tilde{\tau}_{0+} > \tau_{a-}]
 & \frac12 - \frac12 \E_{a}[ \beta_p^{\tilde{\tau}_{0+}-1}; \tau_{a-} \geq \tilde{\tau}_{0+}] 
  & - \frac12 - \frac12 \E_{a}[ \beta_p^{\tilde{\tau}_{0+}-1}; \tau_{a-} \geq \tilde{\tau}_{0+}] \\  
   p - p \E_a [ \beta_p^{\tilde{\tau}_{0+}-1}; \tilde{\tau}_{0+} \leq \tau_{a-}] 
  & 1 - p \E_{a} [ \beta_p^{\tau_{a-}-1}; \tau_{a-} < \tilde{\tau}_{0+}] & 
   - p \E_{a} [ \beta_p^{\tau_{a-}-1}; \tau_{a-} < \tilde{\tau}_{0+}] \\ 
 \frac{p}{2p-1} \E_a [ \beta_p^{(\tilde{\tau}_{0+}-1) \wedge \tau_{a-} }]  - \frac{p}{2p-1}&   \frac{p}{2p-1} \E_{a} [ \beta_p^{(\tau_{a-} \wedge \tilde{\tau}_{0+})-1} -1]
& 1 +\frac{p}{2p-1} \E_{a} [ \beta_p^{(\tau_{a-} \wedge \tilde{\tau}_{0+})-1} -1] \end{matrix} \right|.
\]
Row and column operations considerably simplify this: after the column operation $C_2 \to C_2-C_3$  and then
the row operation 
$R_3 \to R_3 + \frac{2p}{2p-1} R_1 + \frac{1}{2p-1} R_2$ we reach, when $p \neq \frac12$,
\[
\det^{a,\infty}_3(K) = \frac{1}{2p-1}
 \left| \begin{matrix} 1 - \frac12  \E_a [ \beta_p^{\tau_{a-}}; \tilde{\tau}_{0+} > \tau_{a-}]
 & 1  & - \frac12 - \frac12 \E_{a}[ \beta_p^{\tilde{\tau}_{0+}-1}; \tau_{a-} \geq \tilde{\tau}_{0+}] \\  
  p - p \E_a [ \beta_p^{\tilde{\tau}_{0+}-1}; \tilde{\tau}_{0+} \leq \tau_{a-}] 
  & 1 &  - p \E_{a} [ \beta_p^{\tau_{a-}-1}; \tau_{a-} < \tilde{\tau}_{0+}] \\ 
2p &   2 & -1\end{matrix} \right|.
\]
After $C_1 \to C_1 - pC_2$ and $C_3 \to C_3 + \frac12 C_2$ we reach,  when $p \neq \frac12$, 
\[
\det^{a,\infty}_3(K) = \frac{1}{2p-1}
 \left| \begin{matrix} 1-p - \frac12  \E_a [ \beta_p^{\tau_{a-}}; \tilde{\tau}_{0+} > \tau_{a-}]
 & 1  &  - \frac12 \E_{a}[ \beta_p^{\tilde{\tau}_{0+}-1}; \tau_{a-} \geq \tilde{\tau}_{0+}] \\  
 - p \E_a [ \beta_p^{\tilde{\tau}_{0+}-1}; \tilde{\tau}_{0+} \leq \tau_{a-}] 
  & 1 & \frac12 - p \E_{a} [ \beta_p^{\tau_{a-}-1}; \tau_{a-} < \tilde{\tau}_{0+}] \\ 
0 &   2 & 0\end{matrix} \right|
\]
which reduces to the form stated in the Lemma for $p \neq \frac12$. When $p=\frac12$ 
the calculation is easier and is omitted.  \qed.
\subsection{Asymptotics}  \label{s4.3}
We will derive the asymptotics in Theorems \ref{nonTIszego} and 
\ref{pfaffian_szego_edge}. 
\subsubsection{Random Walk Results} \label{s4.3.1}
We need some variants of the random walk results in Section \ref{s3.3.1} that hold for 
our two-step walk $(S_n, \tilde{S}_n)$. We use the construction and notation from Section \ref{s4.2}.
Recall that $\tilde{\rho}(z) = \int_{\mathbb{R}} \rho(w) \rho(w-z) dw$, which is automatically symmetric. 

 \paragraph*{Two-step Kac's Formula} We extend Kac's formula (\ref{kac_eq_3}) to show that for all $n \geq 1$
 \begin{equation} \label{kac_eq_3+}
\E_0 [\tilde{M}_{n}  \delta_0(S_{n})] =  \frac1n \int_{\mathbb{R}} x \left(\rho^{*n}(x)\right)^2 dx + \Kac_{\tilde{\rho}}(n).
\end{equation}
Note that 
\begin{equation} \label{kac_eq_3+-}
\E_0 [M_{n}  \delta_0(S_{n})] =  - \E_0 [m_{n}  \delta_0(S_{n})] = \Kac_{\tilde{\rho}}(n)
\end{equation}
by a direct application of Kac's formula to $(S_n)$. 
The extra term in (\ref{kac_eq_3+}) arises due the maximum $\tilde{M}_n$ being taken over
$(\tilde{S}_n)$ rather than $(S_n)$. Indeed
\begin{eqnarray}
\tilde{M}_n \delta_0(S_n) & = & \max\{\mathcal{X}_1, \mathcal{X}_1 + \mathcal{Y}_1+\mathcal{X}_2,
\ldots, \mathcal{X}_1 + \mathcal{Y}_1+ \ldots +\ldots \mathcal{Y}_{n-1} + \mathcal{X}_n\} \delta_0(S_n) \nonumber \\
 & = & \mathcal{X}_1 \delta_0(S_n) + 
 \max\{0,  \mathcal{Y}_1+\mathcal{X}_2,
\ldots, \mathcal{Y}_1+ \ldots +\ldots \mathcal{Y}_{n-1} + \mathcal{X}_n\} \delta_0(S_n) \hspace{.2in} \label{K+10}
\end{eqnarray}
Since 
\[
S_n = (\mathcal{Y}_1+\mathcal{X}_2) + (\mathcal{Y}_2+\mathcal{X}_3) + \ldots + (\mathcal{Y}_{n-1} + \mathcal{X}_n) 
+ (\mathcal{Y}_{n} + \mathcal{X}_1) 
\]
the second term in (\ref{K+10}) involves only the maximum of a one step walk with increments equal to 
$\mathcal{X}_k + \mathcal{Y}_k$, and hence has expectation $\Kac_{\tilde{\rho}}(n)$ by Kac's formula
(\ref{kac_eq_3}). The expectation of the first term in (\ref{K+10}) equals, by a cyclic symmetry argument,
\[
\E_0[ \mathcal{X}_1 \delta_0(S_n)] = \frac1n  \sum_{k=1}^n \E_0 \left[ \mathcal{X}_k \delta_0 (S_n) \right] 
= \frac1n \E_0 [ \mathcal{X} \delta_0(\mathcal{X} + \mathcal{Y})]
\]
where $\mathcal{X} = \sum_{k=1}^n \mathcal{X}_k$ has density $\rho^{*n}(x) dx$ and 
$\mathcal{Y} = \sum_{k=1}^n \mathcal{Y}_k$ has density $\rho^{*n}(-x) dx$, which leads to the stated formula.
\paragraph*{Cyclic symmetry} We will use cyclic symmetry to show
 \begin{equation}
 \E_0 [\min\{L,\tilde{M}_{n}\} \delta_0(S_n); \tau_{0-}= n] =
 \frac{1}{n} \, \E_0 [ \min\{L, \tilde{M}_{n} - m_{n}\} \delta_0(S_{n})] \label{CS5}
\end{equation}
and hence its $L \to \infty$ limit
 \begin{equation}
 \E_0 [\tilde{M}_{n}\delta_0(S_n); \tau_{0-}= n] =
 \frac{1}{n} \, \E_0 [ (\tilde{M}_{n} - m_{n}) \delta_0(S_{n})]. \label{CS4}
\end{equation}
The proof is similar to the one-step version (\ref{CS2}). Indeed, we consider the $n$ cyclic permutations of
the variables
$\left( (\mathcal{X}_1,\mathcal{Y}_1),  (\mathcal{X}_2,\mathcal{Y}_2), \ldots,  (\mathcal{X}_n,\mathcal{Y}_n) \right)$.
We define, as in Section \ref{s4.2}, $n$ different two-step walks $(S^{(p)}_{\cdot}, \tilde{S}^{(p)}_{\cdot})$ for
$p=0,1,\ldots,n-1$ as follows: $S^{(p)}_0=0$ and
\[
\tilde{S}^{(p)}_k = S^{(p)}_{k-1} + \mathcal{X}_{p \oplus k}, \quad S^{(p)}_k = \tilde{S}^{(p)}_k + \mathcal{Y}_{p \oplus k} \quad \mbox{for $k = 1,\ldots,n$.}
\]
where $p \oplus k$ is addition modulo $n$. Then the law of $(S^{(p)}_{\cdot}, \tilde{S}^{(p)}_{\cdot})$ is the same for all $p$.
Moreover the final value $S^{(p)}_n = \sum_{k \leq n} (\mathcal{X}_k + \mathcal{Y}_k)$ is independent of $p$.

We define the maxima, minima and stopping times $\tilde{M}^{(p)}_n, m^{(p)}_n,\tau^{(p)}_{0-}$ as in Section \ref{s4.2}, but 
indexed by the superscript $p$ when they are for the two-step walk $(S^{(p)}_{\cdot}, \tilde{S}^{(p)}_{\cdot})$. Then,
using $m_n =0$ whenever $S_n=0$ and $\tau_{0-}=n$, 
\begin{eqnarray*}
\E_0 [\min\{L,\tilde{M}_{n}\} \delta_0(S_n) \I(\tau_{0-} = n)] 
 & = & 
 \E_0 [\min\{L,\tilde{M}_{n}-m_n\} \delta_0(S_n) \I(\tau_{0-} = n)] \\
 & = & \frac{1}{n} 
 \sum_{p=0}^{n-1} \E_0 [\min\{L,\tilde{M}_{n}^{(p)}-m_n^{(p)}\} \I(\tau^{(p)}_{0-} = n)
 \delta_0(S^{(p)}_n) ] \\
 & = & \frac{1}{n} 
 \sum_{p=0}^{n-1} \E_0 [\min\{L,\tilde{M}_{n}-m_n\} \I(\tau^{(p)}_{0-} = n)
 \delta_0(S_n) ] 
\end{eqnarray*}
where the final equality comes from the identities
\[
S^{(p)}_k = S_{p \oplus k} - S_{p} \quad \mbox{and} \quad \tilde{S}^{(p)}_k = \tilde{S}_{p \oplus k} - S_{p}
\quad  \mbox{for all $k,p$ whenever $S_n =0$} 
\]
so that $\tilde{M}_{n}^{(p)}-m_n^{(p)}$ is independent of $p$ whenever $S_n=0$.
Finally there is exactly one cyclic permutation with $\{\tau_{0-}^{(p)} = n\}$, establishing (\ref{CS4}). 
\subsubsection{Proof of Theorem  \ref{nonTIszego}} \label{s4.3.1.5}
%
We note that an assumption that $\int_\R |x|^k \rho(x) dx < \infty$ implies that $\int_\R |x|^k \tilde{\rho}(x) dx < \infty$.
Suppose first that $\beta \in [0,1)$.  
From Lemma \ref{NTIFdetrep} we have 
\begin{eqnarray} 
&& \hspace{-.3in}  \log \Det_{[-L,\infty)}(I - \beta T)   \nonumber \\
& = & - \E_{0} [ \beta^{\tau_{0-}} \,\delta_{0}(S_{\tau_{0-}}) (L- \tilde{M}_{\tau_{0-}})_+ ] \!\!\!\!
\nonumber \\
& = &  -L \E_{0} [ \beta^{\tau_{0-}} \,\delta_{0}(S_{\tau_{0-}}) ] 
+  \E_{0} [ \beta^{\tau_{0-}} \,\delta_{0}(S_{\tau_{0-}}) \min\{L,\tilde{M}_{\tau_{0-}}\} ] \label{poit} \\
 & = &   -L \E_{0} [ \beta^{\tau_{0-}} \,\delta_{0}(S_{\tau_{0-}}) ]
+  \E_{0} [ \beta^{\tau_{0-}} \, \tilde{M}_{\tau_{0-}}  \delta_{0}(S_{\tau_{0-}}) ] +o(1) \nonumber
\end{eqnarray}
(where we show below the variable $\beta^{\tau_{0-}} \, \tilde{M}_{\tau_{0-}}  \delta_{0}(S_{\tau_{0-}})$ is integrable).
The first expectation involves only the walk $(S_n)$ and so is given, as in (\ref{ka1form}), using the increment density $\tilde{\rho}$,
giving the desired formula for $\kappa_1(\beta)$. The second expectation is given, using the cyclic symmetry (\ref{CS4}),  
the Kac formulae (\ref{kac_eq_3+}) and (\ref{kac_eq_3+-}), and then the argument from  (\ref{<1/2k2}) by
\begin{eqnarray*}
 && \hspace{-.4in} \E_{0} [ \beta^{\tau_{0-}} \, \tilde{M}_{\tau_{0-}}  \delta_{0}(S_{\tau_{0-}})] \\
 & = &
  \sum_{n=1}^{\infty} \beta^n  \E_{0} [ \tilde{M}_n \delta_{0}(S_{n}); \tau_{0-} =n ]  \\ 
  & = & \sum_{n=1}^{\infty} \frac{\beta^n}{n}  \E_{0} [ (\tilde{M}_n -m_n) \delta_{0}(S_{n})]  \\ 
 &= & \sum_{n=1}^{\infty}
  \frac{\beta^n}{n^2} \int_{\mathbb{R}} x (\rho^{*n}(x))^2 dx + 
  2 \sum_{n=1}^{\infty} \frac{\beta^n}{n} \Kac_{\tilde{\rho}}(n) \\
 & = &  \sum_{n=1}^{\infty}
  \frac{\beta^n}{n^2} \int_{\mathbb{R}} x (\rho^{*n}(x))^2 dx +
 \int_0^{\infty} x \left( \sum_{n=1}^{\infty} \frac{\beta^n  \tilde{\rho}^{*n}(x)}{n} \right)^2 dx
\end{eqnarray*}
which is finite by the first moment assumption and the boundedness of $\rho$,
completing the formula for $\kappa_2(\beta)$. 

When $\beta = 1$ we follow, with small changes, the argument from Lemma
\ref{12lemma}. Write 
\begin{equation}
 \E_{0} \left[ \delta_{0}(S_{\tau_{0-}}) \min\{L,\tilde{M}_{\tau_{0-}} \} \right] = \sum_{n=1}^{\infty} \tilde{p}(n,L)
 \label{1005+}
\end{equation}
where, using the cyclic symmetry technique (\ref{CS5}), 
\[
\tilde{p}(n,L) : =   \E_{0} \left[ \delta_{0}(S_n) \min\{L,\tilde{M}_n\}; \tau_{0-} = n \right] 
=  \frac1n \E_{0} \left[ \delta_{0}(S_n) \min\{L,\tilde{M}_n - m_n\}  \right]
\]
and we approximate
\begin{eqnarray}
\tilde{p}(n,L) & = & \frac1n \E_{0} \left[ \delta_{0}(S_n) (\tilde{M}_n - m_n)  \right] + \tilde{E}^{(1)}(n,L), \label{P1+} \\
\tilde{p}(n,L) & = & \frac1n \E_{0} \left[ \delta_{0}(W_{n \sigma^2}) \min\{L,\sup_{t \leq n \sigma^2} W_t - \inf_{t \leq n \sigma^2} W_t\}  
\right] + \tilde{E}^{(2)}(n,L). \label{P2+}
\end{eqnarray}
We verify in the subsection \ref{s6.3} that the same error bounds (\ref{E10}) and (\ref{E20}) hold in this case
and then the argument of Lemma \ref{12lemma} goes through, with the only change being the extra term in the
two-step Kac's formula (\ref{kac_eq_3+}). This leads to 
\begin{eqnarray*}
&&  \E_{0} \left[ \delta_{0}(S_{\tau_{0-}}) \min\{L,\tilde{M}_{\tau_{0-}} \} \right] =
 \log L + \frac{3}{2} \log 2 - \log \tilde{\sigma} \\
 &  &  
 \hspace{1.5in} + \sum_{n\geq 1} \frac2n \left( \mbox{Kac}_{\tilde{\rho}}(n) - \frac14 \right) 
+ \frac{1}{n^2} \int_{\mathbb{R}} x \left(\rho^{*n}(x)\right)^2 dx +
 o(1) 
\end{eqnarray*}
where $\tilde{\sigma}^2 = \int x^2 \tilde{\rho}(x) dx$. Using this in (\ref{poit}) completes the proof. \qed.
\subsubsection{Proof for Theorem \ref{pfaffian_szego_edge} when $p \in (0,\frac12)$} \label{s4.3.3}
We apply Lemma \ref{TWmanip2} for the interval $[-L,b]$ and take $b \to \infty$, giving
\[
2 \log \Pf_{[-L,\infty)} (\mathbf{J} - p \mathbf{K}) 
= \log \Det_{[-L,\infty)} (I - \beta_p T) +  \log \det^{-L,\infty}_3(K). 
\]
Lemma \ref{NTIFk3rep}
gives a probabilistic representation for $\det^{-L,\infty}_3(K)$ where it is simple to let $L \to \infty$. 
Indeed in this limit we get
\begin{eqnarray}
&& \hspace{-.2in} \det^{-L,\infty}_3(K) \nonumber \\
& = & \frac{2}{1-2p}
 \left| \begin{matrix} 1-p - \frac12  \E_{-L} [ \beta_p^{\tau_{(-L)-}}; \tilde{\tau}_{0+} > \tau_{(-L)-}]
 &   - \frac12 \E_{-L}[ \beta_p^{\tilde{\tau}_{0+}-1}; \tau_{(-L)-} \geq \tilde{\tau}_{0+}]    \\  
 - p \E_{-L} [ \beta_p^{\tilde{\tau}_{0+}-1}; \tilde{\tau}_{0+} \leq \tau_{(-L)-}] 
  &  \frac12 - p \E_{-L} [ \beta_p^{\tau_{(-L)-}-1}; \tau_{(-L)-} < \tilde{\tau}_{0+}]
  \end{matrix} \right|   \nonumber  \\
  & = & \frac{2}{1-2p}
 \left| \begin{matrix} 1-p - \frac12  \E_{0} [ \beta_p^{\tau_{0-}}; \tilde{\tau}_{L+} > \tau_{0-}]
 &   - \frac12 \E_{0}[ \beta_p^{\tilde{\tau}_{L+}-1}; \tau_{0-} \geq \tilde{\tau}_{L+}] \\  
 - p \E_{0} [ \beta_p^{\tilde{\tau}_{L+}-1}; \tilde{\tau}_{L+} \leq \tau_{0-}] 
  &  \frac12 - p \E_{0} [ \beta_p^{\tau_{0-}-1}; \tau_{0-} < \tilde{\tau}_{L+}]
  \end{matrix} \right|  \label{templim}  \\
& \to & \frac{2}{1-2p}
 \left| \begin{matrix} 1 - p - \frac12  \E_0 [ \beta_p^{\tau_{0-}}]
 &  0 \\  
0 & \frac12  - p \E_{0} [ \beta_p^{\tau_{0-}-1}] 
  \end{matrix} \right| = \frac{1-2p}{1-p} \nonumber
\end{eqnarray}
where we shifted the starting position for the expectations, then 
used $\Pp_0[ \tilde{\tau}_{L+} \to \infty] =1$ to take the limits, and finally 
evaluated the expectations using
Sparre Andersen's formula (\ref{andersen}) (since the stopping times $\tau_{0-}$ involve only the 
walk $(S_n)$ which is a one step random walk with symmetric increments) which gives
$\E_{0} [ \beta_p^{\tau_{0-}}] =2p$ when $p \in (0,\frac12)$.
Combined with Theorem \ref{nonTIszego}, which gives the asymptotic for the Fredholm determinant, this gives
the values of $\kappa_1(p), \kappa_2(p)$ as required.
\subsubsection{Proof for Theorem \ref{pfaffian_szego_edge} when $p \in (\frac12,1)$} \label{s4.3.4}
For $p \in (\frac12,1)$, Sparre Andersen's formula (\ref{andersen}) gives $\E_{0} [ \beta_p^{\tau_{0-}}] = 
2(1-p)$ which, as in (\ref{templim}), implies that  $\lim_{L \to \infty}  \det^{-L,\infty}_3(K) =0$.
Indeed, as in the translationally invariant case, this small determinant contributes to the leading 
order asymptotic. To resolve the asymptotic we use Sparre Andersen's formula via
\[
\E_{0} [ \beta_p^{\tau_{0-}}; \tilde{\tau}_{L+} > \tau_{0-}]
=
\E_{0} [ \beta_p^{\tau_{0-}}]
- \E_{0} [ \beta_p^{\tau_{0-}}; \tau_{0-} \geq \tilde{\tau}_{L+}]
\] 
and then argue that $\E_{0} [ \beta_p^{\tau_{0-}}; \tau_{0-} \geq \tilde{\tau}_{L+}] 
= o(\E_{0} [ \beta_p^{\tau_{L+}}; \tau_{0-} \geq \tilde{\tau}_{L+}]) $ as in (\ref{littleo}). Using these
in (\ref{templim}) we reach
\begin{eqnarray*}
 \det^{-L,\infty}_3(K) 
   & = & \frac{2}{1-2p}
 \left| \begin{matrix} \frac12  \E_{0} [ \beta_p^{\tau_{0-}}; \tau_{0-} \geq \tilde{\tau}_{L+}]
 &   - \frac12 \E_{0}[ \beta_p^{\tilde{\tau}_{L+}-1}; \tau_{0-} \geq \tilde{\tau}_{L+}] \\  
 - p \E_{0} [ \beta_p^{\tilde{\tau}_{L+}-1}; \tau_{0-} \geq  \tilde{\tau}_{L+} ] 
  &  \frac{p}{\beta_p} \E_{0} [ \beta_p^{\tau_{0-}}; \tau_{0-} \geq \tilde{\tau}_{L+}]
  \end{matrix} \right| \\
  & = & \frac{p}{(2p-1)\beta_p^2} \left( \E_{0}[ \beta_p^{\tilde{\tau}_{L+}}; \tau_{0-} \geq \tilde{\tau}_{L+}] \right)^2 (1+o(1)).
  \end{eqnarray*}
The non translationally invariant analogue of Lemma \ref{finalconstant} is as follows.
\begin{lemma} \label{finalconstant+}
Suppose  there exists $\phi_p>0$ so that $\beta_p \int e^{\phi_p x} \tilde{\rho}(x)dx = 1$, and that both
$\int |x| e^{\phi_p x} \tilde{\rho}(x) dx$ and $\int e^{\phi_p |x|} \rho(x) dx$ are finite. Then
\begin{equation} \label{olegf+}
\lim_{L \to \infty}  e^{\phi_p L} \, \E_0 [\beta_p^{\tilde{\tau}_{L+}}; \tau_{0-}  \geq \tilde{\tau}_{L+} ] = \beta_+
 \frac{\sqrt{1-\beta_p}}{\phi_p \E^{(p)}_0[S_1]} 
\left(\Pp^{(p)}_0 [ \tau_{0-} = \infty]\right)^2
\end{equation}
where $\beta_+^{-1} = \int_{\mathbb{R}} e^{\phi_p x} \rho(x) dx$.
\end{lemma}
Combined with Theorem \ref{nonTIszego}, which gives the asymptotic for the Fredholm determinant, 
and the exact formula (\ref{entranceprob}), this lemma leads to the stated forms for $\kappa_1(p),\kappa_2(p)$. 

\vspace{.1in}
\noindent
\textbf{Proof of Lemma \ref{finalconstant+}.}
Choose $\beta_-,\beta_+$ so that $\beta_- \int_\R \rho(-x) e^{\phi_p x} dx = \beta_+ \int_\R \rho(x) e^{\phi_p x} dx =1$.
Then $\beta_p = \beta_- \beta_+$. 
Let $\Pp^{(p)}_x, \E^{(p)}_x$ be the tilted probability and expectation, where the two-step walk
$(\tilde{S}_n,S_n)$ uses increments $\mathcal{X}_i,\mathcal{Y}_i$ with the tilted densities $\beta_- \exp(\phi_p x) \rho(-x) dx$
and $\beta_+ \exp(\phi_p x) \rho(x) dx$. Then defining
\[
\tilde{Z}_n = \beta_- \beta_p^{n-1} e^{\phi_p \tilde{S}_n}  \qquad Z_n = \beta_p^n e^{\phi_p S_n} \quad \mbox{for $n \geq 1$}
\]
the process $(1, \tilde{Z}_1,Z_1,\tilde{Z}_2,Z_2,\ldots)$ is a martingale. Moreover
\begin{eqnarray*}
\E^{(p)}_0[ e^{-\phi_p \tilde{S}_{\tilde{\tau}_{L+}}}; \tau_{0-} \geq \tilde{\tau}_{L+} ]
&=& \E_0 [ e^{-\phi_p \tilde{S}_{\tilde{\tau}_L+}} \tilde{Z}_{\tilde{\tau}_{L+}};  \tau_{0-} \geq \tilde{\tau}_{L+}] \\
&=&  \E_0 [\beta_- \beta_p^{\tilde{\tau}_{L+}-1}; \tau_{0-} \geq \tilde{\tau}_{L+}]
\end{eqnarray*}
so that
\begin{eqnarray}
e^{\phi_p L} \, \E_0 [\beta_p^{\tilde{\tau}_{L+}};  \tau_{0-} \geq \tilde{\tau}_{L+} ] 
&=&  \beta_+ \E^{(p)}_0[ e^{-\phi_p (\tilde{S}_{\tilde{\tau}_{L+}}-L)};  \tau_{0-} \geq \tilde{\tau}_{L+}]. \label{p1234+}
\end{eqnarray}
By conditioning on the value of $\tilde{S}_{1}$ we see that 
\[
\tilde{V}(L) := 
\E^{(p)}_0[ e^{-\phi_p (\tilde{S}_{\tilde{\tau}_{L+}}-L)}] = \int_{-\infty}^L \rho(-x) V(L-x) dx + \int_L^{\infty} \rho(-x) e^{-\phi_p(x-L)} dx
\]
where $V(L) = \E^{(p)}_0[ \exp(-\phi_p (S_{\tau_{L+}}-L))]$. We know from (\ref{renewal}) that
$V(L) \to C_h$ as $L \to \infty$, and we deduce that $\tilde{V}(L) \to C_h$.  
Conditioning on $\sigma(S_{\tau_{0-}})$, we see
\begin{eqnarray*}
\E^{(p)}_0[ e^{-\phi_p (\tilde{S}_{\tilde{\tau}_{L+}}-L)};  \tau_{0-} \geq \tilde{\tau}_{L+}]
&=& \tilde{V}(L) - \E^{(p)}_0[ e^{-\phi_p (\tilde{S}_{\tilde{\tau}_{L+}}-L)}; \tau_{0-} < \tilde{\tau}_{L+} ]. \\
& = & \tilde{V}(L) - \E^{(p)}_0[ \tilde{V}(L- S_{\tau_{0-}}); \tau_{0-} < \tilde{\tau}_{L+}] \\
& \to & C_h - C_h \Pp^{(p)}_0[ \tau_{0-} < \infty] = C_h \Pp^{(p)}_0[ \tau_{0-} = \infty].
\end{eqnarray*}
With (\ref{p1234+}) this implies that $e^{\phi_p L} \, \E_0 [\beta_p^{\tilde{\tau}_{L+}}; \tau_{0-}  \geq \tilde{\tau}_{L+}] 
\to \beta_+ C_h \Pp^{(p)}_0[ \tau_{0-} = \infty]$ and the desired form for the limit follows from the expression for
$C_h$ in Lemma \ref{finalconstant}. \qed.

\subsubsection{Proof for Theorem \ref{pfaffian_szego_edge} when $p = \frac12$} \label{s4.3.5}
Applying Lemma \ref{TWmanip2} for the interval $[-L,b]$ and taking $b \to \infty$, and then
using the probabilitistic representation in Lemma
\ref{NTIFk3rep} for $\det^{-L,\infty}_3(K)$, gives
\begin{equation} \label{ert}
2 \log \Pf_{[-L,\infty)} (\mathbf{J} - p \mathbf{K}) 
=  \log \Det_{[-L,\infty)}(I - T)
+ \log \Pp_0
 [\tilde{\tau}_{L+} \leq \tau_{0-}].
 \end{equation}
Comparing with the translationally invariant analogue (\ref{E1}) we see there is a missing
$\log 2$. This and the slightly different form for the two-step Kac's formula (\ref{kac_eq_3+}) 
turn out to be the only differences between the two asymptotics when $p=\frac12$. 

Let $\mu = \int_{\mathbb{R}} x \rho(-x) dx$. Then the process
\[
(0, \tilde{S}_1-\mu, S_1, \tilde{S}_2 - \mu, S_2, \ldots)
\]
is a martingale under $\Pp_0$. The optional stopping theorem implies
\begin{eqnarray*}
0 &=& \E_0[ S_{\tau_{0-}}; \tau_{0-} < \tilde{\tau}_{L+}] +  \E_0[ \tilde{S}_{\tilde{\tau}_{L+}}-\mu; \tilde{\tau}_{L+} \leq \tau_{0-}] \\
&=& \E_0[ S_{\tau_{0-}}] - \E_0[ S_{\tau_{0-}}; \tilde{\tau}_{L+} \leq \tau_{0-}]  +
\E_0[ \tilde{S}_{\tilde{\tau}_{L+}} -L; \tilde{\tau}_{L+} \leq \tau_{0-}] + 
(L + \mu) \Pp_0[\tilde{\tau}_{L+} \leq \tau_{0-}] 
\end{eqnarray*}
(the argument from Lemma 5.1.1 in \cite{lawler} justifies the optional stopping theorem being valid). 
Rearranging gives
\[
\Pp_0[\tilde{\tau}_{L+} \leq \tau_{0-}]  = - \frac{\E_0[ S_{\tau_{0-}}]}{L+\mu} + o(1) \quad \mbox{as $L \to \infty$}
\]
by arguing as in section \ref{s3.3.4}. Together with the asymptotics from Theorem \ref{nonTIszego} for the Fredholm determinant in (\ref{ert}), and Spitzer's formula (\ref{spitzer}), this finishes the calculation. 
%
\section{The proof of Theorem \ref{Exittheorem}}  \label{s5}
In this section we will establish the Pfaffian structure for exit measures
and find the corresponding kernels. The arguments broadly follow those in
\cite{TZ} and \cite{GPTZ}, which derive the Pfaffian kernels from duality identities. The new 
feature here is that the dual process, which is a system of annihilating motions,  
has immigration of particles. 
\subsection{Product ratio moments}  \label{s5.1}
The starting point that shows 
that the Pfaffian structure still holds is the following Pfaffian formula for product moments
 for annihilating particle systems with immigration.

Consider the following finite particle system: between reactions particles evolve as independent strong Markov process 
motions on $\mathbb{R}$; upon collision any pair of particles instantaneously annihilate. The processes starts from immigrated particles, starting at the space-time points $z_i = (y_i,t_i) \in [0,t] \times \mathbb{R}$ for $i = 1, \ldots, 2n$ for some $n \geq 0$. 
We list the positions of all surviving particles at time $t$ as $Y^1_t< Y^2_t < \ldots $ in increasing order. 
Note that number of particles alive at time $t$ will be even since the total number of immigrated particles 
 is even and we remove particles in pairs upon annihilation. Some restriction is needed on the motion process, for example to ensure no triple collisions occur. Since we need only the two examples of Brownian motions on $\mathbb{R}$ and Brownian motions with reflection on $[0,\infty)$, we restrict to these two cases below (which also makes some of the p.d.e. arguments straightforward), but the proof makes it clear that the result should hold more generally. 
 
 We write $Y_t$ for this point process at time $t$. We write $\mathbb{P}^{A}_{\mathbf{z}}$, where 
 $\mathbf{z} = (z_1, \ldots,z_{2n})$, for the law of this annihilating process $(Y_t: t \geq 0)$. 
\begin{lemma} \label{apmlemma}
Let $g,h:[0,\infty) \to \mathbb{R}$ be bounded and measurable. For the finite annihilating system described above with 
immigration at $\mathbf{z} = (z_1, \ldots,z_{2n}) \in ( [0,t] \times \mathbb{R})^{2n}$, define an alternating product moment by
\begin{equation} \label{apm}
M_{g,h} (Y_t) = \prod_{i \geq 1} g(Y^{2i-1}_t) \prod_{i \geq 1} h(Y^{2i}_t)
\end{equation}
where an empty product is taken to have value $1$. Then
\[
\mathbb{E}^A_{\mathbf{z}} [ M_{g,h}(Y_t)] = \pf \left( \mathbb{E}_{\mathbf{(z_i,z_j)}} [ M_{g,h}(Y_t)]: i < j \leq 2n \right).
\]
\end{lemma}
\noindent
Note that the terms in the Pfaffian use systems with just two particles, and so the 
double product moment $M_{g,h}(Y_t)$ takes either the value $g(Y^1_t) h(Y^2_t)$ or the value $1$, depending on whether the two particles have annihilated. 

\vspace{.1in}

\noindent
\textbf{Proof.} We give the proof in the case of reflected Brownian motions on $\mathbb{R}$ and indicate the slight simplifications for the full space case. The proof follows those in \cite{TZ} and \cite{GPTZ}, where the Kolmogorov equation for the expectation 
is shown to be solved by the Pfaffian. Due to us immigrating particles over the time interval $[0,t]$ we will solve the 
equation in the intervals between immigration times. For this we need more detailed notation, used only in this proof. 
The final time $t$, the number of particles $2n$ and the immigration positions $\mathbf{z}$ are fixed throughout the proof. 
We suppose the points $z_i = (y_i,t_i)$ are listed so that $y_i \geq 0$ and $t =t_0 > t_1 >  \ldots > t_{2n} > 0$; we will establish the result for such 
$\mathbf{z}$, and when there are one or more equalities between the time points $t_i$ the result follows by continuity in these variables. 
We write $\mathbf{z}^p$ for the vector $((y_1,t_1), \ldots, (y_p,t_p))$ when $ 1 \leq p \leq 2n$ (and $\mathbf{z}^0 = \emptyset$). 
Also we take $g,h$ to be continuous, and the measurable case can be established by approximation.

We write  
$V_k^+ = \{ \mathbf{x} = (x_1, \ldots, x_{k}): 0 \leq x_1 < \ldots < x_{k}\}$ for a cell in $\mathbb{R}^k$ and define $\mathbf{x}^q =
(x_1,\ldots,x_q) \in V_q^+$ when $q \leq k$.
From an element 
$\mathbf{x} \in V_k^+ $ we define a set of space-time points by
\[
(\mathbf{x},s) = ((x_1,s), \ldots, (x_k,s)) \quad \mbox{for  $s \in [0,t]$.}
\]
Define the system of functions
\begin{eqnarray*}
m_{s,t}^{(p,q)} ( \mathbf{z}^p, (\mathbf{x},s))
&=&  \mathbb{E}^A_{((y_1,t_1), \ldots, (y_p,t_p), (x_1,s), \ldots, (x_q,s))} [ M_{g,h}(X_t)] \\
&& \hspace{-1in} \mbox{for $p,q \geq 0$, $p+q \leq 2n$ and $p+q$ even, $\mathbf{x} \in V^+_q$, and $s \in [0,t_p)$.}
\end{eqnarray*}
Thus $(\mathbf{x}^q,s)$ will describe the positions particles alive at time $s$, and $\mathbf{z}^p$ describes the remaining
positions for particles to be immigrated after time $s$. Each of these functions satisfies a backwards heat equation with reflected boundary condition
\begin{eqnarray}
&& ( \partial_s + \frac12 \sum_{i=1}^q \partial^2_{x_i}) \; m_{s,t}^{(p,q)} ( \mathbf{z}^p, (\mathbf{x},s)) =0 \quad
\mbox{for $s \in [0,t_p)$ and $\mathbf{x} \in V^+_q$,} \label{pde1} \\
&& \partial_{x_1} m_{s,t}^{(p,q)} ( \mathbf{z}^p, (\mathbf{x},s))  = 0 \quad 
\mbox{for $s \in [0,t_p)$ and $0=x_1<x_2 < \ldots <x_q$.} \label{pde2}
\end{eqnarray}
(In the case of Brownian motions on $\mathbb{R}$ we remove the boundary condition (\ref{pde2}) and allow
$\mathbf{x} \in V_k = \{ \mathbf{x} = (x_1, \ldots, x_{k}): x_1 < \ldots < x_{k}\}$.)

The function $m_{s,t}^{(p,q)} ( \mathbf{z}^p, (\mathbf{x},s))$ extends continuously to $\mathbf{x} \in \overline{V}^+_k, s \in [0,t_p)$ and 
satisfies the boundary conditions, for $i =1,\ldots,q-1$,
\begin{eqnarray}
m_{s,t}^{(p,q)} ( \mathbf{z}^p, (\mathbf{x},s)) &=&  m_{s,t}^{(p,q-2)} ( \mathbf{z}^p, (\mathbf{x}^{i,i+1},s)) \nonumber \\
&& \hspace{-.63in}
\mbox{when $s \in [0,t_p)$ and $0<x_1 < \ldots <x_i = x_{i+1} < \ldots < x_q$} \label{pde3}
\end{eqnarray}
where $\mathbf{x}^{i,i+1} = (x_1,\ldots,x_{i-1},x_{i+1},\ldots,x_q)$ is the vector $\mathbf{x}$ with the co-ordinates
$x_i,x_{i+1}$ `annihilated'. (Conditions on other parts of the boundary $\partial V^+_k$ are not needed to ensure uniqueness.)
Finally they satisfy final conditions
\begin{eqnarray}
\lim_{s \uparrow t} m_{s,t}^{(0,2n)} (\emptyset, (\mathbf{x},s)) &=&   \prod_{i \geq 1} g(x_{2i-1}) \prod_{i \geq 1} h(x_{2i})
\quad \mbox{for $\mathbf{x} \in V^+_{2n}$} \nonumber \\
\lim_{s \uparrow t_p}  m_{t_p,t}^{(p,q)} ( \mathbf{z}^p, (\mathbf{x},s)) &=&  
m_{t_p,t}^{(p-1,q+1)} ( \mathbf{z}^{p-1}, (\mathbf{x}|y_p,s))  \label{pde4}  \\
&& \quad \mbox{when $p \geq 1$, $\mathbf{x} \in V^+_q$, and $y_p \not \in \{x_1, \ldots, x_q\}$, } \nonumber 
\end{eqnarray}
where $\mathbf{x}|y_{p}$ is the element of $V^+_{q+1}$ with coordinates $x_1, \ldots, x_q,y_{p}$ 
listed in increasing order. 

We claim the system (\ref{pde1},\ref{pde2},\ref{pde3},\ref{pde4}) of equations for 
$(m^{(p,q)}: \mbox{even $p+q \leq 2n$})$ has unique bounded solutions, where
$m^{(p,q)} \in C([0,t_p) \times \overline{V}^+_q) \cap C^{1,2} ([0,t_p) \times V^+_q)$.
 This is an inductive proof,
working downwards in $p=2n,2n-1, \ldots,1,0$, and for each fixed $p$ working upwards in $q=0,1,\ldots,2n-p$
(subject to $p+q$ being even) as indicated in the following diagram.
The functions to the left of right arrows determine the  boundary conditions for the functions on the right;
the functions to the northeast of  southwest arrows play the role of the final conditions for the functions to the
southwest of these arrows. Vertical arrows correspond to the final 
conditions for the functions  $m^{(0,2n)}_{s,t}$ at the top layer.
\begin{eqnarray*}\label{array} 
{\tiny
\begin{array}{lcccccccc}
&&g(x_1) h(x_2)&&\ldots&&{\displaystyle \prod_{k=1}^{n-1}}g(x_{2k-1})h(x_{2k})&&{\displaystyle \prod_{k=1}^{n}}g(x_{2k-1})h(x_{2k})\\ \\
&&\downarrow&&&&\downarrow&&\downarrow\\ \\
1=m_{s,t}^{(0,0)} &\mapsto&m_{s,t}^{(0,2)}&\mapsto&\ldots&\mapsto&
m_{s,t}^{(0,2n-2)}&\mapsto& m_{s,t}^{(0,2n)}\\ \\
&& \hspace{-.5in}  \swarrow&&~&&  \hspace{-.5in}  \swarrow&&\hspace{-.5in}  \swarrow\\\\
&m_{s,t}^{(1,1)}&\mapsto&\ldots&\mapsto&m_{s,t}^{(1,2n-3)}&\mapsto&m_{s,t}^{(1,2n-1)}&\\ \\
&\hspace{-.5in} \swarrow&&&&\hspace{-.5in} \swarrow&&\hspace{-.5in} \swarrow&\\ \\
1=m^{(2,0)}_{s,t} &\mapsto\ldots&&\mapsto&m^{(2,2n-4)}_{s,t}&\mapsto&m^{(2,2n-2)}_{s,t}&&\\ \\
&&&&\hspace{-.5in} \swarrow&&\hspace{-.5in} \swarrow&&\\ \\
&&\cdot&\cdot&\cdot&\cdot&&&\\ \\
\end{array}}
\end{eqnarray*}
We will now claim that, when $p+q \leq 2n$ is even, that  for $s \in [0,t_p]$ and $\mathbf{x} \in V^+_q$
\begin{equation} \label{pfaffclaim}
m_{s,t}^{(p,q)} ( \mathbf{z}^p, (\mathbf{x},s)) = \pf
  \left(\begin{array}{@{}c|c@{}}
  m_{s,t}^{(2,0)} ((y_i,t_i),(y_j,t_j))  &  m_{s,t}^{(1,1)} ((y_i,t_i),(x_j,s))\\
1 \leq i < j < p & i=1,\ldots,p, \; j=1,\ldots,q \\ \hline
  &  m_{s,t}^{(0,2)} ((x_i,s),(x_j,s))\\
& 1 \leq i < j < q
  \end{array}\right)
\end{equation}
where we have listed the upper triangular elements of this $(p+q) \times (p+q)$ antisymmetric 
block matrix. Specialising to $p=2n, q=0$ and $s=0$ we find the conclusion of the lemma.

The claim (\ref{pfaffclaim}) follows by uniqueness once 
we verify that the Pfaffian expression on the right hand side also solves the equations 
(\ref{pde1},\ref{pde2},\ref{pde3},\ref{pde4}). The arguments that the Pfaffian solves the p.d.e (\ref{pde1},\ref{pde2}), and 
the boundary conditions (\ref{pde3}), is the same as for the simpler one time period case in \cite{TZ}. 
The final conditions $(\ref{pde4})$ for $m_{t,t}^{(0,2n)}$ require that
\[
\prod_{i \geq 1} g(x^{2i-1}) \prod_{i \geq 1} h(x^{2i}) = \pf( g(x_i) h(x_j): i < j \leq 2n) \quad \mbox{for $\mathbf{x} \in V^+_{2n}$.}
\]
which follows since the antisymmetric matrix in the Pfaffian is of the form $D^T J_{2n} D$ where
$D$ is the diagonal matrix with entries $(g(x_1), h(x_2), \ldots, g(x_{2n-1})h(x_{2n}))$ and $J_{2n}$ is the 
block diagonal matrix with $n$ blocks of the form 
$\left( \begin{array}{cc} 0 & 1 \\ -1 & 0 \end{array} \right)$ along the diagonal (so that $\pf(J_{2n}) =1$). 
The final conditions for all other  $m_{t_p,t}^{(p,q)}$ when $p \geq 1$ follow inductively.
 \qed

\subsection{Dualities and thinning}  \label{s5.2}
We now follow the steps used in \cite{TZ}  (to which we will refer for some details) 
using the Brownian web on a half-space developed in \cite{Toth_Werner}. This gives a duality formula
for a coalescing system on $(0,\infty)$ together with its exit measure on $\{0\} \times [0,\infty)$. 

Let $(X_t: t \geq 0)$ be a system of (instantaneously) coalescing Brownian motions on $(0,\infty) \times [0,\infty)$,
frozen upon hitting the boundary, and let $X_e$ be the exit measure on $\{0\} \times [0,\infty)$. 
We consider $X_t$ as a locally finite simple point measure on $(0,\infty)$ at each $t \geq 0$. 
We start by considering the case where the initial condition $X_0=\mu$ is deterministic and contains a finite
number $\mu(0,\infty)$ of particles. 
We write $\mathbb{P}^{C}_{\mu}$ for the law of this process. 

The following lemma follows immediately from the non-crossing properties of the half-space Brownian web and 
dual Brownian web paths (which are reflected Brownian motions). It characterises the joint law of 
$X_t$ and $X_e|_{\{0\} \times [0,t]}$.
We use, from section \ref{s5.1}, the law
$ \mathbb{P}^{A}_{\mathbf{z}}$ of a finite annihilating system of Brownian motions $(Y_y)$ with reflection on $[0,\infty)$, 
with immigrated particles at $\mathbf{z}$. 
\begin{lemma} (Toth and Werner \cite{Toth_Werner}) 

\noindent
Let $I_1 = \{0\} \times [a_1,a_2], \ldots, I_n = \{0\} \times [a_{2n-1},a_{2n}]$, for 
$\mathbf{a} \in  V^+_{2n}, n \geq 0$ be disjoint intervals inside $\{0\} \times [0,t]$. 
Let $J_1 = \{0\} \times [b_1,b_2], \ldots, J_n = \{0\} \times [b_{2m-1},b_{2m}]$, for 
$\mathbf{b} \in  V^+_{2m}, m \geq 0$,  be disjoint intervals inside $[0,\infty)$. Then 
\begin{equation} \label{duality1}
\mathbb{P}^{C}_{\mu} [ X_e(I_1\cup \ldots \cup I_{n}) =0, \; X_t(J_1\cup \ldots \cup J_{m})=0 ] 
= \mathbb{P}^{A}_{\mathbf{z}} [\mu(S_t) =0]
\end{equation}
where $z_i = (0,t - a_i)$ for $i = 1,\ldots,2n$ and $z_i = (b_{i-2n}, 0)$ for $i = 2n+1, \ldots, 2n+2m$, and where
\[
S_t = (Y^1_t,Y^2_t) \cup   (Y^3_t,Y^4_t) \cup \ldots  
\]
formed from all remaining annihilating particles at time $t$ (and $S_t=\emptyset$ if there are no particles). 
\end{lemma} 

We can obtain a corresponding duality statement for mixed coalescing/annihilating systems CABM($\theta$) 
by thinning. We recall a colouring argument. 
Consider first initial conditions with finitely many particles,
that is $\mu([0,\infty))<\infty$.  Fix an evolution of $(X_t: t \geq 0),X_e$ under $\mathbb{P}^{C}_{\mu}$ 
and colour particles red or blue as follows: at time zero let each particle independently be red $R$ with 
probability $1/(1+\theta)$ and blue $B$ with probability $\theta/(1+\theta)$; colour the particles at later 
times by following the colour change rules
\[
B+B \to B, \qquad R+B \to R, \qquad R+R \to \left\{ \begin{array}{ll}
B & \mbox{with probability $\theta$,} \\
R & \mbox{with probability $1-\theta$,}
\end{array} \right.
\]
independently at each of the finitely many collisions. The particles coloured red form a CABM($\theta$) system.
Moreover, the particles in $X_t$ alive at time $t$ together with the frozen particles 
in $X_e(\{0\}\times [0,t])$ remain independently red $R$ with probability $1/(1+\theta)$ and 
blue $B$ with probability $\theta/(1+\theta)$. This can be shown by checking that after each collision 
this property is preserved. 

Write $\Theta(\mu)$ for the thinned random measure created by deleting each particle of 
$\mu$ independently with probability $1/(1+\theta)$. Writing $\mathbb{P}^{CABM(\theta)}_{\Xi}$ for the law of 
the mixed CABM system started at a (possible random) finite initial condition $\Xi$, the colouring procedure 
above implies the equality in distribution, for finite $\mu$,
\begin{equation} \label{law}
\mbox{$\left( X_t, X_e |_{\{0\}\times [0,t])}\right) $ under $\mathbb{P}^{CABM(\theta)}_{\Theta(\mu)}$}
\;\; \stackrel{\mathcal{D}}{=} \;\;
\mbox{$\left( \Theta(X_t), \Theta(X_e |_{\{0\}\times [0,t])})\right) $ under $\mathbb{P}^{C}_{\mu}$}.
\end{equation}
Thinning a finite set of $n \geq 1$ particles leaves $B(n,(1+\theta)^{-1})$ a Binomial number of remaining particles. 
Note, when $ \theta \in (0,1]$, that
$\mathbb{E} [ (-\theta)^{B(n,(1+\theta)^{-1})}] = 0$ for all $n \geq 1$. Then we have, writing $\mathbb{E}_{\Theta}$ for 
the expectation over the thinning, 
\begin{eqnarray}
&& \hspace{-.3in} 
\mathbb{E}_{\Theta} \mathbb{E}^{CABM(\theta)}_{\Theta(\mu)} 
\left[ (-\theta)^{X_e(I_1\cup \ldots \cup I_{n})} (-\theta)^{X_t(J_1\cup \ldots \cup J_{m})} \right] \nonumber \\
& = & \mathbb{E}_{\Theta} \mathbb{E}^{C}_{\mu} 
\left[ (-\theta)^{\Theta(X_e)(I_1\cup \ldots \cup I_{n})} (-\theta)^{\Theta(X_t)(J_1\cup \ldots \cup J_{m})} \right] 
\quad \mbox{using (\ref{law})}  \nonumber\\
& = & \mathbb{P}^{C}_{\mu} 
\left[ X_e(I_1\cup \ldots \cup I_{n}) =0, \; X_t(J_1\cup \ldots \cup J_{m})=0 \right] \nonumber \\
& = & \mathbb{P}^{A}_{\mathbf{z}} [\mu(S_t) =0] 
\quad \mbox{using (\ref{duality1})}  \nonumber \\
& = & \mathbb{E}_{\Theta} \mathbb{E}^{A}_{\mathbf{z}} \left[ (-\theta)^{\Theta(\mu)(S_t)} \right]. \label{thin10}
\end{eqnarray}
This implies the duality 
\begin{equation} \label{duality2}
\mathbb{E}^{CABM(\theta)}_{\mu} 
\left[ (-\theta)^{X_e(I_1\cup \ldots \cup I_{n})} (-\theta)^{X_t(J_1\cup \ldots \cup J_{m})} \right] =
\mathbb{E}^{A}_{\mathbf{z}} \left[ (-\theta)^{\mu(S_t)} \right]
\end{equation}
for finite $\mu$; this can be checked by induction on the number $\mu(0,\infty)$ of initial particles
by expanding the identity (\ref{thin10}) into the sum over terms where different size subsets of particles 
in $\mu$ remain after the thinning. Note that the duality (\ref{duality2}) contains the duality
(\ref{duality1}) as the limit $\theta \downarrow 0$. Indeed, henceforth we shall take $0^0 =1$ so that
$(-\theta)^k = \I(k=0)$ when $\theta =0$ to allow a unified treatment over $\theta \in [0,1]$.

The extension of (\ref{duality2}) to the case of infinite initial conditions $\mu$ can be established
by approximation arguments. This is (somewhat tersely) sketched in the appendix to \cite{TZ}, 
and we summarize some points here. We can consider the measure $X_t$ as living in the space
$\mathcal{M}$ of locally finite point measures on $[0,\infty)$, which we give the topology of vague convergence.
Due to the instantaneous reactions, we restrict to the subset $\mathcal{M}_0$ of simple point measures.
The arguments in \cite{TZ} show that there is a Feller semigroup on this space, allowing us to construct the law
$\mathbb{P}^{CABM(\theta)}_{\mu}$  for the
CABM($\theta$) process starting from any $\mu \in \mathcal{M}_0$. Moreover there is an entrance law that is the limit
of Poisson$(\lambda)$ initial conditions as $\lambda \uparrow \infty$, which we informally call the maximal entrance 
law. (For the case $\theta =0$ of coalescing particles, this corresponds to the point set process in the Brownian web
starting from the set $[0,\infty)$.) The exit measure $X_e$ also exists under $\mathbb{P}^{CABM(\theta)}_{\mu}$ 
and is the limit of the exit measures for any approximating finite system - the point is the formulae
(\ref{duality2}) characterise the laws of the pair $(X_t, X_e |_{\{0\}\times [0,t])})$ as 
simple locally finite point measures.

\vspace{.1in}
\noindent
\textbf{Remark.} We digress here to record a lemma, based on the same tools, that shows thinnings are useful in the 
study of massive CBMs, where masses add upon coalescence. 
This process yields a point process in position/mass space $\mathbb{R} \times \mathbb{N}_0$
(in the case of integer masses). 
The lemma below gives only a very partial description of the 
process, and a full tractable
multi-particle description for this model is still lacking (though notably the one point distribution 
has been obtained \cite{spouge}).
\begin{lemma} \label{MCBM}
Consider a system $(X_t)$ of massive coalescing Brownian motions where each particle has an integer mass and masses
add upon coalesence.  
Suppose the masses of particles at $t=0$
are independent uniform random variable on $\{1,2,\ldots, q \}$, for a fixed $q \in \{2,3,\ldots\}$.
Let $(R_t)$ be the positions of particles present at time $t$ with labels not divisible by $q$; 
let $(B_t)$ be the positions of particles present at time $t$ with labels divisible by $q$.
Then the process $(R_t,B_t, t\geq 0)$ is a two-species particle system, where
the evolution of types at a collision time is governed by the following rules: 
for $\theta = 1/(q-1)$
\begin{equation} \label{2species}
B+B \rightarrow B, \quad B+R \rightarrow R, \quad 
R+R \stackrel{\theta}{\rightarrow} B, \quad  
R+R \stackrel{1-\theta}{\rightarrow} R.
\end{equation}
Moreover, at a fixed $t \geq 0$, the positions $(B_t)$ are a $1/q$ thinning,
and the positions $(R_t)$ are a $(q-1)/q$ thinning, of the 
positions of the full system $(X_t)$.
\end{lemma}
We have not detailed the initial positions or state space for $(X_t)$ which play no part in
the simple proof - consider finite systems on $\mathbb{R}$ to be specific. 
The proof consists of checking that for two colliding particles, whose masses are independent and 
uniform modulo($q$) on $\{1,\ldots,q\}$, the resultant coalesced particle still has mass that is 
uniform modulo($q$) on $\{1,\ldots,q\}$. This then implies that the $(R_t)$ system evolves
as a CABM($\theta$) system as in (\ref{2species}). 

The point of including this lemma is to show that both strong thinning ($p<1/2$) and weak thinning ($p \geq 1/2$)
of CBM has an interpretation in terms of interacting particle systems:
here weak thinning with probability $(q-1)/q$ singles out particles with masses not divisible by $q$, whereas
strong thinning with probability $1/q$ singles out particles with masses divisible by $q$. 
\subsection{Pfaffian kernels}  \label{s5.3}
We can now read off the Pfaffian kernels in Theorem \ref{Exittheorem} from the duality $(\ref{duality2})$ and
the alternating product moment formulae in Lemma \ref{apmlemma}. 
For $\mu \in \mathcal{M}_0$ and $\theta \in (0,1]$, the duality $(\ref{duality2})$ in the case where 
$\mathbf{b} = \emptyset$, so that we are only interested in the exit measure, and when $t = a_{2n}$, gives
\[
\mathbb{E}^{CABM(\theta)}_{\mu} 
\left[ (-\theta)^{X_e(I_1\cup \ldots \cup I_{n})} \right] = \mathbb{E}^{A}_{\mathbf{z}} \left[ (-\theta)^{\mu(S_{a_{2n}})} \right]
\]
where $z_i = (0,a_{2n}-a_i)$ for $i = 1, \ldots, 2n$. The right hand side is an alternating product moment
(\ref{apm}) where $g(x) = (-\theta)^{\mu[0,x]}$ and $h(x) = (-\theta)^{-\mu[0,x]}$. When $\mu$ is finite then $g,h$ are bounded and 
Lemma {\ref{apmlemma} gives 
\[
\mathbb{E}^{CABM(\theta)}_{\mu} 
\left[ (-\theta)^{X_e(I_1\cup \ldots \cup I_{n})} \right] = \pf (H(a_i,a_j): i < j \leq 2n)
\]
where
\[
H(a_i,a_j) = 
\mathbb{E}^{A}_{(z_i,z_j)} \left[ (-\theta)^{\mu(S_{a_{2n}})} \right] = 
\mathbb{E}^{A}_{(0,0),(0,a_j-a_i)} \left[ (-\theta)^{\mu(S_{a_{j}})} \right].
\]
The same conclusion holds for $\mu \in \mathcal{M}_0$ by taking limits $\mu |_{[0,n]} \to \mu$. 
To derive the correlation function we differentiate in the variables $a_2,\ldots,a_{2n}$ and then
let $a_1 \uparrow a_2, \ldots, a_{2n-1} \uparrow a_{2n}$ (details of this calculation are given in \cite{TZ}). 
We reach, writing $\rho^{CABM(\theta)}_n$ for the $n$ point intensity of $X_e$ under 
$\mathbb{E}^{CABM(\theta)}_{\mu}$, 
\begin{eqnarray*}
&& \hspace{-.3in} (-(1+\theta))^n \rho^{CABM(\theta)}_n (a_2,a_4,\ldots,a_{2n}) \\
&= & \pf \left(
\left( \begin{array}{cc}
H(a_{2i},a_{2j}) & D_2 H(a_{2i},a_{2j}) \\
D_1 H(a_{2i},a_{2j}) & D_{12} H(a_{2i},a_{2j}) 
\end{array} \right): i < j \leq n \right).
\end{eqnarray*}
To massage this into the stated derived form in Theorem \ref{Exittheorem} we 
first conjugate the kernel with the block matrix $A$ with entries $\pm 1$ down the diagonal (using
$\pf(A^TBA) = (-1)^n$), and then define $K = H-1$ to allow for the jump discontinuity in the derived form
(\ref{derived_form}). This leads to  $ \rho^{CABM(\theta)}_n (t_1,\ldots,t_{n}) = \Pf( \mathbf{K}(t_i,t_j): i<j\leq n)$
where
\begin{equation} \label{mukernel}
K(s,t) = \frac{1}{1+\theta} \;
\mathbb{E}^{A}_{(0,0),(0,t-s)} \left[ \left((-\theta)^{\mu(Y^2_t - Y^1_t)}-1\right) \I(\tau > t) \right] \quad \mbox{for $0<s<t$,}
\end{equation}
where $\tau$ is the hitting time of the pair $Y^1,Y^2$. 

For the cases where the initial condition is a Poisson measure $\Xi$, with bounded intensity $\lambda(x) dx$, we restart with the 
duality $(\ref{duality2})$, which gives
\begin{eqnarray*}
\mathbb{E}^{CABM(\theta)}_{\Xi} 
\left[ (-\theta)^{X_e(I_1\cup \ldots \cup I_{n})}  \right] 
&=&
\mathbb{E}_{\Xi} \mathbb{E}^{A}_{\mathbf{z}} \left[ (-\theta)^{X_0(S_t)} \right] \\
&=& \mathbb{E}^{A}_{\mathbf{z}} \bigg[ e^{ - (1+\theta) \sum_{i} \int^{Y^{2i}_t}_{Y^{2i-1}_t} \lambda(z) dz} \bigg] \\
& = & \mathbb{E}^{A}_{\mathbf{z}} [ M_{(g,h)}(Y_t)]
\end{eqnarray*}
for $g(x) = \exp( - (1+ \theta) \int^x_0 \lambda(z) dz)$ and $h(x) = 1/g(x)$. When $\lambda$ is compactly supported $g,h$ are bounded, 
and the product moment Lemma \ref{apmlemma} gives
\begin{equation} \label{poissonduality}
\mathbb{E}^{CABM(\theta)}_{\Xi} 
\left[ (-\theta)^{X_e(I_1\cup \ldots \cup I_{n})} \right] 
= \pf (H(a_i,a_j): i < j \leq 2n)
\end{equation}
where 
\[
H(a_i,a_j) = 
\mathbb{E}^{A}_{(z_i,z_j)} \bigg[ e^{ - (1+\theta) \int_{S_t} \lambda(z) dz} \bigg] = 
\mathbb{E}^{A}_{(0,0),(0,a_j-a_i)} \bigg[ e^{ - (1+\theta) \int_{S_{a_j}} \lambda(z) dz} \bigg].
\]
Approximating $\lambda I[0,n] \to \lambda$ allows the same conclusion for bounded $\lambda$. 
Now we repeat the steps above to extract the kernel showing that, under  $\mathbb{E}^{CABM(\theta)}_{\mu} $
the exit measure $X_e$ is a Pfaffian point process with kernel $\mathbf{K}(s,t)$ in derived form based 
on the scalar kernel
\begin{equation} \label{Pkernel}
K(s,t) = \frac{1}{1+\theta} \;
\mathbb{E}^{A}_{(0,0),(0,t-s)} \bigg[ \bigg(e^{ - (1+\theta) \int_{Y^1_t}^{Y^2_t} \lambda(z) dz}-1\bigg)  \I(\tau > t)  \bigg]
\end{equation}
where $\tau$ is the hitting time of the pair $Y^1,Y^2$. 
The distribution for $(Y^1_t,Y^2_t)$ is given (for example by conditioning at time $t-s$ and then using the 
Karlin McGrgeor formula for non-colliding Markov processes) by
\begin{equation} 
 \mathbb{P}^{A}_{(0,0),(0,t-s)} [ Y^1_t \in dy_1, Y^2_t \in dy_2, \tau >t] \nonumber \\
=
\left| \!\! \begin{array}{cc} 
p^R_s(0,y_1)  &  p^R_t(0,y_1) \\
p^R_s(0,y_2) &   p^R_t(0,y_2) 
\end{array} \!\! \right| dy_1 dy_2,\label{2pttd}
\end{equation}
(recall $p^R_t(x,y)$ is the transition density for reflected Brownian motion on $[0,\infty)$). 
Rewriting the kernels using this density gives the form stated in Theorem \ref{Exittheorem}. Note that 
all derivatives exist of $K(s,t)$ exist and are bounded in the region $0 < s \leq t$. 
The density in (\ref{2pttd}) is of the form $\phi(s) \psi(t) - \psi(s) \phi(t)$ and this 
allows one to check that the antisymmetric extension of $K(s,t)$ to $s,t > 0$  is a $\mathcal{C}^2$.
For deterministic locally finite $\mu$, or Poisson $\Xi$ with a smooth intensity $\lambda$, the kernel extends 
to a $\mathcal{C}^2$ function on $s,t \geq 0$. 
\section{Further proofs}  \label{s6}
Here we collect proofs of more technical statements made in Sections
\ref{s2}, \ref{s2.2}.
\subsection{Proof of Fourier transform formulae for $\kappa_1(p),\kappa_2(p)$} \label{s6.1}
We derive the formulae (\ref{k1FT}) and (\ref{k2FT}). Recall we are assuming, for simplicity,
that $\rho$ is in Schwarz class. 
Since $\rho$ is symmetric, then $\hat{\rho}$ is symmetric and real valued. For $k \in \mathbb{R}$, 
 $\hat{\rho}(k)< 1$ for all $k \neq 0$ and decays faster than polynomially, and near zero
has an expansion
\begin{equation} \label{ska}
\hat{\rho}(k) = 1 - \frac{\sigma^2 k^2}{2} + O(|k|^4).
\end{equation}
These imply that the integral (\ref{Ldefn}) defining $L_{\rho}(p,x)$ is well defined
and absolutely integrable for all $p$.

Fourier inversion and Fubini's Theorem imply that
\begin{eqnarray*}
\sum_{n=1}^{\infty} \frac{(4p(1-p))^n}{n} \rho^{*n}(0) 
& = & \sum_{n=1}^{\infty} \frac{(4p(1-p))^n}{2 \pi n} \int_{\mathbb{R}} (\hat{\rho}(k))^n  dk \\
& = &   - \frac{1}{2 \pi} \int_{\mathbb{R}} \log (1 - 4p(1-p) \hat{\rho}(k))  dk = - L_{\rho}(p,0).
\end{eqnarray*}
which completes identity (\ref{k1FT}) for $\kappa_1(p)$. Similarly, when $p \neq \frac12$,
\[
 \int_0^{\infty} \! x\, \Big( \sum_{n=1}^{\infty} \frac{(4p(1-p))^n  \tilde{\rho}^{*n}(x)}{n} \Big)^2 dx
 = \int_0^{\infty} \! x\, L^2_{\rho}(p,x) dx
 \]
and 
\begin{eqnarray*}
&& \hspace{-.4in} \sum_{n=1}^{\infty}  \frac{(4p(1-p))^n}{n} \int_{-\infty}^0  e^{\phi_p x} \rho^{*n}(x) dx \\
& = & - \frac{1}{2 \pi} \int_{-\infty}^0  \int_{\mathbb{R}} e^{(\phi_p-ik)x} \log(1-4p(1-p) \hat{\rho}(k)) dk dx \\
& = &  - \frac{1}{2 \pi}  \int_{\mathbb{R}} \frac{\phi_p}{\phi_p^2+k^2} \log(1-4p(1-p) \hat{\rho}(k)) dk
\end{eqnarray*}
completing the formula (\ref{k2FT}) for $\kappa_2(p)$ in the cases $p \neq \frac12$. 

The case $p=1/2$ needs some care. Recall that we suppose that $\rho$ has a finite exponential moment. Therefore
$\hat{\rho}(k)$ is analytic in a strip $|k| < 2 \mu$ and the small $k$ expansion (\ref{ska}) also 
holds for $k \in \mathbb{C}$. 
By (\ref{1011}) the infinite series for $\kappa_2(1/2)$ is
absolutely convergent so that 
\begin{eqnarray}
~~~~~~~\kappa_2\left(\fr\right) - \log 2 & = & 
\frac{1}{2}\sum_{n=2}^{\infty}  \left( \sum_{k=1}^{n-1}
\int_0^{\infty} x \frac{\rho^{*k}(x) 
\rho^{*(n-k)}(x)}{k(n-k)} dx - \frac{1}{2n} \right) - \frac14 \label{temp317}  \\
& = & \lim_{\epsilon \downarrow 0}
\frac{1}{2}\sum_{n=2}^{\infty} (1- \epsilon)^n \left( \sum_{k=1}^{n-1}
\int_0^{\infty} x \frac{\rho^{*k}(x) 
\rho^{*(n-k)}(x)}{k(n-k)} dx - \frac{1}{2n} \right) - \frac14 \nonumber \\
& = & \lim_{\epsilon \downarrow 0} \frac{\log \epsilon}{4} + \frac{1}{8 \pi^2} \int_0^{\infty} x \left( \int_{\mathbb{R}} e^{-ikx} 
\log(1-(1-\epsilon) \hat{\rho}(k) dk\right)^2 dx \nonumber  \\
& = & \lim_{\epsilon \downarrow 0} \frac{\log \epsilon}{4}- \frac{(1-\epsilon)^2}{8 \pi^2} 
\int_0^{\infty} \! x^{-1} \! \left( \int_{\mathbb{R}} e^{-ikx} 
\frac{\hat{\rho}'(k)}{1-(1-\epsilon) \hat{\rho}(k)} dk\right)^2 \! dx  \nonumber 
\end{eqnarray}
where we have integrated by parts in the $dk$ integral to help understand the divergence in $\epsilon$.
Indeed the function 
\begin{equation} \label{fdefn}
f_{\epsilon}(x) = \frac{1}{2 \pi i} \int_{\mathbb{R}} e^{-ikx} \frac{\hat{\rho}'(k)}{1-(1-\epsilon) \hat{\rho}(k)} dk
\end{equation}
has an integrand with two poles that approaches the real axis as $\epsilon \downarrow 0$ and this will
lead to the cancellation of the term $\frac14 \log \epsilon$. 
The asymptotics (\ref{ska}) allow us to fix $\mu>0$ so that 
the denominator $1-(1-\epsilon) \hat{\rho}(k)$ has, for small enough $\epsilon$, 
only two zeros on $|k| \leq \mu$, at 
$\pm r_{\epsilon} i$, where 
\begin{equation} \label{rdefn}
r_{\epsilon} = \frac{\sqrt{2 \epsilon}}{\sigma} + O(\epsilon).
\end{equation}
We move the contour defining $f_{\epsilon}$ from the real axis to the curve
$C_{\mu}$ consisting of the segments $(-\infty,-\mu)$, $(\mu, \infty)$ on the real axis and the half circle
$\{-\mu e^{it}: t \in [0,\pi]\}$. This move crosses the 
the pole at $-r_{\epsilon}i$ so that, evaluating the residue at $-r_{\epsilon}i$, we have
\begin{equation} \label{tfdefn}
f_{\epsilon}(x) = \frac{1}{1-\epsilon} e^{-r_{\epsilon} x} + \tilde{f}_{\epsilon}(x), \quad
\mbox{where} \quad \tilde{f}_{\epsilon}(x) =  \frac{1}{2 \pi i} \int_{C_{\mu}} e^{-ikx} \frac{\hat{\rho}'(k)}{1-(1-\epsilon) \hat{\rho}(k)} dk.
\end{equation}  
Substituting this into the the expression (\ref{temp317}) we find
\begin{eqnarray*}
&& \hspace{-.4in} \kappa_2(1/2) - \log 2 \\ 
& = &
\lim_{\epsilon \downarrow 0} \frac{\log \epsilon}{4} +  \frac{(1-\epsilon)^2}{2} 
\int_0^{\infty} \! x^{-1} f_{\epsilon}^2(x) dx \\
& = &
\lim_{\epsilon \downarrow 0} \frac{\log \epsilon}{4} - (1-\epsilon)^2
\int_0^{\infty} \log x f_{\epsilon}(x) f_{\epsilon}'(x) dx \\
& = & \lim_{\epsilon \downarrow 0} \frac{\log \epsilon}{4} - (1-\epsilon)^2
\int_0^{\infty} \log x \left(\frac{1}{1-\epsilon} e^{-r_{\epsilon} x} + \tilde{f}_{\epsilon}(x)\right) 
 \left(\frac{-r_{\epsilon}}{1-\epsilon} e^{-r_{\epsilon} x} + \tilde{f}'_{\epsilon}(x)\right) dx \\
  & = &  \frac14 \log \big(\frac{\sigma^2}{8}\big) - \frac{\gamma}{2}  \qquad \quad \mbox{(recall 
  $\gamma = - \int^{\infty}_0 e^{-x} \log x dx$)} \\
&&  - \lim_{\epsilon \downarrow 0} (1-\epsilon)^2
\int_0^{\infty} \log x  \left( - \tilde{f}_{\epsilon}(x)
\frac{r_{\epsilon}}{1-\epsilon} e^{-r_{\epsilon} x} +
\left(\frac{1}{1-\epsilon} e^{-r_{\epsilon} x} + \tilde{f}_{\epsilon}(x)\right) \tilde{f}'_{\epsilon}(x)\right) dx \\
&=& \frac14 \log \big(\frac{\sigma^2}{8}\big) - \frac{\gamma}{2} - 
\int_0^{\infty} \log x 
\left(1 + \tilde{f}_{0}(x)\right) \tilde{f}'_{0}(x) dx.
\end{eqnarray*}
To justify passing to the limit in the last equality one can verify (by integrating by parts 
in the definition of $\tilde{f}_{\epsilon}$ and $\tilde{f}'_{\epsilon}$ and noting that 
$1-\hat{\rho}$ does not vanish on $C_{\mu}$) that there exists $\epsilon_0,C$ so that 
$|\tilde{f}_{\epsilon}(x)| \vee |\tilde{f}_{\epsilon}'(x)| \leq C (1+x^2)^{-1}$ for all
$x \geq 0$ and all $0 \leq \epsilon \leq \epsilon_0$. 

The last step is to rewrite $\tilde{f}_0(x)$ in terms of $L_{\rho}$. We move the contour of integration
in 
$\tilde{f}_0(x) =  \frac{1}{2 \pi i} \int_{C_{\mu}} e^{-ikx} \frac{\hat{\rho}'(k)}{1- \hat{\rho}(k)} dk$
back to the real line. The integrand $\frac{\hat{\rho}'(k)}{1- \hat{\rho}(k)}$
has a simple the pole at the origin, so that letting $\mu \downarrow 0$ we get half the residue at the origin 
and the principle value for the integral around the origin, that is 
\[
\tilde{f}_0(x) =  \frac{1}{2 \pi i} \int_{C_{\mu}} e^{-ikx} \frac{\hat{\rho}'(k)}{1- \hat{\rho}(k)} dk
= 1 + \frac{1}{2 \pi i} P.V. \int_{\mathbb{R}} e^{-ikx} \frac{\hat{\rho}'(k)}{1- \hat{\rho}(k)} dk.
\]
It is not hard to check that one may integrate by parts to identify
\[
\frac{1}{2 \pi i} P.V. \int_{\mathbb{R}} e^{-ikx} \frac{\hat{\rho}'(k)}{1- \hat{\rho}(k)} dk
= - x L_{\rho}(1/2,x)
\]
completing the proof. 
\subsection{Regularity of $p \to \kappa_i(p)$ for the Gaussian kernel} \label{s6.2}
We complete the proof of Corollary \ref{gaussiankernel}. We recall the expression 
(\ref{k1C1}) for $\kappa_1(p)$: 
\begin{equation} \label{k1C1+}
\kappa_1(p) 
= \frac{1}{4\sqrt{\pi t}} \mbox{Li}_{3/2}(4p(1-p))
+  \I(p >1/2) \left( - t^{-1} \log 4p(1-p)\right)^{1/2}. 
\end{equation}
This shows that $\kappa_1(p)$
 is a smooth function of $p \in (0,\frac12) \cup (\frac12,1)$. To examine the behaviour at $p=\frac12$
 we use a series representation (section 9 of \cite{Wood}) for $\mbox{Li}_s$, for $s \not = 1,2,\ldots$,
\begin{equation}
\label{lismart}
\mbox{Li}_s(\beta)=\Gamma(1-s)\left(-\log(\beta)\right)^{s-1}+\sum_{n=0}^\infty \zeta(s-n)
\frac{\log^n(\beta)}{n!},
\end{equation}
where $\zeta$ is Riemann's zeta function and the infinite series converges for $|\log(\beta)| < 2 \pi$. 
Using this for $s=3/2$ and taking  $p = \frac12 + \epsilon$ in (\ref{k1C1+}) we reach
\[
t^{1/2} \kappa_1\left(\frac12 + \epsilon\right) = \frac12 \sqrt{ - \log(1-4 \epsilon^2)} \, \sgn(\epsilon) + A(\epsilon)
\]
where $A$ is analytic for $\epsilon \in (-\frac12, \frac12)$ and given by 
$A(\epsilon) = \frac{1}{4 \sqrt{\pi}}  \sum_{n=0}^\infty \zeta \left(\frac32 -n\right)
\frac{\log^n(1-4 \epsilon^2)}{n!}$.  Also 
\[
\frac12 \sqrt{ - \log(1-4 \epsilon^2)} \sgn(\epsilon) 
 =  \epsilon \Psi^{1/2}(4 \epsilon^2) \quad 
 \mbox{where $\Psi(z) = - \frac{\log(1-z)}{z}$}
\]
showing that $\kappa_1(p)$ is analytic for $p \in (0,1)$. 

The infinite series in (\ref{DZ50}) and (\ref{DZ53}) for $\kappa_2(p)$, together with their derivatives in $p$, converge uniformly 
for $p$ in compacts inside $[0,\frac12) \cup (\frac12,1)$.  
This implies the continuous differentiability of  $\kappa_2$
except at the point $p=1/2$. For $p< \frac12$ the formula (\ref{DZ50}) can be re-written as
\[
 \log \frac{1}{1-p} - p(1-p) + \frac{1}{4 \pi} \sum_{n=2}^{\infty} \frac{(4p(1-p))^n}{n} \Big(
 \sum_{k=1}^{n-1} \frac{1}{\sqrt{k(n-k)}} - \pi \Big).
\]
The absolute convergence of the sum over $n$ for all $p$,  using 
$\sum_{k=1}^{n-1}\frac{1}{\sqrt{k(n-k)}}-\pi =O(n^{-1/2})$, 
implies the left continuity of
$\kappa_2(p)$ as $p \uparrow \frac12$. 
A straightforward re-arrangement of the terms in (\ref{DZ53}) 
constituting $\kappa_2(p)$ for $p>1/2$ leads to
\begin{equation} \label{rtcty}
\lim_{p\downarrow \fr} \kappa_2(p)=\kappa_2(1/2)-2\log 2-
\lim_{\delta \downarrow 0} \Big( \log \delta+\sum_{n=1}^\infty \frac{1}{n}
\mathrm{erfc} (\sqrt{n\delta}) \Big)
\end{equation}
(using the complementary error function $\mathrm{erfc}(x) = \frac{2}{\sqrt{\pi}} \int_x^{\infty} \exp(-t^2) dt 
= 1- \mbox{erf}(x)$). 
To compute the limit in the right hand side, we fix $c>0$ and write
\begin{eqnarray}
&& \hspace{-.2in} \lim_{\delta \downarrow 0} \Big( \log \delta+\sum_{n=1}^\infty \frac{1}{n}
\mathrm{erfc} (\sqrt{n\delta}) \Big) \nonumber \\
&=& \lim_{ \delta \downarrow 0} \Big( \log \delta + \sum_{n  \geq c/\delta}\frac{1}{n}
\mathrm{erfc} (\sqrt{n\delta})  + \sum_{n < c/\delta} \frac{1}{n} 
(1-\mbox{erf} (\sqrt{n \delta})) \Big) \nonumber  \\
& = & \log c +\gamma+
\lim_{ \delta \downarrow 0} \Big( \sum_{n  \geq c/\delta}\frac{1}{n}
\mathrm{erfc} (\sqrt{n\delta})  - \sum_{n < c/\delta} 
\frac{1}{n} \mbox{erf} (\sqrt{n \delta}) \Big) \nonumber  \\
& = & 
\log c +\gamma +\int_c^\infty \frac{1}{x} \mathrm{erfc}(\sqrt{x}) dx
 - \lim_{ \delta \downarrow 0} \mathcal{E}_{c,\delta} \label{temp130}
\end{eqnarray}
where we have used $\sum_{n=1}^N \frac{1}{n}=\log N+\gamma+O\left(N^{-1}\right)$ 
for $\gamma$ the Euler-Mascheroni constant, and we may estimate (using $\mbox{erf}(x) \leq x$)
\[
0 \leq \mathcal{E}_{c,\delta} = \sum_{n < c/\delta} \frac{1}{n} \mbox{erf} (\sqrt{n \delta})
\leq  \sqrt{\delta}  \sum_{n < c/\delta} n^{-1/2} \leq  2 \sqrt{c}.
\]
Integrating by parts,
\[
\int_c^\infty \frac{1}{x} \mathrm{erfc}(\sqrt{x}) dx
=  \frac{1}{\sqrt{\pi}} \int_c^{\infty}
\frac{\log x}{\sqrt{x}} e^{-x} dx - \mathrm{erfc}(\sqrt{c}) \log c.
\]
Therefore, taking the limit as $c \downarrow 0$ in (\ref{temp130}), 
\[
\lim_{\delta \downarrow 0} \Big( \log \delta+\sum_{n=1}^\infty \frac{1}{n}
\mathrm{erfc} (\sqrt{n\delta}) \Big)
=\gamma+  \frac{1}{\sqrt{\pi}} \int_0^{\infty}\frac{\log x}{\sqrt{x}} e^{-x} dx = -2 \log 2
\]
using the known special value of the digamma function 
$ \psi^{(0)}(\frac12) =  \frac{1}{\sqrt{\pi}} \int_0^{\infty}\frac{\log x}{\sqrt{x}} e^{-x} dx$.
From (\ref{rtcty}), the right continuity of $\kappa_2$ at $1/2$ is proved.

We may directly calculate $\kappa'_2(p)$ for $p \in [0,\frac12) \cup (\frac12,1)$ from the formulae
(\ref{DZ50}) and (\ref{DZ53}), leading to (and writing $\beta_p = 4p(1-p)$)

\begin{equation} \label{kappap}
\kappa_2'(p)=
\left\{
\begin{array}{ll}
(2p-1) + \frac{1}{1-p} +\frac{1-2p}{\pi} \left(\frac{1}{\beta_p}\mbox{Li}_{\fr}^2(\beta_p)-\frac{\pi \beta_p}{1-\beta_p} \right) &p<\frac12,\\
(2p-1) + \frac{1}{1-p} +\frac{1-2p}{\pi}\left(\frac{1}{\beta_p}\mbox{Li}_{\fr}^2(\beta_p)-\frac{\pi \beta_p}{1-\beta_p} \right) & \\
\hspace{.2in} -\frac{1-2p}{\beta_p \sqrt{-\log \beta_p}}\left(\frac{1}{\sqrt{\pi}}\mbox{Li}_{\fr}(\beta_p) - \frac{1}{\sqrt{-\log \beta_p}} 
\right)&p>\frac12.
\end{array}
\right.
\end{equation}
Using the series representation (\ref{lismart}) for $L_{\frac12}(\beta)$
and computing the limits one finds
\[
\lim_{p\downarrow \fr} \kappa_2'(p)=2+\frac{2}{\sqrt{\pi}} \zeta(1/2)
=\lim_{p\uparrow \fr} \kappa_2'(p)
\]
which establishes the continuous differentiability of $\kappa_2$ at $\frac12$.

\subsection{Proof of error bounds for $p=\frac12$ asymptotics} \label{s6.3}
We give the proofs of the error bound for the asymptotic (\ref{1011}),
the error bounds (\ref{E10}),(\ref{E20}) and 
their analogues needed for the non translationally invariant case in (\ref{P1+}), (\ref{P2+}).

We use a local central limit theorem in the form (see Theorem 2 of XVI.2 \cite{feller} using the 
symmetry of $\rho$ to imply the third moment $\mu_3$ is zero) 
\begin{equation} \label{lclt}
\left|\rho^{*n}(x) - g_{n}(x) \right| \leq C n^{-3/2}  \quad \mbox{for all $n \geq 1, x \in \mathbb{R}$.}
\end{equation} 
for the Gaussian density $g_t(x) = (2 \pi \sigma^2 t)^{-1/2} \exp(-x^2/2 \sigma^2 t)$. We therefore approximate
\begin{eqnarray}
 \Kac_{\rho}(n) &=& \frac{n}{2} \int_0^{\infty} x 
 \sum_{k=1}^{n-1} \frac{\rho^{*k}(x)\rho^{*(n-k)}(x)}{k(n-k)} dx  \nonumber \\
 & = & \frac{n}{2} \int_0^{\infty} x  \sum_{k=1}^{n-1} \frac{g_{k}(x)g_{n-k}(x)}{k(n-k)} dx  + \mathcal{E}_n 
 \nonumber \\
 & = & \frac{1}{4\pi} \sum_{k=1}^{n-1} \frac{1}{\sqrt{k(n-k)}} + \mathcal{E}_n. 
 \label{2666}
 \end{eqnarray}
The error $\mathcal{E}_n$ is bounded by
 \begin{eqnarray*}
 \mathcal{E}_n  & \leq & n \sum_{1 \leq k \leq n/2} \int_0^{\infty} \frac{x}{k(n-k)}
 \left|\rho^{*k}(x)\rho^{*(n-k)}(x) - g_{k}(x)g_{n-k}(x) \right| dx. 
 \end{eqnarray*}
Letting $C$ depend on $\sigma$ and vary from line to line, we bound the 
sum over $k \leq n^{1/2}$, using $\rho^{(n-k)*}(x) \leq C n^{-1/2}$ and 
$g_{n-k}(x) \leq C n^{-1/2}$, by 
\[
C n^{1/2} \sum_{1 \leq k \leq n^{1/2}} \int_0^{\infty} \frac{x}{k(n-k)} (\rho^{*k}(x) + g_{k}(x))dx 
\leq C n^{-1/2} \sum_{1 \leq k \leq n^{1/2}} \frac{1}{k^{1/2}} = O(n^{-1/4}),
\]
and the sum over $n^{1/2}< k \leq n/2$ using (\ref{lclt}) by 
 \begin{eqnarray*}
&& \hspace{-.3in} C n \sum_{n^{1/2} < k \leq n/2} \int_0^{\infty} \frac{x}{k(n-k)} 
\left( \rho^{(n-k)*} |\rho^{*k}-  g_{k}|  + g_{k}|  \rho^{*(n-k)}-  g_{n-k}|\right) (x) dx \\
& \leq & C n \sum_{n^{1/2} < k \leq n/2} \int_0^{\infty} \frac{x}{k(n-k)} 
\left( \rho^{(n-k)*}(x) k^{-3/2} + g_{k}(x) n^{-3/2} \right) dx \\
& \leq & C n \sum_{n^{1/2} < k \leq n/2} \left( \frac{1}{n^{1/2} k^{5/2}} + \frac{1}{n^{5/2} k^{1/2}}\right) = O(n^{-1/4}).
 \end{eqnarray*}
The sum in (\ref{2666}) is a Riemann approximation to the Beta integral $(4 \pi)^{-1} B(\frac12,\frac12) = \frac14$ with a 
further error that is $O(n^{-1/2})$. This completes the asymptotic (\ref{1011}).

For the error term bound (\ref{k1FT}) we start with
\begin{eqnarray*}
E^{(1)}(n,L) & = & p(n,L) - \frac1n \E_{0} \left[ \delta_{0}(S_n) (M_n - m_n)  \right]  \\
& = & \frac1n \E_{0} \left[ \delta_{0}(S_n) \min\{L,M_n - m_n\}  \right] 
 - \frac1n \E_{0} \left[ \delta_{0}(S_n) (M_n - m_n)  \right]  \\
 & = & - \frac1n \E_{0} \left[ \delta_{0}(S_n) ((M_n - m_n)-L)_+  \right] 
\end{eqnarray*}
so that, using $(a+b)_+ \leq a_+ + b_+$, 
\begin{eqnarray*}
|E^{(1)}(n,2L) | & \leq & \frac1n \E_{0} \left[ \delta_{0}(S_n) (M_n -L)_+  \right] + \frac1n
\E_{0} \left[ \delta_{0}(S_n) (- m_n-L)_+  \right] \\
& = & \frac{2}{n} \E_{0} \left[ \delta_{0}(S_n) (M_n -L)_+  \right]\\
& = & \frac{2}{n} \E_{0} \left[ \delta_{0}(S_n) (M_{n-1} -L)_+  \right] \\
& \leq & \frac{2 \|\rho\|_{\infty}}{n} \E_{0} \left[(M_{n-1} -L)_+  \right] \\
& \leq &  \frac{2 \|\rho\|_{\infty}}{n} \E_{0} \left[(M_{n} -L)_+  \right]
\end{eqnarray*}
where in the penultimate step we bounded the density of the single step $S_n-S_{n-1}$. Then
\begin{eqnarray*}
\E_{0} \left[(M_{n} -L)_+  \right] & \leq & C L^{-3}  \E_{0} \left[|M_{n}|^4 \right] \\
& \leq & C L^{-3}  \E_{0} \left[|S_{n}|^4 \right] \quad \mbox{by Doob's inequality,} \\
& \leq & C L^{-3}  n^2 \quad \mbox{by a Marcinkiewicz-Zygmund inequality,}
\end{eqnarray*}
which completes the proof of (\ref{E10}). 
The changes needed for $\tilde{E}^{(1)}(n,L)$ in the non translationally invariant case are minor:
the only new term that arises is
\[
\E_{0} \left[ \delta_{0}(S_n) (\tilde{M}_n -L)_+  \right] \leq \|\rho\|_{\infty} 
\E_{0} \left[ (\tilde{M}_n -L)_+  \right] 
\] 
(by again averaging over the final step $S_n-\tilde{M}_n = \mathcal{Y}_n$). It is not
difficult to check that the bound $\E_{0} \left[(\tilde{M}_{n} -L)_+  \right] \leq C L^{-3}  n^2$
still holds.  

For the second error term $E^{(2)}(n,L)$ we use a Skorokhod embedding of the walk into a Brownian 
motion $(W(t):t \geq 0)$ run at speed $\sigma^2$. We choose 
stopping times $(T_1,T_2,\ldots)$ so that 
$(S_1,S_2,\ldots) \stackrel{\mathcal{D}}{=} (W(T_1),W(T_2),\ldots)$ and so that 
$T_1,(T_2-T_1),(T_3-T_2),\ldots$ are  an i.i.d. set of non-negative variables with 
$\E[T_k] = k$ and $\E[(T_k-T_{k-1})^2] \leq 4 \E[ X_1^4]/\sigma^4 < \infty$.
Let $\hat{n} = \lceil n - n^{\alpha} \rceil$ for some $\alpha \in (\frac12,1)$ and let
\[
\Omega_n = \{ \max_{k \leq \hat{n}} |T_k- k| \leq n^{\beta}\}
\] 
for some $\beta \in (\frac12,\alpha)$. Since $(T_k-k)$ is a square integrable martingale, Doob's inequality
implies that $P[\Omega^c_n] = O(n^{1-2\beta})$.
Then we make the following approximations:
\begin{eqnarray}
&& \hspace{-.3in} \E_{0} \left[ \delta_{0}(S_n) \min\{L,M_n - m_n\}  \right] \nonumber \\
& = & \E_{0} \left[ \delta_{0}(S_n) \min\{L,M_{\hat{n}} - m_{\hat{n}}\}\right] + \mathcal{E}_1  \nonumber \\
& = & \E_{0} \left[ \delta_{0}(S_n) \min\{L,M_{\hat{n}} - m_{\hat{n}}\}; \Omega_n   \right] + \mathcal{E}_2  \nonumber \\
& = & \E_{0} \left[ \delta_{0}(W(n)) \min\{L,M_{\hat{n}} - m_{\hat{n}}\}; \Omega_n   \right] + \mathcal{E}_3  \nonumber \\
& = & \E_{0} \left[ \delta_{0}(W(n)) \min\{L,W^*(\hat{n}) - W_*(\hat{n})\}; \Omega_n   \right] + \mathcal{E}_4  \nonumber \\
& = & \E_{0} \left[ \delta_{0}(W(n)) \min\{L,W^*(\hat{n}) - W_*(\hat{n})\} \right]  + \mathcal{E}_5 \nonumber \\
& = & \E_{0} \left[ \delta_{0}(W(n)) \min\{L,W^*(n) - W_*(n)\}  \right] + \mathcal{E}_6 \label{aaargh}
\end{eqnarray}
Thus we aim to estimate $|E^{(2)}(n,L)| = \frac1n |\mathcal{E}_6|$.
The reason for this slightly messy set of approximations is in order to be able to use the 
local central limit theorem to estimate the expectation of the delta functions 
$\delta_0(S_n)$, $\delta_{0}(W(n))$
by conditioning at an earlier time.

Using first  $|\min\{L, x\} - \min\{L,y\}| \leq |x-y|$, the symmetry of $\rho$, and then the simple
inequality 
$|\max\{x,y\} - x| \leq y$ for $x,y \geq 0$, we have
\begin{eqnarray*}
|\mathcal{E}_1| 
& \leq & 2 \E_{0} \left[ \delta_{0}(S_n) |M_n - M_{\hat{n}}| \right] \\
& \leq & 2 \E_{0} \left[ \delta_{0}(S_n) \max\{ S_k: \hat{n} < k \leq n\}  \right] \\
& = & 2 \E_{0} \left[ \delta_{0}(S_n) \max \{ S_k: k \leq n - \hat{n}\}  \right] \\
& \leq & 2 \E_{0} \left[M_{n - \hat{n}}  \right] \|\rho^{\hat{n}*}\|_{\infty} 
\end{eqnarray*}
where we have used time reversal and symmetry of the increments in the equality above, and 
then conditioned at time $n - \hat{n}$ in the final step. 
By Doob's inequality $\E_{0} [M_n] \leq C n^{1/2}$;
by $(\ref{lclt})$ we have $|\rho^{*n}(x)| \leq C n^{-1/2}$;
using $n-\hat{n} = O(n^{\alpha})$ 
we conclude that $\mathcal{E}_1 = O(n^{(\alpha-1)/2})$. 
The bound $|\mathcal{E}_6 - \mathcal{E}_5|= O(n^{(\alpha-1)/2})$
is similar.

Next, also by similar steps, 
\begin{eqnarray*}
|\mathcal{E}_2-\mathcal{E}_1| 
& \leq & 2 \E_{0} \left[ \delta_{0}(S_n) M_{\hat{n}} \I(\Omega_n^c) \right] \\
& \leq & 2 \E_{0} \left[ M_{\hat{n}} \I(\Omega_n^c)    \right] \|\rho^{(n-\hat{n})*}\|_{\infty} \\
& \leq & 2 (\E_{0}[ M_{\hat{n}}^2])^{1/2} (\Pp[\Omega_n^c])^{1/2} \|\rho^{(n-\hat{n})*}\|_{\infty} = O(n^{1- \beta - \frac{\alpha}{2}}).
\end{eqnarray*}
The bound $|\mathcal{E}_5 - \mathcal{E}_4|= O(n^{1- \beta - \frac{\alpha}{2}})$
is similar.

Next recall that we have embedded $M_{\hat{n}} = \max_{k \leq \hat{n}} W(T_k)$. 
Conditioning on $\mathcal{F}^W_{T_{\hat{n}}}$ we see
\begin{eqnarray*}
|\mathcal{E}_3-\mathcal{E}_2|
& = & \left|\E_{0} \left[ \min\{L,M_{\hat{n}} - m_{\hat{n}}\} \I(\Omega_n)  \left(\delta_{0}(W(n)) - \delta_{0}(S_n) \right)
 \right] \right| \\
 & = & \left|\E_{0} \left[ \min\{L,M_{\hat{n}} - m_{\hat{n}}\} \I(\Omega_n)  
 \left(g_{n-T_{\hat{n}}}(S_{\hat{n}}) - \rho^{(n-\hat{n})*}(S_{\hat{n}}) \right)  \right] \right| \\
  & \leq & 2 \E_{0} \left[ M_{\hat{n}} \I(\Omega_n)  \|g_{n-T_{\hat{n}}} -  \rho^{(n-\hat{n})*}\|_{\infty}  \right] \\
    & \leq & 2 \E_{0} \left[ M_{\hat{n}} \I(\Omega_n) \left( \|g_{n-T_{\hat{n}}} -  g_{n-\hat{n}} \|_{\infty} 
    + \|g_{n-\hat{n}}  -  \rho^{(n-\hat{n})*}\|_{\infty}\right)  \right].
\end{eqnarray*}
On the set $\Omega_n$ we have $|\hat{n}-T_{\hat{n}}| \leq n^{\beta}$ and then 
 $\|g_{n-T_{\hat{n}}} -  g_{n-\hat{n}} \|_{\infty} \leq C n^{\beta - \frac32 \alpha}$. 
 Combined with (\ref{lclt}) we find $|\mathcal{E}_3-\mathcal{E}_2| = O(n^{\frac12 +\beta - \frac32 \alpha})$. 

Finally we use the modulus of continuity for a Brownian motion showing, for $\epsilon> 0$, there is a 
variable $H_{\epsilon}$ with finite moments, so that 
$|W(t) - W(s)| \leq H_{\epsilon} n^{\epsilon} |t-s|^{\frac12 - \epsilon}$ for all $0 \leq s,t \leq n$, almost surely.
The last errror is
\begin{eqnarray}
|\mathcal{E}_4-\mathcal{E}_3|
& \leq & 2 \E_{0} \left[ \delta_{0}(W(n)) | \max_{k \leq \hat{n}} W(T_k) - W^*(\hat{n})|; \Omega_n \right]  \nonumber \\
& \leq & C n^{-\frac{\alpha}{2}} \E_{0} \left[  | \max_{k \leq \hat{n}} W(T_k) - W^*(\hat{n})|; \Omega_n \right]  \nonumber \\
& \leq & C n^{-\frac{\alpha}{2}} \E_{0} \left[
| \max_{k \leq \hat{n}} W(T_k) - W^*(T_{\hat{n}})| + |W^*(T_{\hat{n}}) - W^*(\hat{n})|; \Omega_n \right].  \label{temp567}
\end{eqnarray}
On the set $\Omega_n$, using the modulus of continuity we have 
$|W^*(T_{\hat{n}}) - W^*(\hat{n})| \leq H_{\epsilon}  n^{(\frac12 - \epsilon)\beta + \epsilon}$.
Also
\begin{eqnarray*}
| \max_{k \leq \hat{n}} W(T_k) - W^*(T_{\hat{n}})| & \leq & H_{\epsilon} n^{\epsilon} \max_{k < \hat{n}}
|T_k - T_{k+1}|^{\frac12 - \epsilon} \\
& \leq & H_{\epsilon} n^{\epsilon} \left( 1 + 2 \max_{k \leq \hat{n}}
|T_k - k| \right)^{\frac12 - \epsilon}\\
& \leq & C H_{\epsilon}  n^{(\frac12 - \epsilon)\beta + \epsilon}
\end{eqnarray*}
(by the triangle inequality $|T_k - T_{k+1}| \leq |T_k - k| +1+ |(k+1)- T_{k+1}|$). Using these 
estimates in (\ref{temp567}) we reach $|\mathcal{E}_4-\mathcal{E}_3| = O(n^{-\frac{\alpha}{2} + (\frac12 - \epsilon)\beta + \epsilon})$.
Collecting all error estimates and choosing $\alpha = \frac56$ and $\beta = \frac23$ leads to an overall error
$|\mathcal{E}_6| = O(n^{-\frac{1}{12} + \epsilon})$ completing the proof of (\ref{E20}).

The changes needed for $|\tilde{E}^{(2)}(n,L)| $ in the non-translationally invariant case are again small.
We leave the chain of approximations (\ref{aaargh}) exactly as before, except that it starts with the 
expectation  $\E_{0} \left[ \delta_{0}(S_n) \min\{L,\tilde{M}_n - m_n\}  \right]$. 
This implies that we only need to re-estimate the error in the first step, which requires a bound on 
the new term 
\begin{equation} \label{last}
\E_{0} \left[ \delta_{0}(S_n) |\tilde{M}_n - M_{\hat{n}}| \}  \right] \leq 
\E_{0} \left[ \delta_{0}(S_n) |\tilde{M}_n - \tilde{M}_{\hat{n}}| \}  \right] + \E_{0} \left[ \delta_{0}(S_n) |\tilde{M}_{\hat{n}} - M_{\hat{n}}| \}  \right].
\end{equation}
The second term in (\ref{last}) is estimated as $O(n^{(\alpha-1)/2})$ using the local central limit theorem 
to bound the density of $S_n - S_{\hat{n}}$ as before.
For the first term
we use time reversal again
\begin{eqnarray*}
\E_{0} \left[ \delta_{0}(S_n) |\tilde{M}_n - \tilde{M}_{\hat{n}}| \}  \right] & \leq &
\E_{0} \left[ \delta_{0}(S_n) \max\{\tilde{S}_k: \hat{n} < k \leq n \}  \right] \\
& = & \E_{0} \left[ \delta_{0}(S_n) \max\{\tilde{S}_k: 1 \leq k \leq n- \hat{n} \}  \right] \\
& \leq & C n^{-1/2} \E_{0} \left[ \max\{\tilde{S}_k: 1 \leq k \leq n- \hat{n} \}  \right]  = O(n^{(\alpha-1)/2})
\end{eqnarray*}
where in the final inequality we have estimated the 
density $\tilde{\rho}^{\hat{n}*} * \rho$ of $(S_n - \tilde{S}_{n- \hat{n}})$
again by the local central limit theorem.
\begin{acks}[Acknowledgments]
We are grateful to Thomas Bothner for many useful discussions.
\end{acks}
\begin{funding}
The first author was supported by EPSRC as part of the MASDOC DTC, Grant.
No. EP/HO23364/1.
\end{funding}
\bibliographystyle{imsart-number} 
\bibliography{FPfinal_aop_revised}       

\begin{thebibliography}{41}

\bibitem{Akhiezer}
\begin{barticle}[author]
\bauthor{\bsnm{Akhiezer},~\bfnm{No~I}\binits{N.~I.}}
(\byear{1964}).
\btitle{The continuous analogues of some Theorems on Toeplitz matrices}.
\bjournal{Ukrainian Math. J}
\bvolume{16}
\bpages{445--462}.
\end{barticle}
\endbibitem

\bibitem{anderson2010}
\begin{bbook}[author]
\bauthor{\bsnm{Anderson},~\bfnm{Greg~W}\binits{G.~W.}},
  \bauthor{\bsnm{Guionnet},~\bfnm{Alice}\binits{A.}} \AND
  \bauthor{\bsnm{Zeitouni},~\bfnm{Ofer}\binits{O.}}
(\byear{2010}).
\btitle{An introduction to random matrices}.
\bseries{Cambridge Studies in Advanced Mathematics (118)}.
\bpublisher{Cambridge university press}.
\end{bbook}
\endbibitem

\bibitem{baik2020largest}
\begin{barticle}[author]
\bauthor{\bsnm{Baik},~\bfnm{Jinho}\binits{J.}} \AND
  \bauthor{\bsnm{Bothner},~\bfnm{Thomas}\binits{T.}}
(\byear{2020}).
\btitle{The largest real eigenvalue in the real {G}inibre ensemble and its
  relation to the {Z}akharov--{S}habat system}.
\bjournal{The Annals of Applied Probability}
\bvolume{30}
\bpages{460--501}.
\end{barticle}
\endbibitem

\bibitem{baik2020edge}
\begin{barticle}[author]
\bauthor{\bsnm{Baik},~\bfnm{Jinho}\binits{J.}} \AND
  \bauthor{\bsnm{Bothner},~\bfnm{Thomas}\binits{T.}}
(\byear{2020}).
\btitle{Edge distribution of thinned real eigenvalues in the real {G}inibre
  ensemble}.
\bjournal{arXiv preprint arXiv:2008.01694}.
\end{barticle}
\endbibitem

\bibitem{borodin2016erratum}
\begin{barticle}[author]
\bauthor{\bsnm{Borodin},~\bfnm{Alexei}\binits{A.}},
  \bauthor{\bsnm{Poplavskyi},~\bfnm{Mihail}\binits{M.}},
  \bauthor{\bsnm{Sinclair},~\bfnm{Christopher~D}\binits{C.~D.}},
  \bauthor{\bsnm{Tribe},~\bfnm{Roger}\binits{R.}} \AND
  \bauthor{\bsnm{Zaboronski},~\bfnm{Oleg}\binits{O.}}
(\byear{2016}).
\btitle{Erratum to: The {G}inibre ensemble of real random matrices and its
  scaling limits}.
\bjournal{Communications in Mathematical Physics}
\bvolume{346}
\bpages{1051--1055}.
\end{barticle}
\endbibitem

\bibitem{Borodin_Sinclair}
\begin{barticle}[author]
\bauthor{\bsnm{Borodin},~\bfnm{A.}\binits{A.}} \AND
  \bauthor{\bsnm{Sinclair},~\bfnm{C.~D.}\binits{C.~D.}}
(\byear{2009}).
\btitle{The {G}inibre Ensemble of Real Random Matrices and its Scaling Limits}.
\bjournal{Commun. Math. Phys.}
\bvolume{291}
\bpages{177–224}.
\end{barticle}
\endbibitem

\bibitem{Persistence_review}
\begin{barticle}[author]
\bauthor{\bsnm{Bray},~\bfnm{Alan~J.}\binits{A.~J.}},
  \bauthor{\bsnm{Majumdar},~\bfnm{Satya~N.}\binits{S.~N.}} \AND
  \bauthor{\bsnm{Schehr},~\bfnm{Grégory}\binits{G.}}
(\byear{2013}).
\btitle{Persistence and first-passage properties in nonequilibrium systems}.
\bjournal{Advances in Physics}
\bvolume{62}
\bpages{225-361}.
\bdoi{10.1080/00018732.2013.803819}
\end{barticle}
\endbibitem

\bibitem{daley_vere_jones}
\begin{bbook}[author]
\bauthor{\bsnm{Daley},~\bfnm{D.~J.}\binits{D.~J.}} \AND
  \bauthor{\bsnm{Vere-Jones},~\bfnm{D.}\binits{D.}}
(\byear{2003}).
\btitle{An Introduction to the Theory of Point Processes}.
\bpublisher{Springer}.
\end{bbook}
\endbibitem

\bibitem{Deift_Its_krasovsky}
\begin{barticle}[author]
\bauthor{\bsnm{Deift},~\bfnm{Percy}\binits{P.}},
  \bauthor{\bsnm{Its},~\bfnm{Alexander}\binits{A.}} \AND
  \bauthor{\bsnm{Krasovsky},~\bfnm{Igor}\binits{I.}}
(\byear{2011}).
\btitle{Asymptotics of {T}oeplitz, {H}ankel, and {T}oeplitz {H}ankel
  Determinants with {F}isher-{H}artwig Singularities.}
\bjournal{Annals of Mathematics}
\bvolume{174}
\bpages{1243-1299}.
\end{barticle}
\endbibitem

\bibitem{Dembo_Poonen_Shao_Zeitouni}
\begin{barticle}[author]
\bauthor{\bsnm{Dembo},~\bfnm{Amir}\binits{A.}},
  \bauthor{\bsnm{Poonen},~\bfnm{Bjorn}\binits{B.}},
  \bauthor{\bsnm{Shao},~\bfnm{Qi-Man}\binits{Q.-M.}} \AND
  \bauthor{\bsnm{Zeitouni},~\bfnm{Ofer}\binits{O.}}
(\byear{2002}).
\btitle{Random polynomials having few or no real zeros}.
\bjournal{J. Amer. Math. Soc.}
\bvolume{15}
\bpages{857 -- 892}.
\end{barticle}
\endbibitem

\bibitem{Derrida_Hakim_Pasquier}
\begin{barticle}[author]
\bauthor{\bsnm{Derrida},~\bfnm{Bernard}\binits{B.}},
  \bauthor{\bsnm{Hakim},~\bfnm{Vincent}\binits{V.}} \AND
  \bauthor{\bsnm{Pasquier},~\bfnm{Vincent}\binits{V.}}
(\byear{1996}).
\btitle{Exact exponent for the number of persistent spins in the
  zero-temperature dynamics of the one-dimensional {P}otts model}.
\bjournal{Journal of statistical physics}
\bvolume{85}
\bpages{763--797}.
\end{barticle}
\endbibitem

\bibitem{Derrida_Zeitak}
\begin{barticle}[author]
\bauthor{\bsnm{Derrida},~\bfnm{Bernard}\binits{B.}} \AND
  \bauthor{\bsnm{Zeitak},~\bfnm{Reuven}\binits{R.}}
(\byear{1996}).
\btitle{Distribution of domain sizes in the zero temperature {G}lauber dynamics
  of the one-dimensional {P}otts model}.
\bjournal{Physical Review E}
\bvolume{54}
\bpages{2513}.
\end{barticle}
\endbibitem

\bibitem{dornic2018universal}
\begin{barticle}[author]
\bauthor{\bsnm{Dornic},~\bfnm{Ivan}\binits{I.}}
(\byear{2018}).
\btitle{Universal {P}ainlev\'e {V}{I} Probability Distribution in {P}faffian
  Persistence and Gaussian First-Passage Problems with a $\sech$-Kernel}.
\bjournal{arXiv preprint arXiv:1810.06957}.
\end{barticle}
\endbibitem

\bibitem{durrett2019probability}
\begin{bbook}[author]
\bauthor{\bsnm{Durrett},~\bfnm{Rick}\binits{R.}}
(\byear{2019}).
\btitle{Probability: theory and examples}
\bvolume{49}.
\bpublisher{Cambridge university press}.
\end{bbook}
\endbibitem

\bibitem{feller}
\begin{bbook}[author]
\bauthor{\bsnm{Feller},~\bfnm{William}\binits{W.}}
(\byear{1966}).
\btitle{An introduction to probability theory and its application}
\bvolume{II}.
\bpublisher{John Wiley and Sons}.
\end{bbook}
\endbibitem

\bibitem{fitzgerald2020sharp}
\begin{barticle}[author]
\bauthor{\bsnm{FitzGerald},~\bfnm{Will}\binits{W.}},
  \bauthor{\bsnm{Tribe},~\bfnm{Roger}\binits{R.}} \AND
  \bauthor{\bsnm{Zaboronski},~\bfnm{Oleg}\binits{O.}}
(\byear{2020}).
\btitle{Sharp asymptotics for {F}redholm {P}faffians related to interacting
  particle systems and random matrices}.
\bjournal{Electronic Journal of Probability}
\bvolume{25}
\bpages{1--15}.
\end{barticle}
\endbibitem

\bibitem{FWR}
\begin{bphdthesis}[author]
\bauthor{\bsnm{FitzGerald},~\bfnm{William~Robert}\binits{W.~R.}}
(\byear{2019}).
\btitle{Exactly solvable interacting particle systems},
\btype{PhD thesis},
\bpublisher{University of Warwick},
\baddress{Coventry, United Kingdom}.
\end{bphdthesis}
\endbibitem

\bibitem{Forrester}
\begin{barticle}[author]
\bauthor{\bsnm{Forrester},~\bfnm{Peter~J}\binits{P.~J.}}
(\byear{2010}).
\btitle{The limiting {K}ac random polynomial and truncated random orthogonal
  matrices}.
\bjournal{Journal of Statistical Mechanics: Theory and Experiment}
\bvolume{2010}
\bpages{P12018}.
\end{barticle}
\endbibitem

\bibitem{Forrester_bulk_gap}
\begin{barticle}[author]
\bauthor{\bsnm{Forrester},~\bfnm{Peter~J}\binits{P.~J.}}
(\byear{2015}).
\btitle{Diffusion processes and the asymptotic bulk gap probability for the
  real {G}inibre ensemble}.
\bjournal{Journal of Physics A: Mathematical and Theoretical}
\bvolume{48}
\bpages{324001}.
\bdoi{10.1088/1751-8113/48/32/324001}
\end{barticle}
\endbibitem

\bibitem{Forrester_Nagao}
\begin{barticle}[author]
\bauthor{\bsnm{Forrester},~\bfnm{Peter~J.}\binits{P.~J.}} \AND
  \bauthor{\bsnm{Nagao},~\bfnm{Taro}\binits{T.}}
(\byear{2007}).
\btitle{Eigenvalue Statistics of the Real {G}inibre Ensemble}.
\bjournal{Phys. Rev. Lett.}
\bvolume{99}
\bpages{050603}.
\bdoi{10.1103/PhysRevLett.99.050603}
\end{barticle}
\endbibitem

\bibitem{GPTZ}
\begin{barticle}[author]
\bauthor{\bsnm{Garrod},~\bfnm{Barnaby}\binits{B.}},
  \bauthor{\bsnm{Poplavskyi},~\bfnm{Mihail}\binits{M.}},
  \bauthor{\bsnm{Tribe},~\bfnm{Roger~P}\binits{R.~P.}} \AND
  \bauthor{\bsnm{Zaboronski},~\bfnm{Oleg~V}\binits{O.~V.}}
(\byear{2018}).
\btitle{Examples of Interacting Particle Systems on $\mathbb{Z}$ as {P}faffian
  Point Processes: {A}nnihilating and {C}oalescing {R}andom {W}alks}.
\bjournal{Annales Henri Poincar{\'e}}
\bvolume{19}
\bpages{3635--3662}.
\end{barticle}
\endbibitem

\bibitem{GTZ}
\begin{barticle}[author]
\bauthor{\bsnm{Garrod},~\bfnm{Barnaby}\binits{B.}},
  \bauthor{\bsnm{Tribe},~\bfnm{Roger}\binits{R.}} \AND
  \bauthor{\bsnm{Zaboronski},~\bfnm{Oleg}\binits{O.}}
(\byear{2020}).
\btitle{Examples of Interacting Particle Systems on $\mathbb{Z}$ as {P}faffian
  Point Processes: {C}oalescing--{B}ranching {R}andom {W}alks and
  {A}nnihilating {R}andom {W}alks with {I}mmigration}.
\bjournal{Annales Henri Poincar{\'e}}
\bvolume{21}
\bpages{885--908}.
\end{barticle}
\endbibitem

\bibitem{GGK}
\begin{bbook}[author]
\bauthor{\bsnm{Gohberg},~\bfnm{Israel}\binits{I.}},
  \bauthor{\bsnm{Goldberg},~\bfnm{Seymour}\binits{S.}} \AND
  \bauthor{\bsnm{Krupnik},~\bfnm{Nahum}\binits{N.}}
(\byear{2012}).
\btitle{Traces and determinants of linear operators}
\bvolume{116}.
\bpublisher{Birkh{\"a}user}.
\end{bbook}
\endbibitem

\bibitem{Kac}
\begin{barticle}[author]
\bauthor{\bsnm{Kac},~\bfnm{Mark}\binits{M.}}
(\byear{1954}).
\btitle{Toeplitz matrices, translation kernels and a related problem in
  probability theory}.
\bjournal{Duke Mathematical Journal}
\bvolume{21}
\bpages{501--509}.
\end{barticle}
\endbibitem

\bibitem{krajenbrink2020painleve}
\begin{barticle}[author]
\bauthor{\bsnm{Krajenbrink},~\bfnm{Alexandre}\binits{A.}}
(\byear{2020}).
\btitle{From {P}ainlev{\'e} to {Z}akharov--{S}habat and beyond: {F}redholm
  determinants and integro-differential hierarchies}.
\bjournal{Journal of Physics A: Mathematical and Theoretical}
\bvolume{54}
\bpages{035001}.
\end{barticle}
\endbibitem

\bibitem{krajenbrink2021inverse}
\begin{barticle}[author]
\bauthor{\bsnm{Krajenbrink},~\bfnm{Alexandre}\binits{A.}} \AND
  \bauthor{\bsnm{Doussal},~\bfnm{Pierre~Le}\binits{P.~L.}}
(\byear{2021}).
\btitle{The inverse scattering of the {Z}akharov-{S}habat system solves the
  weak noise theory of the {K}ardar-{P}arisi-{Z}hang equation}.
\bjournal{arXiv preprint arXiv:2103.17215}.
\end{barticle}
\endbibitem

\bibitem{lawler}
\begin{bbook}[author]
\bauthor{\bsnm{Lawler},~\bfnm{Gregory~F}\binits{G.~F.}} \AND
  \bauthor{\bsnm{Limic},~\bfnm{Vlada}\binits{V.}}
(\byear{2010}).
\btitle{Random walk: a modern introduction}
\bvolume{123}.
\bpublisher{Cambridge University Press}.
\end{bbook}
\endbibitem

\bibitem{Lax}
\begin{bbook}[author]
\bauthor{\bsnm{Lax},~\bfnm{Peter~D.}\binits{P.~D.}}
(\byear{2002}).
\btitle{Functional Analysis}.
\bpublisher{John Wiley and Sons}.
\end{bbook}
\endbibitem

\bibitem{Matsumoto_Shirai}
\begin{barticle}[author]
\bauthor{\bsnm{Matsumoto},~\bfnm{Sho}\binits{S.}} \AND
  \bauthor{\bsnm{Shirai},~\bfnm{Tomoyuki}\binits{T.}}
(\byear{2013}).
\btitle{Correlation functions for zeros of a {G}aussian power series and
  {P}faffians}.
\bjournal{Electron. J. Probab}
\bvolume{18}
\bpages{1--18}.
\end{barticle}
\endbibitem

\bibitem{mehta2004random}
\begin{bbook}[author]
\bauthor{\bsnm{Mehta},~\bfnm{Madan~Lal}\binits{M.~L.}}
(\byear{2004}).
\btitle{Random matrices}.
\bpublisher{Elsevier}.
\end{bbook}
\endbibitem

\bibitem{Poplavskyi_Schehr}
\begin{barticle}[author]
\bauthor{\bsnm{Poplavskyi},~\bfnm{Mihail}\binits{M.}} \AND
  \bauthor{\bsnm{Schehr},~\bfnm{Gr\'egory}\binits{G.}}
(\byear{2018}).
\btitle{Exact Persistence Exponent for the $2D$-Diffusion Equation and Related
  {K}ac Polynomials}.
\bjournal{Phys. Rev. Lett.}
\bvolume{121}
\bpages{150601}.
\bdoi{10.1103/PhysRevLett.121.150601}
\end{barticle}
\endbibitem

\bibitem{PTZ}
\begin{barticle}[author]
\bauthor{\bsnm{Poplavskyi},~\bfnm{Mihail}\binits{M.}},
  \bauthor{\bsnm{Tribe},~\bfnm{Roger}\binits{R.}} \AND
  \bauthor{\bsnm{Zaboronski},~\bfnm{Oleg}\binits{O.}}
(\byear{2017}).
\btitle{On the distribution of the largest real eigenvalue for the real
  {G}inibre ensemble}.
\bjournal{Ann. Appl. Probab.}
\bvolume{27}
\bpages{1395--1413}.
\bdoi{10.1214/16-AAP1233}
\end{barticle}
\endbibitem

\bibitem{rains2000}
\begin{barticle}[author]
\bauthor{\bsnm{Rains},~\bfnm{Eric~M}\binits{E.~M.}}
(\byear{2000}).
\btitle{Correlation functions for symmetrized increasing subsequences}.
\bjournal{arXiv preprint math/0006097}.
\end{barticle}
\endbibitem

\bibitem{Rider_Sinclair}
\begin{barticle}[author]
\bauthor{\bsnm{Rider},~\bfnm{Brian}\binits{B.}} \AND
  \bauthor{\bsnm{Sinclair},~\bfnm{Christopher~D.}\binits{C.~D.}}
(\byear{2014}).
\btitle{Extremal laws for the real {G}inibre ensemble}.
\bjournal{Ann. Appl. Probab.}
\bvolume{24}
\bpages{1621--1651}.
\bdoi{10.1214/13-AAP958}
\end{barticle}
\endbibitem

\bibitem{soshnikov2000determinantal}
\begin{barticle}[author]
\bauthor{\bsnm{Soshnikov},~\bfnm{Alexander}\binits{A.}}
(\byear{2000}).
\btitle{Determinantal random point fields}.
\bjournal{Russian Mathematical Surveys}
\bvolume{55}
\bpages{923}.
\end{barticle}
\endbibitem

\bibitem{spouge}
\begin{barticle}[author]
\bauthor{\bsnm{Spouge},~\bfnm{John~L}\binits{J.~L.}}
(\byear{1988}).
\btitle{Exact solutions for a diffusion-reaction process in one dimension}.
\bjournal{Physical review letters}
\bvolume{60}
\bpages{871}.
\end{barticle}
\endbibitem

\bibitem{szego}
\begin{barticle}[author]
\bauthor{\bsnm{Szeg\H{o}},~\bfnm{Gabor}\binits{G.}}
(\byear{1952}).
\btitle{On certain Hermitian forms associated with the Fourier series of a
  positive function}.
\bjournal{Comm. S{\'e}m. Math. Univ. Lund [Medd. Lunds Univ. Mat. Sem.]}
\bvolume{1952}
\bpages{228--238}.
\end{barticle}
\endbibitem

\bibitem{Toth_Werner}
\begin{barticle}[author]
\bauthor{\bsnm{T{\'o}th},~\bfnm{B{\'a}lint}\binits{B.}} \AND
  \bauthor{\bsnm{Werner},~\bfnm{Wendelin}\binits{W.}}
(\byear{1998}).
\btitle{The true self-repelling motion}.
\bjournal{Probability Theory and Related Fields}
\bvolume{111}
\bpages{375--452}.
\end{barticle}
\endbibitem

\bibitem{Tracy_Widom}
\begin{barticle}[author]
\bauthor{\bsnm{Tracy},~\bfnm{Craig~A}\binits{C.~A.}} \AND
  \bauthor{\bsnm{Widom},~\bfnm{Harold}\binits{H.}}
(\byear{1996}).
\btitle{On orthogonal and symplectic matrix ensembles}.
\bjournal{Communications in Mathematical Physics}
\bvolume{177}
\bpages{727--754}.
\end{barticle}
\endbibitem

\bibitem{TZ}
\begin{barticle}[author]
\bauthor{\bsnm{Tribe},~\bfnm{Roger}\binits{R.}} \AND
  \bauthor{\bsnm{Zaboronski},~\bfnm{Oleg}\binits{O.}}
(\byear{2011}).
\btitle{Pfaffian formulae for one dimensional coalescing and annihilating
  systems}.
\bjournal{Electron. J. Probab}
\bvolume{16}
\bpages{2080--2103}.
\end{barticle}
\endbibitem

\bibitem{Wood}
\begin{btechreport}[author]
\bauthor{\bsnm{Wood},~\bfnm{David}\binits{D.}}
(\byear{1992}).
\btitle{The Computation of Polylogarithms}
\btype{Technical Report} No. \bnumber{15-92*},
\bpublisher{University of Kent, Computing Laboratory},
\baddress{University of Kent, Canterbury, UK}.
\end{btechreport}
\endbibitem

\end{thebibliography}
\end{document}